\documentclass[dvipsnames, reqno]{amsart}

\newif\iffinal
\finaltrue

\pdfoutput=1 % for arxiv.org

\usepackage[utf8]{inputenc}
\usepackage[T1]{fontenc}
\usepackage{lmodern}
\usepackage[english]{babel}
\usepackage{amssymb}
\usepackage{mathtools}
\usepackage{caption}
\usepackage{subcaption}
\usepackage[pdfencoding=auto, colorlinks=true, linkcolor=blue, citecolor=blue]{hyperref}
\usepackage{tikz}
\usetikzlibrary{calc, patterns, decorations.pathreplacing}
\usepackage{xcolor, colortbl}
\usepackage{url}

\usepackage{xifthen}
\usepackage{wrapfig}
\usepackage{fontawesome5}
\usepackage[bottom]{footmisc}
\usepackage{letltxmacro}
\usepackage{todonotes}
\usepackage{multirow}
\usepackage{booktabs}
\LetLtxMacro\todonotestodo\todo
\renewcommand{\todo}[2][]{\todonotestodo[backgroundcolor=yellow, #1]{TODO: {#2}}}
\iffinal\renewcommand{\todo}[2][]{}\fi

\DeclareMathOperator{\Ass}{Ass} 

\newcommand{\ic}[1]{\overline{#1}}
\newcommand{\ictwoideal}[1]{\ic{\twoideal{#1}}}

\newcommand{\longmid}{\,\,\middle\vert\,\,}

\newcommand{\mingens}[2][]{\mathsf{G}\ifthenelse{\isempty{#1}}{(#2)}{\!\left(#2\right)}}
\newcommand{\mingensast}[2][]{\mathsf{G}^{\ast}\!\ifthenelse{\isempty{#1}}{(#2)}{\!\left(#2\right)}}
\newcommand{\N}{\mathbb{N}}
\newcommand{\twoideal}[1]{(#1)}
\newcommand{\wdd}[1]{\mathsf{d}_{#1}}
\newcommand{\wdeg}[1]{\deg_{#1}}
\newcommand{\red}[1]{{#1}_{\anchor}}
\newcommand{\K}{\mathsf{k}}
\newcommand{\xy}{\bullet}
\newcommand{\white}[1]{\color{white}#1\color{black}}
\newcommand{\leftpart}{L}
\newcommand{\middlepart}{M}
\newcommand{\rightpart}{R}
\newcommand{\combpart}[1]{\mathsf{C}_{#1}}

\newcommand{\Apart}{\mathsf{A}}
\newcommand{\Hpart}{\mathsf{H}}
\newcommand{\Bpart}{\mathsf{B}}
\newcommand{\Cpart}{\mathsf{C}}

\newcommand{\anchor}{\mathbin{\scalebox{0.5}{\begin{tikzpicture}[scale=0.15]
                                \draw[very thick] (0,1.6) circle(0.3);
                                \draw[very thick] (-0.6,0.7) -- (0.6,0.7);
                                \draw[very thick] (0,1.4) -- (0,-1);
                                \draw[very thick] (0,-1) to [out =140, in =250] (-1,0.2);
                                \draw[very thick] (0,-1) to [out =40, in =-70] (1,0.2);
\end{tikzpicture}}}}

\newcommand{\ylink}{\mathbin{\raisebox{-0.85pt}{\scalebox{0.5}{\begin{tikzpicture}[scale=0.22]
                                \draw (0,0) circle(1.1);
                                \node at (0,0) {$\boldsymbol{y}$};
                                \end{tikzpicture}}}}}

\newcommand{\bigylink}{\mathop{\hbox{\raisebox{-5pt}{\begin{tikzpicture}
                                \node at (0,0) {\rule{0pt}{1.7ex}};
                                \draw[thick] (0,0) circle(0.24);
                                \node at (0,0) {\tiny $\boldsymbol{y}$};
                                \end{tikzpicture}}}}}

\newcommand{\xlink}{\mathbin{\raisebox{-0.85pt}{\scalebox{0.5}{\begin{tikzpicture}[scale=0.22]
                                \draw (0,0) circle(1.1);
                                \node at (0,0) {$\boldsymbol{x}$};
                                \end{tikzpicture}}}}}

\newcommand{\dotlink}{\mathbin{\raisebox{-0.85pt}{\scalebox{0.5}{\begin{tikzpicture}[scale=0.22]
                                \draw (0,0) circle(1.1);
                                \node at (0,0) {$\xy$};
                                \end{tikzpicture}}}}}

\renewcommand{\odot}{\ylink}
\renewcommand{\bigodot}{\bigylink}

\DeclareMathOperator{\dist}{dist}
\newcommand{\maxset}[2]{\dist_{#1}\!#2}
\newcommand{\maxsetdeg}[3][]{\dist_{#2}\ifthenelse{\isempty{#1}}{(#3)}{\!\left(#3\right)\!}}

\renewcommand\fbox{\fcolorbox{gray}{white}}

\newtheorem{proposition}{Proposition}[section]
\newtheorem{theorem}[proposition]{Theorem}
\newtheorem{corollary}[proposition]{Corollary}
\newtheorem{lemma}[proposition]{Lemma}
\theoremstyle{definition}
\newtheorem{definition}[proposition]{Definition}
\newtheorem{example}[proposition]{Example}

\newtheorem{notation}[proposition]{Notation}
\newtheorem{convention}[proposition]{Convention}
\newtheorem{fact}[proposition]{Fact}
\newtheorem{remark}[proposition]{Remark}
\newtheorem*{rem*}{Remark}
\numberwithin{equation}{section}

\makeatletter
\newcommand\footnoteref[1]{\protected@xdef\@thefnmark{\ref{#1}}\@footnotemark}
\makeatother

\makeatletter
\renewcommand{\fnum@figure}{Fig.~\thefigure}
\makeatother

\author{Jutta Rath}
\address{Institut für Mathematik\\Alpen-Adria-Universität Klagenfurt\\
  Universitätsstraße 65-67\\9020 Klagenfurt am Wörthersee\\Austria}
\email{\href{mailto:jutta.rath@aau.at}{jutta.rath@aau.at}}

\author{Roswitha Rissner}
\address{Institut für Mathematik\\Alpen-Adria-Universität Klagenfurt\\
  Universitätsstraße 65-67\\9020 Klagenfurt am Wörthersee\\Austria}
\email{\href{mailto:roswitha.rissner@aau.at}{roswitha.rissner@aau.at}}
\thanks{This research was funded in part by the Austrian
  Science Fund (FWF) [10.55776/DOC78]. For open access purposes, the
  authors have applied a CC BY public copyright license to any
  author-accepted manuscript version arising from this submission.}

\title[Large powers of bivariate monomial ideals]{Minimal generating
  sets of large powers of bivariate monomial ideals}


\newcolumntype{b}{>{\columncolor{gray!10}}r}
\newcolumntype{R}{>{\flushright}p{3.2em}}

\newcommand{\coloropacity}{0.2}
\newcommand{\whiteopacity}{0.6}
\newlength{\pointdim}
\newlength{\unitdim}
\newcommand{\offset}{0.2}
\newcommand*{\Jstring}{0/4,9/0}%

\newcommand{\fillColoredLatticePoints}[1]{
  \foreach \x in {0,...,{\xmax}} {
    \foreach \y in {0,...,{\ymax}} {
      \coordinate (point) at ($\scalefactor*(\x,\y)$);
      \pgfmathparse{(\x >= \xa+\xshift && \y >= \yshift|| (\x>=\xshift &&\y >= \yb+\yshift))}
      \ifnum\pgfmathresult=1
        \draw [blue, fill](point) circle (0.5pt);
      \else
        \pgfmathparse{#1>=1 && (\x -\xshift)/\xa + (\y-\yshift)/\yb >= 1 && \x>=\xshift &&\y >=\yshift} % if draw integral closure: color dots
        \ifnum\pgfmathresult=1
          \draw [fill, magenta](point) circle (0.5pt);
        \else
          \draw [fill, black](point) circle (0.3pt);
        \fi
      \fi
    }
  }
}

\newcommand{\drawonestep}[1]{
  \coordinate (shift) at ($\scalefactor*(\xshift, \yshift)$);
  \coordinate (g1) at ($\scalefactor*(0,\yb) + (shift)$);
  \coordinate (g2) at ($\scalefactor*(\xa,0) + (shift)$);

  \node [fill = blue, circle, inner sep=1pt] at (g1) {};
  \node [fill = blue, circle, inner sep=1pt] at (g2) {};

  \draw [blue, thick]
  let
  \p1 = (g1),
  \p2 = (g2),
  \p3 = (xend),
  \p4 = (yend),
  \p5 = (shift),
  \p6 = ($(\p3)+(0,\y5)$),
  \p7 = ($(\p4)+(\x5,0)$),
  in
  (\p6) -- (\p2) -- (\x2, \y1) -- (\p1) -- (\p7);

  \draw [fill=blue, opacity=0.2]
  let
  \p1 = (g1),
  \p2 = (g2),
  \p3 = (xend),
  \p4 = (yend),
  \p5 = (shift),
  \p6 = ($(\p3)+(0,\y5)$),
  \p7 = ($(\p4)+(\x5,0)$),
  in
  (\p6) -- (\p2) -- (\x2, \y1) -- (\p1) -- (\p7) -- (\x6,\y7);

  \pgfmathparse{#1 > 0} % if draw integral closure: color area
   \ifnum\pgfmathresult=1
    \draw [magenta, thick](g1) -- (g2);
    \draw [fill=magenta, opacity=0.2]
    let
    \p1 = (g1),
    \p2 = (g2),
    in
    (\p1) -- (\x2,\y1) -- (\p2);
  \fi

}

\newcommand{\getJdim}{
    \foreach \x/\y [count=\j] in \Jstring{
       \coordinate (scp) at ($\scalefactor*(\x,\y)$);
        \ifnum\j=1\relax
          \coordinate (last) at ($\scalefactor*(\x,\y)$);
          \coordinate (yJ) at (last);
        \fi
      \coordinate (last) at (scp);
    }
    \coordinate (xJ) at (last);

}

\newcommand{\computeR}{

  \coordinate (unit) at (0,1);
  \coordinate (A) at (yJ);

  \pgfextracty\pointdim{\pgfpointanchor{A}{center}}
  \pgfextracty\unitdim{\pgfpointanchor{unit}{center}}
  \pgfmathsetmacro{\ydimJ}{\pointdim/\unitdim}

  \pgfmathsetmacro{\found}{0}

  \pgfmathsetmacro{\r}{0}
  \foreach \rval in {1,...,\power}{
    \pgfmathparse{\rval*\scalefactor*\yb > \ydimJ}
    \ifnum\pgfmathresult=1
      \breakforeach
    \fi
    \coordinate (rval) at (0,\rval);
  }
  \pgfextracty\pointdim{\pgfpointanchor{rval}{center}}
  \pgfmathsetmacro{\r}{{\pointdim/\unitdim}}
}

\newcommand{\computeParameters}[1][0]{
  \coordinate (xJ) at (0,0);
  \coordinate (yJ) at (0,0);
  \pgfmathsetmacro{\r}{0}
  \pgfmathparse{#1 >=1} % if: calculate xy-dimensions of J
  \ifnum\pgfmathresult=1
    \getJdim
    \computeR
  \fi

  \pgfmathparse{(\power*(\xa + \xshift) + 1.5)}
  \pgfmathsetmacro{\xmax}{\pgfmathresult}
  \pgfmathparse{(\power*(\yb+ \yshift) + 1.5)}
  \pgfmathsetmacro{\ymax}{\pgfmathresult}

  \coordinate (xaxismax) at ($\scalefactor*(\xmax,0)+\scalefactor*(0.5,0)+(xJ)$);
  \coordinate (yaxismax) at ($\scalefactor*(0,\ymax)+\scalefactor*(0,0.5)+(yJ)$);

  \coordinate (xend) at ($\scalefactor*(\xmax,0)+(xJ)$);
  \coordinate (yend) at ($\scalefactor*(0,\ymax)+(yJ)$);
}

\newcommand{\drawAxes}{
  \coordinate (origin)   at (0,0);

  \draw [thin, gray,-latex] (origin) -- (xaxismax);
  \draw [thin, gray,-latex] (origin) -- (yaxismax);

}

\newcommand{\drawShift}[1]{
  \coordinate (shift) at ($\scalefactor*(\xshift, \yshift)$);
  \coordinate (g1) at ($\scalefactor*(0,\yb) + (shift)$);
  \coordinate (g2) at ($\scalefactor*(\xa,0) + (shift)$);
  \coordinate (g) at ($\scalefactor*(\xshift, \yshift)$);
  \pgfmathparse{\xshift>0 || \yshift>0}
  \ifnum\pgfmathresult=1
    \node [fill, black, circle, inner sep=1pt, label={[xshift=-3.5pt, yshift=-2pt,black, scale=0.9]45:#1}] at (g) {};
    \draw [dotted] (g1) -- (g) -- (g2);
  \fi
}

\newcommand{\drawStairs}[2][0]{
  \coordinate (shift) at ($\scalefactor*(\xshift, \yshift)$);
  \coordinate (g1) at ($\scalefactor*(0,\yb) + (shift)$);
  \coordinate (g2) at ($\scalefactor*(\xa,0) + (shift)$);

  \pgfmathparse{ #2==1 || #2==2} % if draw upper end
  \ifnum\pgfmathresult=1
    \pgfmathparse{ #1==1 } % optional argument: fill ideal area
    \ifnum\pgfmathresult=1
      \fill [gray, opacity=0.1]
       let
       \p1=($\power*(g1)$),
       \p2 = ($(\x1,0)+(yend)$),
       \p3 = (xend),
        in
        (\p2) -- (\p1) -- (\x3,\y1) -- (\x3,\y2);
    \fi
    \draw [black, thick]
     let
      \p1=($\power*(g1)$),
      in
      ($(\x1,0)+(yend)$) -- (\p1);
  \fi

  \pgfmathparse{ #2==1 || #2==3} % if draw lower end
  \ifnum\pgfmathresult=1
    \draw[black, thick]
     let
      \p1=($\power*(g2)$),
     in
     ($(0,\y1)+(xend)$) -- (\p1);
  \fi
  \foreach \i in {0,...,\power} {
    \draw [black, fill]($\i*(g1) + \power*(g2) - \i*(g2)$) circle (0.9pt);
    \pgfmathparse{\i<\power}
    \ifnum\pgfmathresult=1
      \pgfmathparse{ #1==1} % optional argument: fill ideal area
      \ifnum\pgfmathresult=1
        \fill [gray, opacity=0.1]
         let
         \p1 = ($\i*(g1) + \power*(g2) - \i*(g2)$),
         \p2 = ($(\p1) +(g1) - (g2)$),
         \p3 = (xend),
          in
          (\p1) -- (\x1,\y2) -- (\x3,\y2) -- (\x3,\y1);
      \fi
      \draw [black, thick]
      let
      \p1 = ($\i*(g1) + \power*(g2) - \i*(g2)$),
      \p2 = ($(\p1) +(g1) - (g2)$),
      in
      (\p1) --(\x1,\y2) --(\p2);
    \fi
  }

}

\newcommand{\drawJshape}[1]{
  \coordinate (shift) at ($\scalefactor*(\xshift, \yshift)$);
  \coordinate (g1) at ($\scalefactor*(0,\yb) + (shift)$);
  \coordinate (g2) at ($\scalefactor*(\xa,0) + (shift)$);

  \getJdim
  \coordinate (J) at ($(xJ) +(yJ)$);

  \def\mycolours{{"Mulberry","OliveGreen","RoyalBlue","BurntOrange"}}

  \foreach \i [evaluate=\i as \usecolor using {\mycolours[Mod(\i,4)]}]in {0,...,\power} {
      \draw [\usecolor, fill=\usecolor, draw, fill opacity=\coloropacity]
      let
      \p1 = ($\i*(g1) + \power*(g2) - \i*(g2)$),
      \p2 = ($(\p1)+(J)$),
      in
      (\p1) -- (\x2,\y1) -- (\p2) -- (\x1,\y2) -- (\p1);
      \pgfmathparse{\i>=\r+1 && #1 >=1} % if show x cut offs: white covers
      \ifnum\pgfmathresult=1
        \fill [white, opacity=\whiteopacity]
        let
        \p1 = ($\i*(g1) + \power*(g2) - \i*(g2)$),
        \p2 = ($(\p1) - \r*(g1)+ \r*(g2) $),
        \p3 = (J),
        \p4 = ($(\p1)+(\p3)$),
        \p5 = ($(\p2)+(0,\y3)$)
        in
       (\p5)--(\x4,\y5) -- (\p4) --(\x5,\y4);
      \fi
  }
  \pgfmathparse{#1 >=1}
   \ifnum\pgfmathresult=1
    \draw [white,fill, opacity=\whiteopacity]
        let
        \p1 = ($\power*(g2)$),
        \p2 = (J),
        \p3 = ($(\p1)+(\p2)$),
        \p4 = ($(\p1)+\r*(0,\y2)$),
        in
         (\p4) -- (\x4,\y3) -- (\p3);
   \fi
  \foreach \i [evaluate=\i as \usecolor using {\mycolours[Mod(\i,4)]}]in {0,...,\power} {
      \pgfmathparse{\i<=\power-\r && #1 >= 1} % if show x cut offs: dashed lines
      \ifnum\pgfmathresult=1
        \draw [\usecolor, dotted, thick]
        let
        \p1 = ($\i*(g1) + \power*(g2) - \i*(g2)$),
        \p2 = (J),
        \p3 = ($(\p1)+(0,\y2)$),
        \p4 = ($(\p3)+\r*\scalefactor*(0,\yb)$),
         in
         (\p3)  --  (\p4);

        \draw[fill, \usecolor]
        let
        \p1 = ($\i*(g1) + \power*(g2) - \i*(g2)$),
        \p2 = (J),
        \p3 = ($(\p1)+(0,\y2)$),
         in
         (\p3) circle (1pt);
      \fi
  }

}

\newcommand{\drawJalone}[2]{
    \coordinate (start) at (#1);

  \foreach \x/\y [count=\j] in \Jstring{
    \coordinate (scp) at ($(start)+\scalefactor*(\x,\y)$);
    \ifnum\j=1\relax
      \coordinate (last) at ($(start) + \scalefactor*(\x,\y)$);
      \draw[#2] (start) -- (last);
    \fi
    \ifnum\j>1\relax
      \fill[#2, opacity=\coloropacity]
      let
      \p1 = (last),
      \p2 = (scp)
      in
      (start)  --(\p1) -- (\x2,\y1) -- (\p2);
      \draw[#2]
      let
      \p1 = (last),
      \p2 = (scp)
      in
      (\p1) -- (\x2,\y1) -- (\p2);
    \fi
    \draw[black, fill]  (last) circle (0.5pt);
    \draw[black, fill]  (scp) circle (0.5pt);
    \coordinate (last) at (scp);

  }
  \draw[#2] (start) -- (last);

}

\newcommand{\whiten}[2]{
  \coordinate (start) at (#1);
   \foreach \x/\y [count=\j] in \Jstring{
        \coordinate (scp) at ($(start)+\scalefactor*(\x,\y)$);
        \ifnum\j=1\relax
          \coordinate (last) at ($(start) + \scalefactor*(\x,\y)$);
          \draw [white, fill, opacity=#2]
           let
           \p1 = (yend),
           \p2 = (xend),
           \p3 = (scp),
           \p4 = ($(\p3) + (0.07,0)$),
           \p5 = ($(\p3) + (0,0.07)$),
           \p6 = ($(\x3,\y1)$),
           in
           (\p4) -- (\p5)--(\x3,\y6)-- (\x2,\y6) --(\x2,\y3);
        \else
          \fill [white, opacity=#2]
           let
           \p1 = (last),
           \p2 = (xend),
           \p3 = (scp),
           \p4 = ($(\p3) + (0.07,0)$),
           \p5 = ($(\p3) + (0,0.07)$),
           in
           (\p4) -- (\p5)--(\x3,\y1)-- (\x2,\y1) --(\x2,\y3);
        \fi
      \coordinate (last) at (scp);
  }

}
\newcommand{\drawFullJ}[2][0]{
  \coordinate (shift) at ($\scalefactor*(\xshift, \yshift)$);
  \coordinate (g1) at ($\scalefactor*(0,\yb) + (shift)$);
  \coordinate (g2) at ($\scalefactor*(\xa,0) + (shift)$);

  \def\mycolors{{"Mulberry","OliveGreen","RoyalBlue","BurntOrange"}}

  \foreach \i [evaluate=\i as \usecolor using {\mycolors[Mod(\i,4)]}]in {0,...,\power} {
    \coordinate (start) at ($\i*(g1) + \power*(g2) - \i*(g2)$);
    \pgfmathparse{#1==2} % if:  choose gray instead of colors
    \ifnum\pgfmathresult=1
      \def\usecolor{gray}
    \fi
    \drawJalone{start}{\usecolor}
  }
  \pgfmathparse{#2>0} % if show x cut offs: white covers
  \ifnum\pgfmathresult=1
    \foreach \i [evaluate=\i as \usecolor using {\mycolors[Mod(\i,4)]}]in {0,...,\power} {
      \coordinate (start) at ($\i*(g1) + \power*(g2) - \i*(g2)$);
      \whiten{start}{#2}
    }
  \fi

  \pgfmathparse{#1 ==1} % if 1: display J in upper corner
  \ifnum\pgfmathresult=1
    \getJdim
    \coordinate (J) at ($(xJ) +(yJ)$);
    \coordinate (start) at ($\scalefactor*(\power*\xa + 1 -\positionx, \power*\yb +1-\positiony)$);
    \drawJalone{start}{gray}
    \node at ($(start) + (J) + (\offset,\offset)$){};
    \draw[gray] ($(start)-(\offset,\offset)$) rectangle ($(start) + (J) + (\offset,\offset)$);
    \node[gray] at ($(start) + (J) -\scalefactor*4*(\offset,\offset)$) {$J$};
  \fi

}

\newcommand{\drawYSections}[1][1]{
  \pgfkeys{/pgf/number format/relative*={2}}
  \coordinate (shift) at ($\scalefactor*(\xshift, \yshift)$);
  \coordinate (g1)    at ($\scalefactor*(0,\yb) + (shift)$);
  \coordinate (g2)    at ($\scalefactor*(\xa,0) + (shift)$);

  \pgfmathparse{#1 >=1} % if show x cut offs: white covers
  \ifnum\pgfmathresult=1
    \foreach \i in {\r,...,\power} {
      \draw [gray, thick, dashed]
        let
        \p1 = ($\i*(g1) + \power*(g2) - \i*(g2)$),
        \p2 = (xend),
        in
        (\x2,\y1) -- (-0.15,\y1) node [left, xshift=-2pt, scale=0.7] {$j=\pgfmathprintnumber{\i}$} ;
      }
   \else
    \foreach \i in {\r,...,\power} {
      \draw [gray, thick, dashed]
        let
        \p1 = ($\i*(g1) + \power*(g2) - \i*(g2)$),
        \p2 = (xend),
        in
        (\x2,\y1) -- (-0.15,\y1) node [left, xshift=-2pt, scale=0.7] {} ;
      }
    \fi

}

\begin{document}

\maketitle

\begin{abstract}
  It is known that for a monomial ideal $I$, the number of minimal
  generators, $\mu(I^n)$, eventually follows a polynomial pattern for
  increasing $n$.  In general, little is known about the power at
  which this pattern emerges. Even less is known about the exact form
  of the minimal generators after this power.  Let
  $s\ge \mu(I)(d^2-1)+1$, where $d$ is a constant bounded above by the
  maximal $x$- or $y$-degree appearing in the set $\mingens{I}$ of
  minimal generators of $I$.  We show that every higher power
  $I^{s+\ell}$ for any $\ell \ge 0$ can be constructed from certain
  subideals of $I^s$.  This provides an explicit description
  of~$\mathsf{G}(I^{s+\ell})$ in terms of $\mathsf{G}(I^s)$.  Given
  $\mingens{I^s}$, this construction significantly reduces
  computational complexity in determining larger powers of~$I$.  This
  further enables us to explicitly compute $\mu(I^n)$ for all $n\ge s$
  in terms of a linear polynomial in $n$.  We include runtime
  measurements for the attached implementation in SageMath.
\end{abstract}

\section{Introduction}
Powers of monomial ideals have been studied in many different
contexts. Brodmann~\cite{Brodmann:1979:analytic-spread,
  Brodmann:1979:asympstab} showed that the set $\Ass(R/I^n)$ of
associated primes stabilizes for large~$n$ and that
$\operatorname{depth}(R/I^n)$ eventually becomes constant.  Since
then, much research has extended these results.  For a recent survey,
we refer to Carlini, H\`a, Harbourne, and Van Tuyl's lecture
notes~\cite{Carlini-Ha-Harbourne-Tuyl:powers-of-ideals}.

When studying powers of a monomial ideal, its minimal generators play
a crucial role.  However, it is far from trivial to determine which of
the $n$-fold products of (minimal) generators of $I$ are minimal
generators of $I^n$ and very little is known for monomial ideals, even
in the bivariate case.

The existing research mostly focuses on the number $\mu(I^n)$ of
minimal generators of $I^n$ rather than the actual set of minimal
generators.  Indeed, the emphasis lies on small powers, as $\mu(I^n)$
is eventually described by a polynomial---the Hilbert polynomial of
the fiber ring of $I$.  Eliahou, Herzog, and
Saem~\cite{Eliahou-Herzog-Saem:2018:tiny-squares} studied the question
how small $\mu(I^2)$ can be in terms of $\mu(I)$ for a bivariate
monomial ideal $I$.  For any given $n\in\N$, Abdolmaleki and
Kumashiro~\cite{Abdolmaleki-Kumashiro:2021:gen-descend} construct a
bivariate monomial ideal~$I$ such that
$\mu(I)>\mu(I^2)>\cdots>\mu(I^n)$. Gasanova~\cite{Gasanova:2020:tiny-powers}
showed that for every $d$ there exists a monomial ideal $I$ in any
number of variables such that the inequality $\mu(I)>\mu(I^n)$ holds
for any $n\le d$.

While the unexpected behavior of small powers is fascinating, our
focus is set on large $n$.  The asymptotic behavior of the Hilbert
function gives reason to suspect that eventually the actual set of
minimal generators of $I^n$ behaves well, in the sense that
cancellations among the $n$-fold products of generators of $I$ can be
predicted.

The main objective of this paper is to describe the sets of minimal
generators of large powers of monomial ideals $I$ in
$\mathsf{k}[x,y]$, where $\mathsf{k}$ is a field.

\begin{wrapfigure}{r}{0.45\textwidth}
  \centering \begin{minipage}[t]{0.22\textwidth}
   \centering
  \begin{subfigure}{1\textwidth}
    \centering
\fbox{
    \begin{tikzpicture}
      \pgfmathsetmacro{\scalefactor}{0.23}
      \pgfmathsetmacro{\xa}{5}
      \pgfmathsetmacro{\yb}{4}
      \pgfmathsetmacro{\power}{1}
      \pgfmathsetmacro{\xshift}{0}
      \pgfmathsetmacro{\yshift}{0}

      \computeParameters
      \drawAxes

      \pgfmathsetmacro{\xa}{1}
      \pgfmathsetmacro{\yb}{2}
      \pgfmathsetmacro{\xshift}{0}
      \pgfmathsetmacro{\yshift}{1}
      \drawStairs[1]{2}

      \coordinate (left) at (g1);
      \pgfmathsetmacro{\xa}{3}
      \pgfmathsetmacro{\yb}{1}
      \pgfmathsetmacro{\xshift}{1}
      \pgfmathsetmacro{\yshift}{0}
      \drawStairs[1]{3}

      \fill[pattern color=PineGreen,  pattern=north west lines]
      let
      \p1 = ($(left)+(0.03, 0.03)$),
      \p2 = ($(g1)+(0.03, 0.03)$),
      \p3 = ($(g2)+(0.03, 0.03)$),
      \p4 = (xend),
      \p5 = (yend),
      in
      (\x1,\y5) -- (\p1) -- (\x2,\y1) -- (\p2) --(\x3,\y2) -- (\p3) -- (\x4, \y3) -- (\p4) -- (\x4,\y5);

      \draw[DarkOrchid, very thick, shorten >=1pt] (origin) -- (g2);
      \draw[-latex, shorten >=3pt] ($(g2) + (0.3,0.3)$)--(g2);
    \end{tikzpicture}
 }
    \caption{$I = (\textcolor{DarkOrchid}{x^4}, xy, y^3)$}
  \end{subfigure}

\vspace*{1pt}
  \begin{subfigure}{1\textwidth}
  \centering
\fbox{
    \begin{tikzpicture}
      \pgfmathsetmacro{\scalefactor}{0.23}
      \pgfmathsetmacro{\xa}{5}
      \pgfmathsetmacro{\yb}{3}
      \pgfmathsetmacro{\xshift}{0}
      \pgfmathsetmacro{\yshift}{0}
      \pgfmathsetmacro{\power}{1}

      \computeParameters
      \drawAxes

      \pgfmathsetmacro{\xa}{2}
      \pgfmathsetmacro{\yb}{2}
      \pgfmathsetmacro{\xshift}{0}
      \pgfmathsetmacro{\yshift}{0}
      \drawStairs[1]{1}

      \draw[PineGreen, very thick, shorten >=1pt] (origin) -- (g1);

      \draw[-latex, shorten >=3pt] ($(g1) + (0.3,0.3)$)--(g1);

      \fill[DarkOrchid, opacity=0.3]
      let
      \p2 = (g1),
      \p3 = (g2),
      \p4 = (xend),
      \p5 = (yend),
      in
      (\x2,\y5) -- (\p2) --(\x3,\y2) -- (\p3) -- (\x4, \y3) -- (\p4) -- (\x4,\y5);

    \end{tikzpicture}
}
  \caption{$J =(x^2,  \textcolor{PineGreen}{y^2})$}
  \end{subfigure}
\end{minipage}
\begin{minipage}[h]{0.22\textwidth}
  \centering
  \begin{subfigure}{1\textwidth}
    \centering
    \fbox{
    \begin{tikzpicture}
      \pgfmathsetmacro{\scalefactor}{0.23}
      \pgfmathsetmacro{\xa}{5}
      \pgfmathsetmacro{\yb}{6}
      \pgfmathsetmacro{\xshift}{0}
      \pgfmathsetmacro{\yshift}{0}
      \pgfmathsetmacro{\power}{1}

      \computeParameters
      \drawAxes

      \pgfmathsetmacro{\xa}{1}
      \pgfmathsetmacro{\yb}{2}
      \pgfmathsetmacro{\xshift}{0}
      \pgfmathsetmacro{\yshift}{3}
      \drawStairs[1]{2}

      \coordinate (left) at (g1);

      \pgfmathsetmacro{\xa}{3}
      \pgfmathsetmacro{\yb}{1}
      \pgfmathsetmacro{\xshift}{1}
      \pgfmathsetmacro{\yshift}{2}
      \drawStairs[1]{0}

      \def\col{PineGreen}
      \draw[\col, very thick]
      let
      \p1 = (g2)
      in
      (origin) -- (0,\y1);
      \draw[white, shorten <=1pt, shorten >=1pt]
      let
      \p1 = (g2)
      in
      (left) -- (0,\y1)--(g2);
      \draw[\col, dotted, very thick, shorten <=1pt, shorten >=1pt]
      let
      \p1 = (g2),
      \p2 = (xend),
      \p3 = (left),
      in
      (\p3) -- (0,\y1)--(\p1);
      \fill[pattern color=PineGreen,  pattern=north west lines]
      let
      \p1 = ($(left)+(0.03, 0.03)$),
      \p2 = ($(g1)+(0.03, 0.03)$),
      \p3 = ($(g2)+(0.03, 0.03)$),
      \p4 = (xend),
      \p5 = (yend),
      in
      (\x1,\y5) -- (\p1) -- (\x2,\y1) -- (\p2) --(\x3,\y2) -- (\p3) -- (\x4, \y3) -- (\p4) -- (\x4,\y5);

      \draw[\col, fill]
      let
      \p1 = (g2)
      in
      (0,\y1) circle (1pt);

      \pgfmathsetmacro{\xa}{2}
      \pgfmathsetmacro{\yb}{2}
      \pgfmathsetmacro{\xshift}{4}
      \pgfmathsetmacro{\yshift}{0}
      \drawStairs[1]{3}

      \def\col{DarkOrchid}
      \draw[\col, very thick]
      let
      \p1=(g1),
      in
      (origin) -- (\x1,0);
      \draw[white,  shorten <=1pt, shorten >=1pt]
      let
      \p1 = (g1)
      in
      (\p1) -- (\x1,0)--(g2);
      \draw[\col, dotted, very thick, shorten <=1pt, shorten >=1pt]
      let
      \p1 = (g1),
      \p2 = (yend),
      in
      (\p1) -- (\x1,0)--(g2);
      \draw[\col, fill]
      let
      \p1 = (g1)
      in
      (\x1,0) circle (1pt);
      \fill[DarkOrchid, opacity=0.3]
      let
      \p2 = (g1),
      \p3 = (g2),
      \p4 = (xend),
      \p5 = (yend),
      in
      (\x2,\y5) -- (\p2) --(\x3,\y2) -- (\p3) -- (\x4, \y3) -- (\p4) -- (\x4,\y5);

      \draw[-latex,shorten >=3pt] ($(g1) - (0.3,0.3)$)--(g1);

    \end{tikzpicture}
    }
    \caption{$I \odot J = \textcolor{PineGreen}{y^2}I + \textcolor{DarkOrchid}{x^4}J$}
  \end{subfigure}
\end{minipage}
%%% Local Variables:
%%% mode: latex
%%% TeX-master: "generators-powers-bivariate"
%%% End:
  \caption{Visualization of the ideal link}
  \label{figure:link}
\end{wrapfigure}
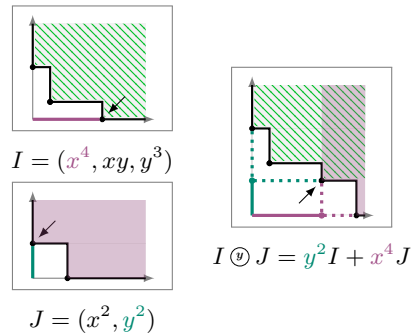

Specifically, we show that there exists $s_0$ such that for all
$n\ge s\ge s_0$ every segment of the staircase diagram---and
consequently, the set of minimal generators---of $I^n$ is already
determined by the staircase diagram of $I^{s}$.  In other words, for
$n\ge s$, the staircase diagram of $I^n$ can be built by aligning the
staircase diagrams of certain subideals of~$I^{s}$.  We prove that
$s_0\le\mu(I)(d^2-1)+1$, where $d$ is a constant depending on the
degrees of the minimal generators of $I$ and is at most the maximal
$x$- or $y$-degree appearing in $\mingens{I}$.

To formalize the process of aligning staircases of ideals, we
introduce the concept of the \emph{link of ideals} with respect to
$y$, denoted by $\odot$ (Definition~\ref{definition:link}).  The
staircase of the ideal $I\odot J$ is obtained by shifting $I$ in
$y$-direction and $J$ in $x$-direction such that their staircases meet
in precisely one minimal generator.  This is illustrated in
Figure~\ref{figure:link}.

In Theorem~\ref{theorem:I-with-k-persistent-stabilizes} we show that
for all $s$ larger than an explicitly given power $s_0$, all powers
$I^{s+\ell}$ for any $\ell\ge0$ can be written as the link of specific
subideals of $I^s$.

It immediately follows that
$\mu(I^{s+\ell})=\mu(I^s)+\ell\cdot\left(\mu(I^{s+1})-\mu(I^{s})\right)$
for $s\ge s_0$ and $\ell\ge 0$, see Corollary~\ref{corollary:mu}.

Moreover, the minimal generators of $I\odot J$ can be determined by
$\mu(I) + \mu(J) - 1$ many additions of (monomial) exponents. Thus,
$\mingens[l]{I^{s+\ell}}$ can be computed from $\mingens[l]{I^{s}}$ in
$\mathcal{O}(\ell)$ additions of non-negative integers
(Corollary~\ref{corollary:runtime}).

This paper is structured as follows:
Section~\ref{section:preliminaries} summarizes the necessary
background about the integral closure of bivariate monomial ideals.
In Section~\ref{section:persistent-generators}, we establish that the
lattice points on the boundary of the Newton polyhedron of $I$ play a
special role among the minimal generators of powers of $I$.  We call
them (weakly) persistent generators
(Definitions~\ref{def:persistent-generators}
and~\ref{def:weakly-persistent}), as their powers remain minimal
generators of all powers of $I$.  The main result of this section is
Theorem~\ref{theorem:I^D+l}.  This theorem is a refinement of the
known fact that the ideal generated by the persistent generators is a
reduction of
$I$~\cite[Proposition~2.1]{singla:2006:monomial-reduction}. It allows
us to break down large enough powers of $I$ into sums of simpler
components, using its weakly persistent generators.  In
Section~\ref{section:staircase-factors} we describe the minimal
generators of each such component separately
(Theorem~\ref{theorem:generators-one-segment}) and then determine
their sum in Theorem~\ref{theorem:gluing-more-segments}.
Section~\ref{section:min-gens} combines the results of
Section~\ref{section:persistent-generators} and
Section~\ref{section:staircase-factors} to the main results of this
paper (Theorem~\ref{theorem:I-with-k-persistent-stabilizes} and
Corollary~\ref{corollary:mu}).  We provide an implementation to
compute $I^{s+\ell}$ in SageMath\footnote{The program code associated
  with this paper is available as ancillary file from the arXiv page
  \url{https://arxiv.org/abs/2503.21466} of this paper.}. We conclude
the paper with examples and runtime measurements in practice in
Section~\ref{sec:runtime}.

We suspect that powers of ideals in more than two variables eventually
show analogous patterns. The immediate challenge is the generalization
of the link operation, as it strongly depends on the bivariate
structure.

We point out that the presented approach is not suitable to determine
minimal generators of small ideal powers.

\section{Preliminaries}\label{section:preliminaries}

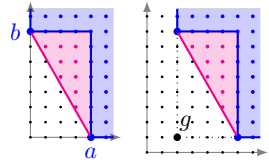
\begin{wrapfigure}{r}{0.5\textwidth} \centering
\captionsetup{width=.9\linewidth}
  \begin{minipage}[t]{.15\textwidth}
  \begin{tikzpicture}[baseline=(current bounding box.north)]
    \pgfmathsetmacro{\scalefactor}{0.2}
    \pgfmathsetmacro{\xa}{4}
    \pgfmathsetmacro{\yb}{7}
    \pgfmathsetmacro{\xshift}{0}
    \pgfmathsetmacro{\yshift}{0}
    \pgfmathsetmacro{\power}{1}

    \computeParameters
    \drawAxes
    \drawonestep{1}
    \fillColoredLatticePoints{1}
    
    \node[blue, scale=0.9] at (\xa*\scalefactor,-0.2) {$a$};
    \node[blue, scale=0.9] at (-0.2, \yb*\scalefactor) {$b$};
  \end{tikzpicture}
\end{minipage}%\hspace*{0.5em}
\begin{minipage}[t]{.15\textwidth}
  \begin{tikzpicture}[baseline=(current bounding box.north)]
    \pgfmathsetmacro{\scalefactor}{0.2}
    \pgfmathsetmacro{\xa}{4}
    \pgfmathsetmacro{\yb}{7}
    \pgfmathsetmacro{\xshift}{2}
    \pgfmathsetmacro{\yshift}{1}
    \pgfmathsetmacro{\power}{1}
    
    \computeParameters
    \drawAxes
    \drawonestep{1}
    \fillColoredLatticePoints{1}
    \drawShift{$g$}
  \end{tikzpicture}
\end{minipage}
%%% Local Variables:
%%% mode: latex
%%% TeX-master: "generators-powers-bivariate"
%%% End:
 \caption{Left: The integral closure of $(x^a, y^b)$ adds
    all lattice points above the line connecting $(a,0)$ and
   $(0,b)$ (red). Right: Visualization of Fact~\ref{fact:ic}(2).}
 \label{figure:integral-closure}
\end{wrapfigure}
Throughout this paper we frequently identify a monomial $x^ay^b$ with
the lattice point $(a,b)\in\mathbb{R}^2$.  Note that in the bivariate
case, the minimal generators of an ideal are always of the form
$x^{a_i}y^{b_i}$ with
\vspace{-2ex}

\begin{minipage}{0.45\textwidth}
  \begin{align*}
    a_0&<a_1<\dots<a_n,\text{ and }\\
    b_0&>b_1>\dots>b_n.
  \end{align*}
\end{minipage}

We always assume that $I$ is not a principal ideal.  Moreover, many
properties of monomial ideals, such as the number of minimal
generators, remain invariant under shifts.  Therefore, when it
simplifies notation, we divide $I$ by the greatest common divisor of
all its monomials, denoted by $\gcd(I)$.

\begin{definition}
  For a monomial ideal $I$ , we denote $\red{I}\coloneqq I:\gcd(I)$.
  We call a monomial ideal $I$ \emph{anchored}, if $\red{I}=I$.
\end{definition}

\begin{definition}
  For a monomial ideal $I$, the \emph{integral closure} $\ic{I}$ is
  defined as the set of all monomials $f$ for which there exists
  $n\in\N$ with $f^n\in I^n$.
\end{definition}

\begin{notation}
  We denote by $(E)$ the ideal generated by the set
  $E\subseteq \K[x,y]$.
\end{notation}

\begin{fact}[{cf.~\cite[Proposition~1.4.6]{Huneke-Swanson:2006:integral-closure}}]\label{fact:ic}
  For $a$, $b\in \N_0$ and a monomial $g\in \K[x,y]$ the following are
  equivalent:
  \begin{enumerate}
  \item $x^uy^v\in \ictwoideal{x^a, y^b}$
  \item $g\cdot x^{u}y^{v}\in \ic{\twoideal{g\cdot x^{a}, g\cdot y^{b}}}$
  \item $\frac{u}{a} + \frac{v}{b} \ge 1$
  \end{enumerate}
  Moreover, $\frac{u}{a} + \frac{v}{b} = 1$ implies
  $x^uy^v\in \mingens[l]{\ictwoideal{x^a, y^b}}$.
\end{fact}

\begin{definition}
  Let $f$, $g$, $h\in \K[x,y]$ be monomials such that
  $f\notin \twoideal{g,h}$. We say that $f$ \emph{lies between} $g$
  and $h$ if
  \begin{align*}
    \min\{\deg_xg, \deg_xh\} &< \deg_xf \text{ and}\\
    \min\{\deg_yg, \deg_yh\} &< \deg_yf.
  \end{align*}
\end{definition}

\begin{remark}\label{remark:integral-closure-geometrically}
  Let $f$ be a monomial in between two other monomials, $g$ and $h$,
  in $\K[x,y]$.  Geometrically, Fact~\ref{fact:ic} says that $f$ is in
  the integral closure of $(g,h)$ if and only if it lies above the
  line passing through $g$ and $h$, see
  Figure~\ref{figure:integral-closure}.
\end{remark}

Throughout, we use $\xy$ as a placeholder for either $x$ or
$y$, with the choice remaining fixed within a given context.
\begin{definition}
  Let $I\subseteq\K[x,y]$ be a monomial ideal. We define the
  \emph{$\xy$-distance} $\maxset{\xy}{I}$ of $I$ as
  \begin{align*}
    \maxset{\xy}{I}&\coloneqq\max\{\deg_{\xy}g\mid g\in \mingens{\red{I}}\}.
  \end{align*}
  For a set of monomials $E\subseteq\K[x,y]$, we define
  $\maxset{\xy}{E}$ as $\dist_{\xy}$ of the ideal $(E)$.
\end{definition}

\begin{remark}\label{remark:dist-sum}
  \begin{enumerate}
  \item The $\xy$-distance of an anchored ideal $I$ is the distance
    from the origin to the Newton polyhedron of $I$ measured along the
    $\xy$-axis.
  \item Let $f$, $g$, and $h$ be monomials such that $f$ lies between
    $g$ and $h$. Then
    $\maxsetdeg{\xy}{h,f}+\maxsetdeg{\xy}{f,g} =
    \maxsetdeg{\xy}{h,g}$.
  \item
    $\maxset{\xy}{I}=\max\{\deg_{\xy}g\mid g\in
    \mingens{I}\}-\min\{\deg_{\xy}g\mid g\in \mingens{I}\}$.
  \end{enumerate}
\end{remark}

\begin{definition}\label{definition:deg-gh}
  Let $g$ and $h$ be two monomials. We define the non-standard grading
  $\wdeg{g, h}$ additively by setting
  \begin{align*}
    \wdeg{g,h}(x) &\coloneqq \maxsetdeg{y}{g,h}\text{ and}\\
    \wdeg{g,h}(y) &\coloneqq \maxsetdeg{x}{g,h}.
  \end{align*}
  We set $\wdd{g,h} \coloneqq \wdeg{g,h}(g)= \wdeg{g,h}(h)$.
\end{definition}

\begin{remark}\label{remark:wdd-line}
  \begin{enumerate}
  \item $\wdd{x^a,y^b} = ab$
  \item Let $g$ and $h$ be two monomials and
    $x^{\alpha}y^{\beta}\coloneqq\gcd(g,h)$. Then
    \begin{equation*}
      \wdeg{g,h}(x^uy^v) \ge \wdd{g,h} \quad\Longleftrightarrow\quad  \frac{u-\alpha}{\maxsetdeg{x}{g,h}} + \frac{v-\beta}{\maxsetdeg{y}{g,h}}\ge 1,
    \end{equation*}
    where equality holds simultaneously on both sides. Note that here,
    $x^uy^v$ does not necessarily lie between $g$ and $h$.
  \item Geometrically, the equivalence in (2) states
    $\wdeg{g,h}(x^uy^v) \ge \wdd{g,h}$ if and only if $x^uy^v$ lies
    above the line through $g$ and $h$. Equality holds if it is
    exactly on the line.
  \end{enumerate}
\end{remark}

\begin{lemma}\label{lemma:ic-and-wdeg}
  Let $f$, $g$, $h\in \K[x,y]$ be monomials such that $f$ lies between
  $g$ and~$h$.

  Then the following assertions are equivalent:
  \begin{enumerate}
  \item $f \in \ictwoideal{g,h}$
  \item $\wdeg{g,h}(f) \ge \wdd{g,h}$
  \item $\wdeg{f,h}(g) \le \wdd{f,h}$
  \item $\wdeg{g,f}(h) \le \wdd{g,f}$
  \end{enumerate}

  Moreover, $\wdeg{g,h}(f) = \wdd{g,h}$ implies $f \in \mingens[l]{\ictwoideal{g,h}}$.
\end{lemma}

\begin{proof}
  (1) $\Leftrightarrow$ (2) follows from
  Remark~\ref{remark:wdd-line}(2) and~Fact~\ref{fact:ic}.
  (2)~$\Leftrightarrow$~(3) and (2)~$\Leftrightarrow$~(4) follow from
  Remark~\ref{remark:wdd-line}(3).
\end{proof}

\section{The role of persistent generators}\label{section:persistent-generators}
To analyze higher powers of monomial ideals we introduce the notion of
\emph{(weakly) persistent generators}.  They are monomials whose
powers continue to appear among the minimal generators of $I^n$ for
every $n$.  We use these generators to decompose large powers of $I$
into simpler components.

\begin{definition}\label{def:persistent-generators}
  Let $I$ be a monomial ideal in $\K[x,y]$. We say $f\in \mingens{I}$
  is a \emph{persistent generator of $I$} if
  $f\notin \ictwoideal{g,h}$ for all monomials $g$,
  $h\in I\setminus \{f\}$. We denote the set of all persistent
  generators of $I$ by $P(I)$.
\end{definition}

\begin{remark}\label{remark:persistent-geom}
  \begin{enumerate}
  \item The minimal generators of $I$ with maximal $x$-degree and
    $y$-degree, respectively, are persistent.
  \item The persistent generators of $I$ are the corners of the Newton polyhedron of~$I$.
  \item If $f\in P(I)$, then $f^n\in\mingens[l]{I^n}$ for any $n\in\N$.
  \end{enumerate}
\end{remark}

\begin{definition}\label{def:weakly-persistent}
  Let $I$ be a monomial ideal in $\K[x,y]$. We say that
  $f\in\mingens{I}$ is \emph{weakly persistent}, if
  $f^n\in\mingens{I^n}$ holds for all $n\in\N$. We denote by
  $P^{\ast}(I)$ the set of all weakly persistent generators of $I$.
\end{definition}

\begin{remark}
  By Remark~\ref{remark:persistent-geom}(3), the inclusion
  $P(I)\subseteq P^{\ast}(I)$ holds. Let $g_1$, \dots, $g_{k+1}$ be
  the persistent generators of $I$, ordered in descending $y$-degree.
  With Proposition~\ref{proposition:non-minimal-powers} below we will
  see that
  \begin{equation*}
    P^{\ast}(I)=P(I)\cup \left\{f\in \mingens{I}\longmid \wdeg{g_i, g_{i+1}}(f) = \wdd{g_i, g_{i+1}}\text{ for some }1\le i\le k\right\}\!.
  \end{equation*}
  Hence, the weakly persistent generators correspond to the lattice
  points on the boundary of the Newton polyhedron of $I$, see
  Remark~\ref{remark:wdd-line}(3).
\end{remark}

By definition, $f\in\ic{(g,h)}$ is equivalent to $f^n\in (g,h)^n$ for
some $n$. In the next proposition, we explicitly determine such an
$n$.  In addition, we show that if $f\notin\ic{(g,h)}$, a similar
relation holds among the three polynomials.

\begin{proposition}\label{proposition:non-minimal-powers}
  Let $g$, $h$, $f\in \K[x,y]$ with $f$ lying between $g$ and $h$, and
  define
  \begin{equation*}
    \alpha\coloneqq\maxsetdeg{\xy}{f,h}\quad\text{and}\quad n\coloneqq\maxsetdeg{\xy}{g,h}.
  \end{equation*}
  Then the following assertions hold:
  \begin{enumerate}
  \item If $f\in \ic{(g,h)}$ then $g^{\alpha}h^{n-\alpha} \mid f^{n}$
    and hence $f^{n} \in (g,h)^{n}$.
  \item If $f\notin\ic{(g,h)}$ then $f^{n}\mid g^{\alpha}h^{n-\alpha}$
    and hence $g^{\alpha}h^{n-\alpha} \in (f)^{n}$.
  \end{enumerate}
  Moreover, $\wdeg{g, h}(f) = \wdd{g, h}$ if and only if there exist
  $n\in \N$ and $0\le k\le n$ such that $g^{k}h^{n-k} = f^n$.
\end{proposition}
\begin{proof}
  Recall that $n-\alpha = \maxsetdeg{\xy}{g,f}$
  (Remark~\ref{remark:dist-sum}(1)).  Since $f$ lies between~$g$ and~$h$
  it follows that all three monomials are divisible by
  $\gcd(g,h)$.  Therefore, we can assume that $g = y^b$, $h = x^a$, and
  $f=x^uy^v$ with $0<u<a$ and $0<v<b$.
  It suffices to show the assertions for $\xy=x$.
  It follows that $n=a$ and $\alpha=a-u$. Note that
  \begin{equation*}
    \deg_x\!\big(x^{au}y^{b(a-u)}\big) = au = \deg_x\!\left(f^a\right)\!.
  \end{equation*}
  By Fact~\ref{fact:ic}, $f\in\ic{(g,h)}$ if and only if
  $\frac{u}{a} + \frac{v}{b} \ge 1$ which, in turn, is equivalent to
  \begin{equation*}
    \deg_y\!\big(x^{au}y^{b(a-u)}\big) = b(a-u) \le av = \deg_y\!\left(f^a\right)\!.
  \end{equation*}
  For the last assertion note that the right-hand side holds if and
  only if
  \begin{equation*}
    ak = un \text{ and } b(n-k) = vn,
  \end{equation*}
  which is equivalent to $\frac{u}{a} + \frac{v}{b} = 1$. The latter
  is equivalent to the left-hand side,
  cf.~Remark~\ref{remark:wdd-line}(2).
\end{proof}

\begin{remark}
  Note that in the second assertion of
  Proposition~\ref{proposition:non-minimal-powers}, the assumption
  that $f\notin\ic{(g,h)}$ implies that the equality
  $\wdeg{g, h}(f) = \wdd{g, h}$ cannot hold. In particular, we have
  that $f^n\neq g^{\alpha}h^{n-\alpha}$.
\end{remark}

\begin{remark}
  The bound $\min\{\maxsetdeg{x}{g,h}, \maxsetdeg{y}{g,h}\}$ for $n$
  in Proposition~\ref{proposition:non-minimal-powers} is sharp, see
  Example~\ref{example:bound-ic} below.
\end{remark}

\begin{example}\label{example:bound-ic} Let $g=y^5$ and $h=x^6$. Then
  $\min\{\maxsetdeg{x}{g,h}, \maxsetdeg{y}{g,h}\}=5$.
  \begin{enumerate}
  \item The monomial $f=x^5y$ is an element of $\ic{(g,h)}$ for which
    one can easily verify that $f^n\notin(x^6,y^5)^n$ for $n\le 4$.
  \item The monomial $f=xy^4$ lies between $g$ and $h$ such that
    $f\notin\ic{(g,h)}$. A straightforward computation shows that
    there is no $n< 5$ such that there exist $\alpha$, $\beta$ with
    $\alpha+ \beta=n$ and $g^{\alpha}h^{\beta}\in(f)^{n}$.
  \end{enumerate}
\end{example}

\begin{notation}\label{notation:delta}
  Let $F=\{g_1,\dots, g_{k+1}\}$ be a set of monomials such that
  $g_1$, \dots, $g_{k+1}$ are ordered in descending $y$-degree.  We
  set
  \begin{equation*}
    \delta_F \coloneqq \max_{1\le i\le k}\!\big\{\min\{\maxsetdeg{x}{g_i, g_{i+1}}, \maxsetdeg{y}{g_i, g_{i+1}}\}\big\} - 1.
  \end{equation*}
\end{notation}

\begin{remark}
  \begin{enumerate}
  \item If $g$ lies between $g_1$ and $g_{k+1}$, then
    $\delta_{F}\ge\delta_{F\cup\{g\}}$.
  \item If $f\in\mingens{I}$ is not weakly persistent, then
    Proposition~\ref{proposition:non-minimal-powers} implies that
    $f^{\delta_{P(I)}}\notin\mingens{I^{\delta_{P(I)}}}$.
  \end{enumerate}
\end{remark}

By definition, powers of weakly persistent generators remain minimal
generators. Proposition~\ref{proposition:non-minimal-powers} yields a
description of powers of all elements of $I$:

\begin{corollary}\label{corollary:non-minimal-generators}
  Let $I$ be a monomial ideal in $\K[x,y]$ and let
  $P(I)\subseteq P\subseteq P^{\ast}(I)$ such that
  $P=\{g_1,\dots, g_{k+1}\}$ and $g_1,$ \dots, $g_{k+1}$ are ordered
  in descending $y$-degree.  For every $n\ge \delta_P$ and $f\in I$
  there exist $1\le i\le k$ and $a\le \delta_P$ such that
  \begin{equation*}
    f^{n} \in \twoideal{g_i,g_{i+1}}^{n-a}\cdot f^{a}.
  \end{equation*}
\end{corollary}

\begin{proof}
  For elements in $P$, the assertion trivially holds.  Note that the
  elements $f\in I\setminus P$ are not persistent.  By definition,
  there exist $g$, $h\in I\setminus\{f\}$ with $f\in \ic{(g,h)}$.  We
  can choose $g$ and $h$ to be in $P$ such that no other element of
  $P$ lies between them, meaning $g=g_i$ and $h= g_{i+1}$ for some
  $1\le i \le k$.  We write
  $d\coloneqq \min\{\maxsetdeg{x}{g_i, g_{i+1}}, \maxsetdeg{y}{g_i,
    g_{i+1}}\}$. Then $n\ge\delta_P\ge d-1$ and we can write $n=qd+a$
  with $q\in \N_0$ and $a\le d-1$. It follows from
  Proposition~\ref{proposition:non-minimal-powers}(1) that
  \begin{equation*}
    f^{n} = f^{qd+a} \in
    \twoideal{g_i,g_{i+1}}^{qd}\cdot f^{a}.
  \end{equation*}
\end{proof}

Corollary~\ref{corollary:non-minimal-generators} immediately yields:
\begin{corollary}\label{corollary:reduction-number}
  Let $I\subseteq\K[x,y]$ be a monomial ideal, $n\in \N$, and let
  $P(I)\subseteq P\subseteq P^{\ast}(I)$.  Then every minimal
  generator of $I^n$ is of the form
  \begin{equation}\label{eq:gen-form-1}
    \prod_{g\in P}g^{\ell_g}\cdot \prod_{f\in \mingens{I}\setminus P}f^{k_f},
  \end{equation}
  where $0\le k_f \le \delta_P$ for all $f\in \mingens{I}\setminus P$
  and $\sum_g\ell_g + \sum_f k_f = n$.
\end{corollary}

\begin{example}
  Let $I = (g,f,h)$ with $g=y^5$, $f=x^5y$, and $h=x^6$ from
  Example~\ref{example:bound-ic}(1) with $P = P(I) = \{g,h\}$.  Since
  $f^5 = x g h^4$, it follows that, for instance, the monomial
  $g^{20}h^3f^{12} = g^{20}h^3f^{2\cdot 5+2}\in I^{35}$ is divisible
  by $g^{20+2} \cdot h^{3+8} \cdot f^{2}\in I^{35}$.
\end{example}

\begin{remark}\label{remark:reduction-number}
  Singla~\cite[Proposition~2.1]{singla:2006:monomial-reduction}
  established that the ideal $\mathfrak{a}$, generated by the
  persistent generators $P(I)$, is a reduction of $I$. That is, there
  exists $\delta\in \N$ such that for all $n\in\N_0$,
  \begin{equation*}
    I^{\delta+n}=\mathfrak{a}^{n}I^{\delta}.
  \end{equation*}
  Singla's result applies to monomial ideals in any number of
  variables.  Through Corollary~\ref{corollary:reduction-number} we
  recover Singla's result for the bivariate case.  We further show
  that $\delta$ can be chosen to be at most
  $|\mingens{I}\setminus P(I)|\cdot\delta_{P(I)}$.  The minimal choice
  for $\delta$ is called \emph{reduction number} of $I$ with respect
  to $\mathfrak{a}$. This number measures from which power on the
  ideal $\mathfrak{a}$ determines the behavior of all further powers
  of $I$.
\end{remark}

With Proposition~\ref{proposition:non-minimal-powers}(2), we further
refine the statement of Corollary~\ref{corollary:reduction-number} in
Theorem~\ref{theorem:I^D+l} below.

\begin{notation}\label{notation:dI}
  For a monomial ideal $I$ and $P(I)\subseteq P\subseteq
  P^{\ast}(I)$. We set
  \begin{equation*}
    \label{eq:dI}
    d_P \coloneqq
    \begin{cases}
      \min\{\maxset{x}{I}, \maxset{y}{I}\}-2, &\text{if }|P|>2,\\
      0, &\text{if }|P|=2.
    \end{cases}
  \end{equation*}
  Note that $|P|>2$ implies that
  $\min\{\maxset{x}{I}, \maxset{y}{I}\} \ge 2$.
\end{notation}

\begin{theorem}\label{theorem:I^D+l}
  Let $I$ be a monomial ideal in $\K[x,y]$ and let
  $P(I)\subseteq P\subseteq P^{\ast}(I)$ such that
  $P=\{g_1,\dots, g_{k+1}\}$ and $g_1,$ \dots, $g_{k+1}$ are ordered
  in descending $y$-degree.  Further, let
  $D \ge (\mu(I)-|P|)\cdot\delta_P + |P|\cdot d_P$.  Then for all
  $\ell \ge 0$,
  \begin{equation*}
    I^{D+\ell} = \sum_{i=1}^{k}(g_i, g_{i+1})^{\ell}I^D.
  \end{equation*}
\end{theorem}

\begin{proof}
  The inclusion ``$\supseteq$'' is trivial.  We prove
  ``$\subseteq$''. Write $N\coloneqq \mingens{I}\setminus P$.  By
  Corollary~\ref{corollary:reduction-number} we can write
  $I^{D+\ell}= I^{\delta +n} = \mathfrak{a}^{n}I^{\delta}$ where
  $\delta\coloneqq|N|\cdot \delta_P$, and
  $n\coloneqq D-\delta+\ell \ge(k+1)d_P+\ell$, and $\mathfrak{a}$ is
  the ideal generated by $P$.  Thus, every minimal generator $F$ of
  $I^{\delta +n}$ is of the form $F=g\cdot f$, where $g$ is a product
  of $n$ elements in $P$ and $f\in I^{\delta}$.

  \textbf{Claim.} There exists $1\le i\le k$ such that
  \begin{equation*}
    g\in(g_i,g_{i+1})^{n-(k+1)d_P}\mathfrak{a}^{(k+1)d_P}.
  \end{equation*}
  We write $g=g_1^{n_1}\cdots g_{k+1}^{n_{k}}$, where $n_i\in\N_0$
  with $\sum_i n_i=n$.  The assertion of the claim holds trivially in
  the following cases:
  \begin{enumerate}
  \item at most one $n_i> d_P$,
  \item $n_j, n_{j+1}>d_P$ for some $1\le j\le k$ and $n_i\le d_P$ for
    all $i\notin\{ j, j+1\}$.
  \end{enumerate}
  Otherwise, we take $a\coloneqq\min\{i\mid n_i> d_P\}$ and
  $b\coloneqq\max\{i\mid n_i>d_P\}$. Note that $a+1<b$.  The following
  argument may be repeated as needed; in every step either $a$
  increases strictly or $b$ decreases strictly. For readability, we
  may therefore assume without loss of generality that $a = 1$ and
  $b = k+1$. By Proposition~\ref{proposition:non-minimal-powers}(2),
  we then have
  \begin{equation*}
    g_2^{\dist_{\xy}(g_1,g_{k+1})} \mid g_1^{\dist_{\xy}(g_2,g_{k+1})}g_{k+1}^{\dist_{\xy}(g_1,g_2)}.
  \end{equation*}
  Since $g$ is a minimal generator, equality must hold, so
  \begin{equation*}
    g = g_1^{n_1-\dist_{\xy}(g_2,g_{k+1})} g_2^{n_2+\dist_{\xy}(g_1,g_{k+1})} g_3^{n_3}\cdots g_k^{n_k}g_{k+1}^{n_{k+1}-\dist_{\xy}(g_1,g_2)}.
  \end{equation*}
  We iteratively apply
  Proposition~\ref{proposition:non-minimal-powers}(2) until at least
  one of the exponents of $g_1$ or $g_2$ becomes $\le d_P$.  At this
  stage, we redefine $a'\coloneqq\min\{i\mid n_i'> d_P\}$ and
  $b'\coloneqq\max\{i\mid n_i'>d_P\}$ where $n_i'$ is the new exponent
  of $g_i$.  Now $a'>a$ or $b'<b$ must hold.  Hence, by repeating this
  argument from the top we eventually must reach one of the trivial
  cases (1) or (2), where the claim follows immediately.
\end{proof}

\begin{example}\label{example:I^D+l}
  Let $I = (g_1,g_2,g_3)$ with $g_1=y^5$, $g_2=xy^4$, and $g_3=x^6$
  from Example~\ref{example:bound-ic}.  Then
  $P(I) = P^{\ast}(I)=\{g_1,g_2,g_3\}$.  Since $g_2^5\mid g_1^4g_3$,
  it follows that, for instance, the monomial
  $g_1^{11}g_2^{102}g_3^{42}$ is an element of
  $(g_1^3 \cdot g_2^{102+10} \cdot g_3^{40})\subseteq
  (g_2,g_3)^{152}I^3$.
\end{example}

\begin{remark}
  In Example~\ref{example:I^D+l}, a direct computation shows that
  $D=3$ is large enough, although Theorem~\ref{theorem:I^D+l} gives a
  lower bound of $9$.  In general, a tight bound for the minimal
  value for $D$ depends on all the exponents in $\mingens{I}$.
\end{remark}

\begin{remark}\label{remark:use-weakly-persistent-instead}
  The set $P$ may be chosen closer to either $P(I)$ or $P^{\ast}(I)$,
  depending on the specific context in which
  Theorem~\ref{theorem:I^D+l} is applied.  For instance, if the
  objective is to minimize $(\mu(I)-|P|)\cdot\delta_P + |P|\cdot d_P$,
  then $P$ can be selected based on the values of $d_P$
  and~$\delta_P$.
\end{remark}

\section{Ideals with regular staircase factors}
\label{section:staircase-factors}

\begin{wrapfigure}{r}{0.3\textwidth}
   \centering
     \begin{tikzpicture}
    \pgfmathsetmacro{\scalefactor}{0.2}
    \pgfmathsetmacro{\xa}{1}
    \pgfmathsetmacro{\yb}{2}
    \pgfmathsetmacro{\xshift}{0}
    \pgfmathsetmacro{\yshift}{0}
    \pgfmathsetmacro{\power}{3}
    \def\Jstring{0/5,4/0}

    \computeParameters[1]
    \drawAxes
    \drawJshape{0}
    \drawStairs{1}
  \end{tikzpicture}

%%% Local Variables:
%%% mode: latex
%%% TeX-master: "generators-powers-bivariate"
%%% End:
   \caption{}
   \label{figure:staircase-ideal}
\end{wrapfigure}
In this section we study the minimal generators of sums of ideals of
the form $(g,h)^nJ$, in light of Theorem~\ref{theorem:I^D+l}. Here,
$g$ and $h$ are monomials, and $J$ is a (fixed) anchored monomial
ideal in $\K[x,y]$. The main result of this section is
Theorem~\ref{theorem:gluing-more-segments}.

We begin with the special case $(x^u,y^v)^nJ$ for $u$, $v\in \N$.  By
drawing its exponents in the $xy$-plane, the ideal $(x^u,y^v)^n$ looks
like a ``regular staircase'' in the sense that all $n$ steps in its
staircase are of the same size. The minimal generators of the product
$(x^u,y^u)^nJ$ are of the form $x^{u(n-i)}y^{vi}f$ with
$f\in \mingens{J}$. In general, not all elements of this form are
minimal generators, as divisibility relations may occur among them.
In Figure~\ref{figure:staircase-ideal} we visualize the potential
cancellations among the minimal generators in the product
$(x^u,y^v)^nJ$.  Although Figure~\ref{figure:staircase-ideal} does
not show the actual generators of the ideal~$J$, it illustrates how
the overlaps of the shifted copies of~$J$ form a
repeating pattern as $n$ increases.  We formalize this ``pattern
repetition'' in Theorem~\ref{theorem:generators-one-segment} below.
Before that, we establish in Lemma~\ref{lemma:y-sections} that
partitioning the elements of $(x^u,y^v)^nJ$ based on their $y$-degrees
reveals divisibilities by certain powers of $x^u$ or
$y^v$. Figure~\ref{figure:y-sections-x-cuts} visualizes this
partition.

\begin{remark}\label{remark:why-y-sections}
  The choice of partitioning by $y$-degree is arbitrary; all results
  of this section remain valid if we instead partition by $x$-degree,
  simply by interchanging the roles of $x$ and $y$.
\end{remark}

\begin{lemma}\label{lemma:y-sections}
  Let $u$, $v\in \N$ and $J\subseteq\K[x,y]$ be an anchored monomial
  ideal. Moreover, let
  $r \ge\left\lceil \frac{\maxset{y}{J}}{v}\right\rceil$ and $n\ge
  r$. For $r \le j \le n$, we set
  \begin{align*}
    \mathcal{U}_{j} &= \left\{ F\in \twoideal{x^u, y^v}^{n}J \longmid \deg_yF \ge jv\right\}\! \text{, and }\\
    \mathcal{L}_{j} &= \left\{ F\in \twoideal{x^u, y^v}^{n}J \longmid  \deg_yF < jv\right\}\!.
  \end{align*}
  Then, for $r \le j \le n$,
  \begin{equation*}
    \mathcal{U}_{j} \subseteq y^{v(j-r)}\cdot \twoideal{x^u, y^v}^{n-(j-r)}J
    \text{ and }
    \mathcal{L}_{j} \subseteq x^{u(n-j)}\cdot \twoideal{x^u, y^v}^{j}J.
  \end{equation*}
  In particular, for $r\le j \le n-1$,
  \begin{equation*}
    \mathcal{U}_{j} \cap \mathcal{L}_{j+1}
    = x^{u(n-(j+1))}y^{v(j-r)}\cdot\left\{ f\in \twoideal{x^u, y^v}^{r+1}J \longmid rv \le \deg_yf < (r+1)v \right\}\!.
  \end{equation*}
\end{lemma}

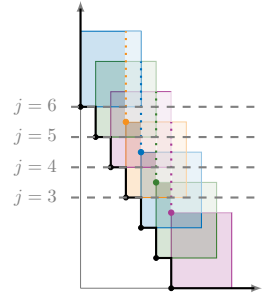
\begin{figure}[h]
  \centering
  \begin{minipage}[t]{0.64\textwidth}
    \caption{We partition $(x^u,y^v)^nJ$ into sections based on the
      $y$-degree (indicated by dashed lines). Additionally, note that
      the upper left corner of each rectangle is in $(x^u,y^v)^nJ$
      (since we assumed $J$ to be anchored). This bounds the
      $x$-degree in each $y$-section; see
      Remark~\ref{remark:y-sections-and-x-degree}.}
    \label{figure:y-sections-x-cuts}
  \end{minipage}\hspace*{-1em}
  \begin{minipage}[t]{0.35\textwidth}
    \centering
    \begin{tikzpicture}[baseline=(current bounding box.north)]
    \pgfmathsetmacro{\scalefactor}{0.2}
    \pgfmathsetmacro{\xa}{1}
    \pgfmathsetmacro{\yb}{2}
    \pgfmathsetmacro{\xshift}{0}
    \pgfmathsetmacro{\yshift}{0}
    \pgfmathsetmacro{\power}{6}
    \def\Jstring{0/5,4/0}

    \computeParameters[1]
    \drawAxes
    \drawJshape{1}
    \drawStairs{1}
    \drawYSections
  \end{tikzpicture}

%%% Local Variables:
%%% mode: latex
%%% TeX-master: "generators-powers-bivariate"
%%% End:
  \end{minipage}\hfill
\end{figure}

\begin{proof}
  For the first two inclusions we write $F=x^{u(n-i)}y^{vi}\tilde f$
  with $0\le i\le n$ and $\tilde f\in \mingens{J}$ and separate into
  two cases.

  $\mathcal{U}_j$: Note that $\deg_y\tilde f \le
  \maxset{y}{J}$. Hence, $\deg_yF \ge jv$ implies
  $i\ge j - \frac{\maxset{y}{J}}{v} \ge j - r\ge 0$, that is,
  $F = y^{v(j-r)}\cdot f$ with
  $f =x^{u(n-i)}y^{v(i-(j-r))}\tilde f\in
  \twoideal{x^u,y^v}^{n-(j-r)}J$.

  $\mathcal{L}_j$: The condition $\deg_yF < jv$ implies $i\le j$ and
  $n-i \ge n-j \ge 0$. Therefore, $F = x^{u(n-j)}\cdot f$ with
  $f =x^{u(j-i)}y^{vi}\tilde f\in (x^u, y^v)^{j}J$.

  Finally, for the last equality, ``$\subseteq$'' follows from the
  above while ``$\supseteq$'' is obvious.
\end{proof}

\begin{remark}\label{remark:y-sections-and-x-degree}
  As preparation for later arguments, we provide a bound on the
  $x$-degrees of elements in $\mathcal{U}_j$: Let $J$ be an anchored
  monomial ideal with $b\coloneqq \maxset{y}{J}$. With the notation of
  Lemma~\ref{lemma:y-sections}, if $n>j-r$ and $j\ge r$, then
  $H\coloneqq x^{u(n-j+r)}y^{v(j-r)}y^b$ is an element of
  $(x^u,y^v)^{n}J$ satisfying
  \begin{equation}\label{eq:H}
    \deg_yH = v(j-r)+b \le jv\quad \text{ and }\quad\deg_xH = (n-j+r)u,
  \end{equation}
  cf.~Figure~\ref{figure:y-sections-x-cuts}.  Consequently, for
  $n>j-r$,
  \begin{equation*}
    f\in \mathcal{U}_{j}\cap \mingens[l]{(x^u,y^v)^{n}J} \quad\Longrightarrow\quad \deg_xf \le (n-j+r)u,
  \end{equation*}
  and equality can only hold if $H=f$. In particular, if there exists
  $f\in \mathcal{U}_{j}\cap \mingens{(x^u,y^v)^{n}J}$ with
  $\deg_xf = (n+r-j)u$, then $H=f\in \mathcal{U}_{j}$ which,
  considering the $y$-degree of $H$, further implies $r= \frac{b}{v}$
  and $\deg_y(f) = \dist_y(J)$.
\end{remark}

We are now set to formalize the pattern repetition mentioned at the
beginning of this section. Based on the $y$-degrees, we define three
subideals $L$, $M$, and $R$ of $(x^u,y^v)^{r+1}J$. We use these to
assemble the minimal generators of $(x^u,y^v)^{r+1+\ell}J$ for all
$\ell\ge0$ by suitable shifts.
\begin{theorem}\label{theorem:generators-one-segment}
  Let $u$, $v\in\N$ and $J\subseteq\K[x,y]$ be an anchored monomial
  ideal.  Then, for all
  $r \ge \left\lceil \frac{\maxset{y}{J}}{v}\right\rceil$ and
  $\ell\in \N_0$,
  \begin{equation}\label{eq:disjoint-union-one-segment}
    \mingens[l]{(x^u,y^v)^{r+1+\ell}J} =
    y^{v\ell}\leftpart  \uplus
    \biguplus_{j=1}^{\ell}x^{uj}y^{v(\ell-j)}\middlepart
    \uplus x^{u\ell}\rightpart ,
  \end{equation}
  where
  \begin{align*}
    \leftpart   &=\left\{f \in \mingens[l]{(x^u,y^v)^{r+1}J}\longmid  \deg_yf \ge rv \right\}\!, \\
    \middlepart &=\left\{f \in \mingens[l]{(x^u,y^v)^{r+1}J}\longmid rv \le \deg_yf < (r+1)v \right\}\!, \text{ and} \\
    \rightpart  &=\left\{f \in \mingens[l]{(x^u,y^v)^{r+1}J}\longmid \deg_yf < rv \right\}\!. \\
  \end{align*}
  In particular,
  \begin{equation*}
    \mu\!\left((x^u,y^v)^{r+1+\ell}J\right)
    = \mu\!\left((x^u, y^v)^{r+1}J\right) + \ell\cdot|\middlepart |.
  \end{equation*}
\end{theorem}

\begin{proof}
  Note that the count is an immediate consequence of the first
  assertion.

  With the notation of Lemma~\ref{lemma:y-sections}, we have
  \begin{equation}\label{eq:union-of-y-cuts}
    \twoideal{x^u, y^v}^{r+1+\ell}J
    = \mathcal{L}_r
    \uplus \biguplus_{j=r}^{r+\ell-1} (\mathcal{U}_j \cap \mathcal{L}_{j+1})
    \uplus \mathcal{U}_{r+\ell}.
  \end{equation}
  We claim that the following three statements hold:
  \begin{enumerate}
  \item\label{pf:1}
    $\mathcal{U}_{r+\ell}\cap\mingens[l]{\twoideal{x^u, y^v}^{r+1+\ell}J}=
    y^{v\ell}\cdot\leftpart$,
  \item\label{pf:2}
    $\mathcal{L}_r\cap \mingens[l]{\twoideal{x^u, y^v}^{r+1+\ell}J}=
    x^{u\ell}\cdot\rightpart$, and
  \item\label{pf:3} for all $r\le j \le r+\ell-1$,
    \begin{equation*}
      (\mathcal{U}_{j}\cap \mathcal{L}_{j+1})\cap\mingens[l]{\twoideal{x^u, y^v}^{r+1+\ell}J}
      =x^{u(r + \ell-j)}y^{v(j-r)}\cdot\middlepart .
    \end{equation*}
  \end{enumerate}
  If the claim holds, then the assertion follows
  from~\eqref{eq:union-of-y-cuts} since, arranging the sets in reverse
  order,
  \begin{equation*}
    \mingens[l]{(x^u,y^v)^{r+1+\ell}J} \cap \biguplus_{j=r}^{r+\ell-1} (\mathcal{U}_j \cap \mathcal{L}_{j+1})
    =  \biguplus_{j=1}^{\ell} x^{uj}y^{v(\ell-j)}\cdot\middlepart.
  \end{equation*}

  In all three cases, the inequalities that the $y$-degrees must
  satisfy are the same on both sides of the equality.  Moreover, the
  inclusions ``$\subseteq$'' all hold due to
  Lemma~\ref{lemma:y-sections} and the fact that a generator
  $g\cdot f$ of $(x^u,y^v)^{r+1+\ell} J$ with
  $g\in (x^u,y^v)^{r+1+\ell}$ and $f\in J$ can only be minimal,
  provided that $f$ is a minimal generator of $J$ (the same holds for
  $g$ but is not relevant here).

  For the reverse inclusions, it is in all three cases left to show
  that every element of the set on the right is a minimal generator of
  $\twoideal{x^u, y^v}^{r+1+\ell}J$.  To do so, take
  \begin{equation*}
    H=x^{u(r+1+\ell-i)} y^{vi} h\in (x^u,y^v)^{r+1+\ell}J
  \end{equation*}
  with $h\in J$ and $0\le i \le \ell+r+1$.

  (\ref{pf:1}) Let $F = y^{v\ell}f$ with
  $f\in \mingens[l]{(x^u, y^v)^{r+1}J}$ and $\deg_yf\ge rv$.  Then
  $\deg_xF = \deg_xf \le (r+1)u$,
  cf.~Remark~\ref{remark:y-sections-and-x-degree}.  If $H$ divides
  $F$, then $\deg_xH\le\deg_xF$ implies $(r+1+\ell -i)u \le (r+1)u$,
  that is, $i\ge \ell$. We cancel out $y^{v\ell}$ to conclude that
  \begin{equation*}
    x^{u(r+1+\ell-i)} y^{v(i-\ell)} h\mid f,
  \end{equation*}
  which, since $f\in \mingens[l]{(x^u, y^v)^{r+1}J}$, implies that $H=F$.

  (\ref{pf:2}) Let $F=x^{u\ell}f$ with
  $f\in \mingens[l]{(x^u, y^v)^{r+1}J}$ and $\deg_yf< rv$.  Again,
  assume that $H$ divides $F$.  Then $\deg_yH=iv+\deg_yh<rv$ and hence
  $i<r$.  This implies that $r+\ell+1-i>\ell+1$, thus we can cancel
  out $x^{u\ell}$ on both sides and end up with the same conclusion as
  in (\ref{pf:1}).

  (\ref{pf:3}) Let $r\le j \le r+\ell-1$ and
  $F=x^{u(r + \ell-j)}y^{v(j-r)}f$ with $f\in M$. Hence,
  \begin{equation*}
        \deg_yF< v(j-r)+v(r+1) = v(j+1).
  \end{equation*}
  As above, if $H\mid F$, then the $y$-degree of
  $H$ must be less or equal than the $y$-degree of $F$, and therefore
  $i\le j+1$. This implies $r+\ell+1-i\ge r+\ell-j$, so we can cancel
  out $x^{u(r+\ell-j)}$ which leaves us with
  \begin{equation*}
    x^{u(j-i+1)}y^{vi}h\mid y^{v(j-r)}f.
  \end{equation*}
  Since $f\in M\subseteq L$, it follows from~(\ref{pf:1}) that
  $y^{v(j-r)}f$ is a minimal generator of $(x^u,y^v)^{j+1}J$.  We
  conclude that $H=F$.
\end{proof}

Let us unravel Theorem~\ref{theorem:generators-one-segment} in an
example.
\begin{example}\label{example:1}
  Let $u=3$, $v=4$, and
  \begin{equation*}
    J = (y^{10}, x^2y^7, x^3y^5, x^5y^4, x^7y^2,x^9)
  \end{equation*}
  as depicted in the (gray) rectangle in
  Figure~\ref{figure:example-thm-2}.  We choose $r=3$, which is the
  minimal possible choice. The lower two dashed lines in the figure
  mark the areas where the $y$-degree is in between $rv$ and $(r+1)v$.

  The left part of Figure~\ref{figure:example-thm-2} shows
  $(x^3,y^4)^{r+1}J$. The set $\leftpart$ consists of generators above
  the line $j=r$, marked with (blue) circles. The set $\rightpart$
  consists of the (orange) squares below $j=r$. The middle set
  $\middlepart$ contains the two encircled (in red) generators between
  the two lines $j=r$ and $j=r+1$.

  On the right side of Figure~\ref{figure:example-thm-2}, we see
  $(x^3,y^4)^{r+3}J$. The minimal generators are a disjoint union of
  the sets
  \begin{equation*}
    y^{4\cdot 2}\leftpart,\quad x^3y^4\middlepart,\quad
    x^{3\cdot 2}\middlepart,\quad \text{and}\quad x^{3\cdot 2}\rightpart.
  \end{equation*}
  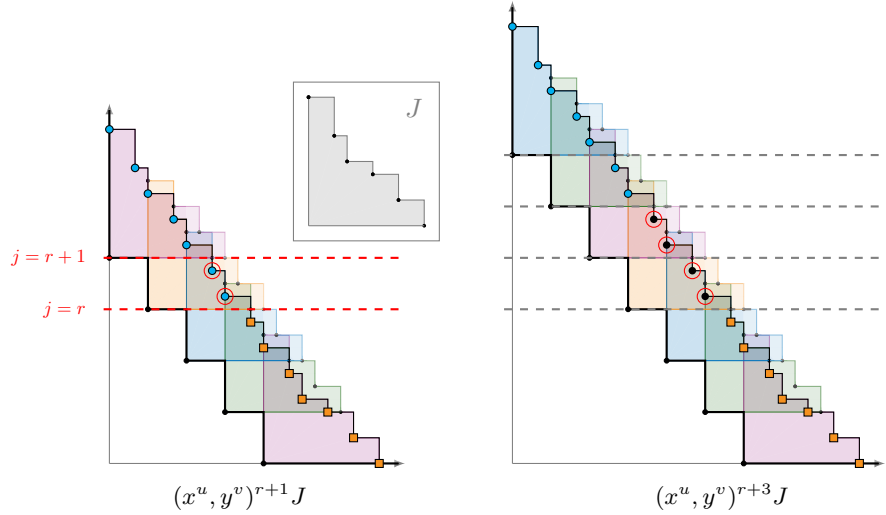
\begin{figure}[h]
    \centering
    \begin{subfigure}[b]{0.49\textwidth}
      \begin{tikzpicture}
  \pgfmathsetmacro{\scalefactor}{0.17}
  \pgfmathsetmacro{\xa}{3}
  \pgfmathsetmacro{\yb}{4}
  \pgfmathsetmacro{\xshift}{0}
  \pgfmathsetmacro{\yshift}{0}
  \pgfmathsetmacro{\power}{4}
  \pgfmathsetmacro{\positionx}{-2.5}
  \pgfmathsetmacro{\positiony}{-1.5}
  \pgfmathsetmacro{\offset}{0.2}
  \renewcommand*{\Jstring}{0/10,2/7,3/5,5/4,7/2,9/0}%

  \computeParameters[1]
  \drawAxes
  \drawFullJ[1]{0.3}
  \drawStairs{1}

  \coordinate (g) at ($\scalefactor*(3,5)$);
  \coordinate (h1)at ($\scalefactor*(0,10)$);
  \coordinate (h2)at ($\scalefactor*(2,7)$);
  \coordinate (h3)at ($\scalefactor*(9,0)$);
  \coordinate (h4)at ($\scalefactor*(7,2)$);
  \coordinate (h5)at ($\scalefactor*(5,4)$);

  \draw[black]
  let
  \p1 = ($\power*(g1)$),
  \p2 = ($\power*(g1) - (g1) +(g2)$),
  \p3 = ($(\p1)+ (h1)$),
  \p4 = ($(\p1)+ (h2)$),
  \p5 = ($(\p1)+ (g) $),
  \p6 = ($(\p2)+ (h2)$),
  \p7 = ($(\p2)+ (g) $),
  in
  (yend) -- (\p3) -- (\x4,\y3) -- (\p4) -- (\x5,\y4) -- (\p5) --(\x6,\y5) --(\p6) --(\x7,\y6)--(\p7);

  \draw[black]
  let
  \p1 = ($\power*(g1) -   (g1) +  (g2)$),
  \p2 = ($\power*(g1) - 2*(g1) +2*(g2)$),
  \p3 = ($(\p1)+ (h1)$),
  \p4 = ($(\p1)+ (h2)$),
  \p5 = ($(\p1)+ (g) $),
  \p6 = ($(\p2)+ (h2)$),
  \p7 = ($(\p2)+ (g) $),
  in
  (\p4) -- (\x5,\y4) -- (\p5) --(\x6,\y5) --(\p6) --(\x7,\y6)--(\p7);

  \draw[black]
  let
  \p1 = ($\power*(g2)$),
  \p2 = ($\power*(g2) - (g2) +(g1)$),
  \p3 = ($(\p1)+ (h3)$),
  \p4 = ($(\p1)+ (h4)$),
  \p5 = ($(\p1)+ (h5) $),
  \p6 = ($(\p1)+ (g)$),
  \p7 = ($(\p1)+ (h2) $),
  \p8 = ($(\p2)+ (g) $),
  in
 (xend) -- (\p3) -- (\x3,\y4) -- (\p4) -- (\x4,\y5) -- (\p5) --(\x5,\y6) --(\p6) --(\x6,\y7)--(\p7) --(\x7,\y8)--(\p8);

  \draw[black]
  let
  \p1 = ($\power*(g2) - 2*(g2)  +2*(g1)$),
  \p2 = ($\power*(g2) - (g2) +(g1)$),
  \p3 = ($(\p2)+ (g) $),
  \p4 = ($(\p2)+ (h2) $),
  \p5 = ($(\p1)+ (g)$),
  in
  (\p3) -- (\x3,\y4) -- (\p4) -- (\x4,\y5) -- (\p5);

  \node[fill=ProcessBlue, draw, inner sep=1pt, circle] at ($\power*(g1)              +(h1)$) {};
  \node[fill=ProcessBlue, draw, inner sep=1pt, circle] at ($\power*(g1)              +(h2)$) {};
  \node[fill=ProcessBlue, draw, inner sep=1pt, circle] at ($\power*(g1)              +(g)$)  {};
  \node[fill=ProcessBlue, draw, inner sep=1pt, circle] at ($\power*(g1)-  (g1)+  (g2)+(h2)$) {};
  \node[fill=ProcessBlue, draw, inner sep=1pt, circle] at ($\power*(g1)-  (g1)+  (g2)+(g)$)  {};
  \node[fill=ProcessBlue, draw, inner sep=1pt, circle] (m1) at ($\power*(g1)-2*(g1)+2*(g2)+(h2)$) {};
  \node[fill=ProcessBlue, draw, inner sep=1pt, circle] (m2) at ($\power*(g1)-2*(g1)+2*(g2)+(g)$)  {};
  \draw[red] (m1) circle (3pt);
  \draw[red] (m2) circle (3pt);

  \node[fill=BurntOrange, draw, inner sep=1.5pt] at ($\power*(g1)-3*(g1)+3*(g2)+(h2)$) {};
  \node[fill=BurntOrange, draw, inner sep=1.5pt] at ($\power*(g1)-3*(g1)+3*(g2)+(g)$)  {};
  \node[fill=BurntOrange, draw, inner sep=1.5pt] at ($\power*(g2)+(h2)$) {};
  \node[fill=BurntOrange, draw, inner sep=1.5pt] at ($\power*(g2)+(g)$)  {};
  \node[fill=BurntOrange, draw, inner sep=1.5pt] at ($\power*(g2)+(h3)$) {};
  \node[fill=BurntOrange, draw, inner sep=1.5pt] at ($\power*(g2)+(h4)$) {};
  \node[fill=BurntOrange, draw, inner sep=1.5pt] at ($\power*(g2)+(h5)$) {};

  \draw [red, thick, dashed]
  let
  \p1 = ($4*(g1)$),
  \p2 = (xend),
  in
  (\x2,\y1) -- (-0.15,\y1) node [left, xshift=-2pt, scale=0.7] {$j=r+1$} ;

  \draw [red, thick, dashed]
  let
  \p1 = ($3*(g1)$),
  \p2 = (xend),
  in
  (\x2,\y1) -- (-0.15,\y1) node [left, xshift=-2pt, scale=0.7] {$j=r$} ;

\end{tikzpicture}
%%% Local Variables:
%%% mode: latex
%%% TeX-master: "generators-powers-bivariate"
%%% End:
      \caption{$(x^u,y^v)^{r+1}J$}
    \end{subfigure}
    \begin{subfigure}[b]{0.5\textwidth}
      \begin{tikzpicture}
  \pgfmathsetmacro{\scalefactor}{0.17}
  \pgfmathsetmacro{\xa}{3}
  \pgfmathsetmacro{\yb}{4}
  \pgfmathsetmacro{\xshift}{0}
  \pgfmathsetmacro{\yshift}{0}
  \pgfmathsetmacro{\power}{6}
  \pgfmathsetmacro{\positionx}{-2.5}
  \pgfmathsetmacro{\positiony}{-1.5}
  \pgfmathsetmacro{\offset}{0.2}
  \renewcommand*{\Jstring}{0/10,2/7,3/5,5/4,7/2,9/0}%

  \computeParameters[1]
  \drawAxes
  \drawFullJ{0.3}
  \drawStairs{1}
  \drawYSections[0]

  \coordinate (g) at ($\scalefactor*(3,5)$);
  \coordinate (h1)at ($\scalefactor*(0,10)$);
  \coordinate (h2)at ($\scalefactor*(2,7)$);
  \coordinate (h3)at ($\scalefactor*(9,0)$);
  \coordinate (h4)at ($\scalefactor*(7,2)$);
  \coordinate (h5)at ($\scalefactor*(5,4)$);

  \draw[black]
  let
  \p1 = ($\power*(g1)$),
  \p2 = ($\power*(g1) - (g1) +(g2)$),
  \p3 = ($(\p1)+ (h1)$),
  \p4 = ($(\p1)+ (h2)$),
  \p5 = ($(\p1)+ (g) $),
  \p6 = ($(\p2)+ (h2)$),
  \p7 = ($(\p2)+ (g) $),
  in
  (yend) -- (\p3) -- (\x4,\y3) -- (\p4) -- (\x5,\y4) -- (\p5) --(\x6,\y5) --(\p6) --(\x7,\y6)--(\p7);

  \draw[black]
  let
  \p1 = ($\power*(g1) -   (g1) +  (g2)$),
  \p2 = ($\power*(g1) - 2*(g1) +2*(g2)$),
  \p3 = ($(\p1)+ (h1)$),
  \p4 = ($(\p1)+ (h2)$),
  \p5 = ($(\p1)+ (g) $),
  \p6 = ($(\p2)+ (h2)$),
  \p7 = ($(\p2)+ (g) $),
  in
  (\p4) -- (\x5,\y4) -- (\p5) --(\x6,\y5) --(\p6) --(\x7,\y6)--(\p7);

  \draw[black]
  let
  \p1 = ($\power*(g2)$),
  \p2 = ($\power*(g2) - (g2) +(g1)$),
  \p3 = ($(\p1)+ (h3)$),
  \p4 = ($(\p1)+ (h4)$),
  \p5 = ($(\p1)+ (h5) $),
  \p6 = ($(\p1)+ (g)$),
  \p7 = ($(\p1)+ (h2) $),
  \p8 = ($(\p2)+ (g) $),
  in
 (xend) -- (\p3) -- (\x3,\y4) -- (\p4) -- (\x4,\y5) -- (\p5) --(\x5,\y6) --(\p6) --(\x6,\y7)--(\p7) --(\x7,\y8)--(\p8);

  \draw[black]
  let
  \p1 = ($\power*(g2) - 2*(g2)  +2*(g1)$),
  \p2 = ($\power*(g2) - (g2) +(g1)$),
  \p3 = ($(\p2)+ (g) $),
  \p4 = ($(\p2)+ (h2) $),
  \p5 = ($(\p1)+ (g)$),
  in
  (\p3) -- (\x3,\y4) -- (\p4) -- (\x4,\y5) -- (\p5);

   % M
  \draw[black]
  let
  \p1 = ($\power*(g1) - 3*(g1) +3*(g2)$),
  \p2 = ($\power*(g1) - 4*(g1) +4*(g2)$),
  \p3 = ($(\p1)+ (g) $),
  \p4 = ($(\p2)+ (h2)$),
  \p5 = ($(\p2)+ (g) $),
  in
  (\p5) -- (\x5,\y4) -- (\p4)-- (\x4,\y3) -- (\p3);
  \draw[black]
  let
  \p1 = ($\power*(g1) - 2*(g1) +2*(g2)$),
  \p2 = ($\power*(g1) - 3*(g1) +3*(g2)$),
  \p3 = ($(\p1)+ (g) $),
  \p4 = ($(\p2)+ (h2)$),
  \p5 = ($(\p2)+ (g) $),
  in
  (\p5) -- (\x5,\y4) -- (\p4)-- (\x4,\y3) -- (\p3);

  \node[fill=ProcessBlue, draw, inner sep=1pt, circle] at ($\power*(g1)              +(h1)$) {};
  \node[fill=ProcessBlue, draw, inner sep=1pt, circle] at ($\power*(g1)              +(h2)$) {};
  \node[fill=ProcessBlue, draw, inner sep=1pt, circle] at ($\power*(g1)              +(g)$)  {};
  \node[fill=ProcessBlue, draw, inner sep=1pt, circle] at ($\power*(g1)-  (g1)+  (g2)+(h2)$) {};
  \node[fill=ProcessBlue, draw, inner sep=1pt, circle] at ($\power*(g1)-  (g1)+  (g2)+(g)$)  {};
  \node[fill=ProcessBlue, draw, inner sep=1pt, circle] at ($\power*(g1)-2*(g1)+2*(g2)+(h2)$) {};
  \node[fill=ProcessBlue, draw, inner sep=1pt, circle] at ($\power*(g1)-2*(g1)+2*(g2)+(g)$)  {};

  \node[fill=BurntOrange, draw, inner sep=1.5pt] at ($\power*(g1)-5*(g1)+5*(g2)+(h2)$) {};
  \node[fill=BurntOrange, draw, inner sep=1.5pt] at ($\power*(g1)-5*(g1)+5*(g2)+(g)$)  {};
  \node[fill=BurntOrange, draw, inner sep=1.5pt] at ($\power*(g2)+(h2)$) {};
  \node[fill=BurntOrange, draw, inner sep=1.5pt] at ($\power*(g2)+(g)$)  {};
  \node[fill=BurntOrange, draw, inner sep=1.5pt] at ($\power*(g2)+(h3)$) {};
  \node[fill=BurntOrange, draw, inner sep=1.5pt] at ($\power*(g2)+(h4)$) {};
  \node[fill=BurntOrange, draw, inner sep=1.5pt] at ($\power*(g2)+(h5)$) {};

  \node[black, fill, inner sep=1pt, circle] (m1) at ($\power*(g1)-4*(g1)+4*(g2)+(h2)$) {};
  \node[black, fill, inner sep=1pt, circle] (m2) at ($\power*(g1)-4*(g1)+4*(g2)+(g)$)  {};
  \node[black, fill, inner sep=1pt, circle] (m3) at ($\power*(g1)-3*(g1)+3*(g2)+(h2)$) {};
  \node[black, fill, inner sep=1pt, circle] (m4) at ($\power*(g1)-3*(g1)+3*(g2)+(g)$)  {};
  \draw[red] (m1) circle (3pt);
  \draw[red] (m2) circle (3pt);
  \draw[red] (m3) circle (3pt);
  \draw[red] (m4) circle (3pt);

\end{tikzpicture}
%%% Local Variables:
%%% mode: latex
%%% TeX-master: "generators-powers-bivariate"
%%% End:
      \caption{$(x^u,y^v)^{r+3}J$}
    \end{subfigure}
    \caption{Visualization of the sets $\leftpart$, $\middlepart$, and
      $\rightpart$ in Example~\ref{example:1}.}
    \label{figure:example-thm-2}
  \end{figure}
\end{example}

\begin{corollary}\label{corollary:find-g}
  Let $u$, $v\in \N$, $J$ be an anchored monomial ideal,
  $r \ge \left\lceil \frac{\maxset{y}{J}}{v}\right\rceil$, and
  \begin{equation*}
    \middlepart = \left\{f \in \mingens[l]{(x^u,y^v)^{r+1}J}\longmid rv \le \deg_yf < (r+1)v \right\}\!.
  \end{equation*}
  Then $M \neq \emptyset$.
\end{corollary}
\begin{proof}
  Let $b\coloneqq \maxset{y}{J}$ and choose any natural number
  $\ell>\frac{b}{v}$. Assume that $M=\emptyset$. Then, by (\ref{pf:3})
  in the proof of Theorem~\ref{theorem:generators-one-segment}, we
  have
  \begin{equation*}
    (\mathcal{U}_j\cap\mathcal{L}_{j+1})\cap \mingens[l]{(x^u,y^v)^{r+\ell+1}J}=\emptyset
  \end{equation*}
  for all $r\le j\le r+\ell+1$.  Consequently, there are no minimal
  generators of $(x^u,y^v)^{r+\ell+1}J$ satisfying
  \begin{equation}\label{eq:degree-contradiction}
        rv\le\deg_yf<(r+\ell)v.
  \end{equation}
  However, by the choice of $\ell$, the element
  $f\coloneqq y^{rv}x^{(\ell+1)u}y^b\in(x^u,y^v)^{r+\ell+1}J$ fulfills
  these inequalities. Hence there must be a minimal generator
  $g\in\mingens[l]{(x^u,y^v)^{r+\ell+1}J}$ dividing $f$.  Then
  $\deg_x g\le (\ell+1)u$ and by \eqref{eq:degree-contradiction},
  $\deg_y g<rv$ must hold, and hence $g\notin(x^u,y^v)^{r+\ell+1}J$, a
  contradiction.
\end{proof}

Theorem~\ref{theorem:generators-one-segment} describes how the minimal
generators of $(x^u,y^v)^{r+1+\ell}J$ change in a predictable pattern
as $\ell$ increases.  In other words, the staircase of the ideal
$(x^u,y^v)^{r+1+\ell}J$ is formed by aligning the staircases of
certain repeatedly occurring subideals: from left to right, we begin
with the staircase of $(\leftpart)$, then we repeat the staircase of
$(\middlepart)$ for $\ell$ times, and finally add the
staircase of $(\rightpart)$.

We want to formalize this idea of ``connecting'' staircases.

\begin{definition}\label{definition:link}
  For $I$, $J\subseteq \K[x,y]$ monomial ideals, we define the
  \emph{link} $I\odot J$ with respect to $y$ as
  \begin{equation*}
    I\odot J\coloneqq \red{I}\cdot y^{\maxset{y}{J}} + \red{J}\cdot x^{\maxset{x}{I}},
  \end{equation*}
  see Figure~\ref{figure:link}. We write
  $I^{\odot\ell}\coloneqq\underbrace{I\odot I\odot\cdots\odot
    I}_{\ell}$. We call the monomial
  $x^{\maxset{x}{I}}y^{\maxset{y}{J}}$ the \emph{link point} of
  $I\odot J$.
\end{definition}

\begin{remark}\label{remark:link-x-y}
  As mentioned above in Remark~\ref{remark:why-y-sections}, we want to
  be able to reverse the roles of $x$ and $y$ throughout. However,
  this affects the order of the arguments of the link. Note that
  $I \xlink J = J\ylink I$, which is why we have to include the
  variable used for the partition in the notation of the link.
\end{remark}

\begin{remark}\label{remark:count-circ-generators}
  If $I$ and $J$ are monomial ideals with $\maxset{x}{I}=a$ and
  $\maxset{y}{J}=b$, then the minimal generators of $\red{I}\cdot y^b$
  and $\red{J}\cdot x^a$ only intersect in one element, namely the
  link point $x^ay^b$, cf.~Figure~\ref{figure:link} (indicated by an
  arrow).  Therefore,
  \begin{align*}
    \mingens{I\odot J}
    &= \mingens{\red{I}}\cdot y^{b}\cup\mingens{\red{J}}\cdot x^{a}\\
    &=\mingens{\red{I}}\cdot y^{b}\uplus \left(\mingens{\red{J}}\setminus \{y^{b}\}\right)\cdot x^{a}.
  \end{align*}
  This implies $\mu(I\odot J)= \mu(I) + \mu(J) -1$.
\end{remark}

\begin{remark}\label{remark:link-to-sum}
  Let $I$ be an anchored monomial ideal that can be written as the
  link of ideals, i.e.,
  \begin{equation*}
    I = J_0\odot\cdots\odot J_{k}
  \end{equation*}
  for some anchored monomial ideals $J_0,$ \dots, $J_{k}$. For
  $1\le i\le k$ we denote the link point between
  $(J_0\odot\cdots\odot J_{i-1})$ and $(J_{i}\odot\cdots\odot J_{k})$
  with $h_i$, and we set $h_0\coloneqq y^{\dist_{y}I}$ and
  $h_{k+1}\coloneqq x^{\dist_x I}$. Then it follows that
  \begin{equation*}
    I = \sum_{j=0}^{k}\gcd(h_{i},h_{i+1})J_{i}
  \end{equation*}
  and, in particular, $I:\gcd(h_{i},h_{i+1})=J_{i}$ for all
  $0\le i\le k$.
\end{remark}

With the necessary tools in place, we now return to the goal of
expressing $(x^u,y^v)^{r+1+\ell}J$ as the link of the ideals generated
by $\leftpart$, $\middlepart$, and $\rightpart$.  However, there are
``gaps'' between the staircases of these ideals, so we must first
expand them by suitable link points.

\begin{definition}\label{definition:stable-components}
  Let $u$, $v\in\N$, $J\subseteq \K[x,y]$ be an anchored monomial
  ideal, and $r \ge\left \lceil\frac{\maxset{y}{J}}{v}\right
  \rceil$. Further, we set $g\coloneqq x^{\alpha}y^{\beta}$ to be the
  minimal generator of $(x^u,y^v)^{r+1}J$ with
  \begin{equation*}
    \beta = \min\!\left\{\deg_y f\longmid f\in\mingens[l]{(x^u,y^v)^{r+1}J},
      \deg_yf\ge rv\right\}\!.
  \end{equation*}
  With $L$, $M$, and $R$ as in
  Theorem~\ref{theorem:generators-one-segment}, we define the
  \emph{$r$-segments} (with respect to~$y$) of $(x^u,y^v)J$ as
  \begin{align*}
    \Apart &\coloneqq(\leftpart):y^{\beta},\\
    \Hpart &\coloneqq\left(\middlepart  \cup \{x^{\alpha-u}y^{\beta+v}\}\right):x^{\alpha-u}y^{\beta},\text{ and}\\
    \Bpart &\coloneqq\left(\rightpart \cup\{g\}\right):x^{\alpha},
  \end{align*}

  For monomials $g$, $h\in\K[x,y]$ we define the $r$-segments
  of $(g,h)J$ as the $r$-segments of $\red{(g,h)}J$.
\end{definition}

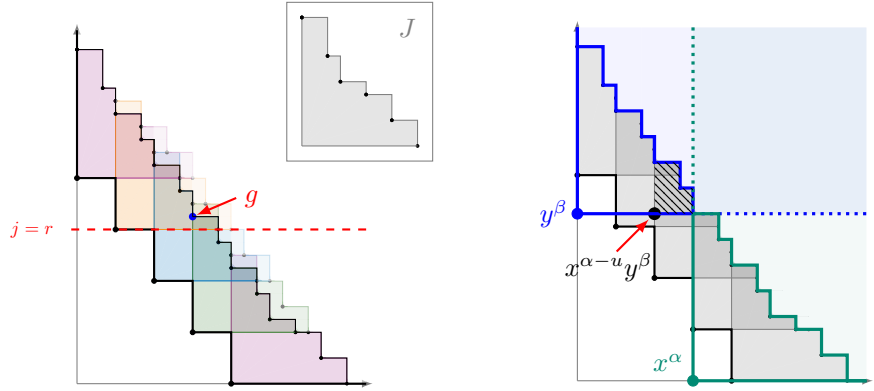
\begin{figure}[h]
  \centering
  \begin{subfigure}{0.55\textwidth}
    \begin{tikzpicture}
  \pgfmathsetmacro{\scalefactor}{0.17}
  \pgfmathsetmacro{\xa}{3}
  \pgfmathsetmacro{\yb}{4}
  \pgfmathsetmacro{\xshift}{0}
  \pgfmathsetmacro{\yshift}{0}
  \pgfmathsetmacro{\power}{4}
  \pgfmathsetmacro{\positionx}{-4.5}
  \pgfmathsetmacro{\positiony}{-1.5}
  \pgfmathsetmacro{\offset}{0.2}
  \def\Jarray{(0,10),(2,7),(3,5),(5,4),(7,2),(9,0)}
  \renewcommand*{\Jstring}{0/10,2/7,3/5,5/4,7/2,9/0}%

  \computeParameters[1]
  \drawAxes
  \drawFullJ[1]{0.5}
  \drawStairs{1}

  \node[blue, fill, circle, inner sep =1pt] (M1) at ($2*(g1)+2*(g2)+\scalefactor*(3,5)$) {};
  \node (g) at ($(M1)+(0.7,0.3)$) {};
  \draw[red, -latex, thick] (g) node [right, xshift=-4pt, yshift=-2pt] {$g$} -- (M1) ;

  \coordinate (g) at ($\scalefactor*(3,5)$);
  \coordinate (h1)at ($\scalefactor*(0,10)$);
  \coordinate (h2)at ($\scalefactor*(2,7)$);
  \coordinate (h3)at ($\scalefactor*(9,0)$);
  \coordinate (h4)at ($\scalefactor*(7,2)$);
  \coordinate (h5)at ($\scalefactor*(5,4)$);

  \draw[black]
  let
  \p1 = ($\power*(g1)$),
  \p2 = ($\power*(g1) - (g1) +(g2)$),
  \p3 = ($(\p1)+ (h1)$),
  \p4 = ($(\p1)+ (h2)$),
  \p5 = ($(\p1)+ (g) $),
  \p6 = ($(\p2)+ (h2)$),
  \p7 = ($(\p2)+ (g) $),
  in
  (yend) -- (\p3) -- (\x4,\y3) -- (\p4) -- (\x5,\y4) -- (\p5) --(\x6,\y5) --(\p6) --(\x7,\y6)--(\p7);

  \draw[black]
  let
  \p1 = ($\power*(g1) -   (g1) +  (g2)$),
  \p2 = ($\power*(g1) - 2*(g1) +2*(g2)$),
  \p3 = ($(\p1)+ (h1)$),
  \p4 = ($(\p1)+ (h2)$),
  \p5 = ($(\p1)+ (g) $),
  \p6 = ($(\p2)+ (h2)$),
  \p7 = ($(\p2)+ (g) $),
  in
  (\p4) -- (\x5,\y4) -- (\p5) --(\x6,\y5) --(\p6) --(\x7,\y6)--(\p7);

  \draw[black]
  let
  \p1 = ($\power*(g2)$),
  \p2 = ($\power*(g2) - (g2) +(g1)$),
  \p3 = ($(\p1)+ (h3)$),
  \p4 = ($(\p1)+ (h4)$),
  \p5 = ($(\p1)+ (h5) $),
  \p6 = ($(\p1)+ (g)$),
  \p7 = ($(\p1)+ (h2) $),
  \p8 = ($(\p2)+ (g) $),
  in
 (xend) -- (\p3) -- (\x3,\y4) -- (\p4) -- (\x4,\y5) -- (\p5) --(\x5,\y6) --(\p6) --(\x6,\y7)--(\p7) --(\x7,\y8)--(\p8);

  \draw[black]
  let
  \p1 = ($\power*(g2) - 2*(g2)  +2*(g1)$),
  \p2 = ($\power*(g2) - (g2) +(g1)$),
  \p3 = ($(\p2)+ (g) $),
  \p4 = ($(\p2)+ (h2) $),
  \p5 = ($(\p1)+ (g)$),
  in
  (\p3) -- (\x3,\y4) -- (\p4) -- (\x4,\y5) -- (\p5);

  \draw [red, thick, dashed]
  let
  \p1 = ($3*(g1)$),
  \p2 = (xend),
  in
  (\x2,\y1) -- (-0.15,\y1) node [left, xshift=-2pt, scale=0.7] {$j=r$} ;

\end{tikzpicture}
% Local Variables:
%%% mode: latex
%%% TeX-master: "generators-powers-bivariate"
%%% End:
  \end{subfigure}\hfill
  \begin{subfigure}{0.45\textwidth}
    \begin{tikzpicture}
  \pgfmathsetmacro{\scalefactor}{0.17}
  \pgfmathsetmacro{\xa}{3}
  \pgfmathsetmacro{\yb}{4}
  \pgfmathsetmacro{\xshift}{0}
  \pgfmathsetmacro{\yshift}{0}
  \pgfmathsetmacro{\power}{4}
  \pgfmathsetmacro{\positionx}{0}
  \pgfmathsetmacro{\positiony}{-10}
  \pgfmathsetmacro{\offset}{0.2}
  \renewcommand*{\Jstring}{0/10,2/7,3/5,5/4,7/2,9/0}%

  \computeParameters[1]
  \drawAxes
  \drawFullJ[2]{1}
  \drawStairs{1}

  \coordinate (g) at ($\scalefactor*(3,5)$);
  \coordinate (h1)at ($\scalefactor*(0,10)$);
  \coordinate (h2)at ($\scalefactor*(2,7)$);
  \coordinate (h3)at ($\scalefactor*(9,0)$);
  \coordinate (h4)at ($\scalefactor*(7,2)$);
  \coordinate (h5)at ($\scalefactor*(5,4)$);

  \draw[blue, fill]
  let
  \p1 = ($\power*(g1) - 2*(g1) +2*(g2)$),
  \p2 = ($(\p1)+ (g) $),
  in
  (0,\y2) circle (2pt) node[left] {$y^{\beta}$};

  \draw[PineGreen, fill]
  let
  \p1 = ($\power*(g1) - 2*(g1) +2*(g2)$),
  \p2 = ($(\p1)+ (g) $),
  in
  (\x2,0) circle (2pt) node[above left] {$x^{\alpha}$};

  \draw[black, thick, fill] ($\power*(g1) - 2*(g1) +(g2)+ (g)$) circle (2pt);
  \draw[red,   thick, -latex, shorten >=3pt]
  let
  \p1 = ($\power*(g1) - 2*(g1) +2*(g2)$),
  \p2 = ($(\p1)+ (g) -(g2)$),
  \p3 = ($(\p2) - (0.5,0.5)$),
  in
  (\p3)  node[black, below, xshift=-3pt, yshift=3pt] {$x^{\alpha-u}y^{\beta}$}  -- (\p2);

 %%%%%%%%%%%%%%%%%%%%%%%%%%%%%%%%%%%%%%%%%%
 % draw A
  \draw[blue, very thick]
  let
  \p1 = ($\power*(g1)$),
  \p2 = ($\power*(g1) - (g1) +(g2)$),
  \p3 = ($(\p1)+ (h1)$),
  \p4 = ($(\p1)+ (h2)$),
  \p5 = ($(\p1)+ (g) $),
  \p6 = ($(\p2)+ (h2)$),
  \p7 = ($(\p2)+ (g) $),
  in
  (yend) -- (\p3) -- (\x4,\y3) -- (\p4) -- (\x5,\y4) -- (\p5) --(\x6,\y5) --(\p6) --(\x7,\y6)--(\p7);

  \draw[blue, very thick]
  let
  \p1 = ($\power*(g1) -   (g1) +  (g2)$),
  \p2 = ($\power*(g1) - 2*(g1) +2*(g2)$),
  \p3 = ($(\p1)+ (h1)$),
  \p4 = ($(\p1)+ (h2)$),
  \p5 = ($(\p1)+ (g) $),
  \p6 = ($(\p2)+ (h2)$),
  \p7 = ($(\p2)+ (g) $),
  in
  (\p4) -- (\x5,\y4) -- (\p5) --(\x6,\y5) --(\p6) --(\x7,\y6)--(\p7);

  \draw[blue, dotted, very thick]
  let
  \p1 = ($\power*(g1) - 2*(g1) +2*(g2)$),
  \p2 = ($(\p1)+ (g) $),
  \p3 = (xend),
  in
  (\p2) -- (\x3,\y2);

  \draw[blue, very thick]
  let
  \p1 = ($\power*(g1) - 2*(g1) +2*(g2)$),
  \p2 = ($(\p1)+ (g) $),
  in
  (yend) -- (0,\y2) -- (\p2);

  \fill[blue,fill opacity=0.05]
  let
  \p1 = ($\power*(g1)$),
  \p2 = ($(\p1)+ (h1)$),
  \p3 = ($(\p1)+ (h2)$),
  \p4 = ($(\p1)+ (g)$),
  \p5 = ($\power*(g1) - (g1) +(g2)$),
  \p6 = ($(\p5)+ (h2)$),
  \p8 = (xend),
  \p9 = (yend),
  in
  (\p9) -- (\p2) -- (\x3,\y2) -- (\p3) --(\x4,\y3) -- (\p4) -- (\x6,\y4) -- (\x6,\y9);

  \fill[blue,fill opacity=0.05]
  let
  \p1 = ($\power*(g1) - (g1) +(g2)$),
  \p2 = ($\power*(g1) - 2*(g1) +2*(g2)$),
  \p3 = ($(\p1)+ (h1)$),
  \p4 = ($(\p1)+ (h2)$),
  \p5 = ($(\p1)+ (g) $),
  \p6 = ($(\p2)+ (h2)$),
  \p7 = ($(\p2)+ (g) $),
  \p8 = (xend),
  \p9 = (yend),
  in
   (\x4,\y9) -- (\p4) -- (\x5,\y4) -- (\p5) --(\x6,\y5) --(\p6) --(\x7,\y6)--(\p7) -- (\x8,\y7) --(\x8,\y9);

 %%%%%%%%%%%%%%%%%%%%%%%%%%%%%%%%%%%
 % draw B
  \draw[PineGreen, very thick]
  let
  \p1 = ($\power*(g2) - 2*(g2)  +2*(g1)$),
  \p2 = ($(\p1)+ (g) $),
  \p3 = (xend),
  in
  (\p2) -- (\x2,0) -- (xend);

  \draw[PineGreen, very thick]
  let
  \p1 = ($\power*(g2)$),
  \p2 = ($\power*(g2) - (g2) +(g1)$),
  \p3 = ($(\p1)+ (h3)$),
  \p4 = ($(\p1)+ (h4)$),
  \p5 = ($(\p1)+ (h5) $),
  \p6 = ($(\p1)+ (g)$),
  \p7 = ($(\p1)+ (h2) $),
  \p8 = ($(\p2)+ (g) $),
  in
 (xend) -- (\p3) -- (\x3,\y4) -- (\p4) -- (\x4,\y5) -- (\p5) --(\x5,\y6) --(\p6) --(\x6,\y7)--(\p7) --(\x7,\y8)--(\p8);

  \draw[PineGreen, very thick]
  let
  \p1 = ($\power*(g2) - 2*(g2)  +2*(g1)$),
  \p2 = ($\power*(g2) - (g2) +(g1)$),
  \p3 = ($(\p2)+ (g) $),
  \p4 = ($(\p2)+ (h2) $),
  \p5 = ($(\p1)+ (g)$),
  in
  (\p3) -- (\x3,\y4) -- (\p4) -- (\x4,\y5) -- (\p5);

  \draw[PineGreen, dotted, very thick]
  let
  \p1 = ($\power*(g2) - 2*(g2) +2*(g1)$),
  \p2 = ($(\p1)+ (g) $),
  \p3 = (yend),
  in
  (\x2,\y3) -- (\p2);

  \fill[PineGreen, fill opacity=0.05]
  let
  \p1 = ($\power*(g2)$),
  \p2 = ($\power*(g2) - (g2) +(g1)$),
  \p3 = ($(\p1)+ (h3)$),
  \p4 = ($(\p1)+ (h4)$),
  \p5 = ($(\p1)+ (h5) $),
  \p6 = ($(\p1)+ (g)$),
  \p7 = ($(\p1)+ (h2) $),
  \p8 = ($(\p2)+ (g) $),
  \p9 = (xend),
  in
 (\p9) -- (\p3) -- (\x3,\y4) -- (\p4) -- (\x4,\y5) -- (\p5) --(\x5,\y6) --(\p6) --(\x6,\y7)--(\p7) --(\x7,\y8)--(\p8) -- (\x9,\y8);

 \fill[PineGreen, opacity=0.05]
  let
  \p1 = ($\power*(g2) - 2*(g2)  +2*(g1)$),
  \p2 = ($\power*(g2) - (g2) +(g1)$),
  \p3 = ($(\p2)+ (g) $),
  \p4 = ($(\p2)+ (h2) $),
  \p5 = ($(\p1)+ (g)$),
  \p6 = (xend),
  \p7 = (yend),
  in
  (\x5,\y7) -- (\p5) -- (\x4,\y5) -- (\p4) -- (\x3,\y4) -- (\p3) -- (\x6,\y3) -- (\x6,\y7);

 %%%%%%%%%%%%%%%%%%%%%%%%%%%%%%%%%%%
 % draw H
  \fill[pattern=north west lines]
  let
  \p1 = ($\power*(g1) - (g1) +(g2)$),
  \p2 = ($\power*(g1) - 2*(g1) +2*(g2)$),
  \p3 = ($(\p1)+ (h1)$),
  \p4 = ($(\p1)+ (h2)$),
  \p5 = ($(\p1)+ (g) $),
  \p6 = ($(\p2)+ (h2)$),
  \p7 = ($(\p2)+ (g) $),
  in
  (\p5) -- (\x5, \y7) -- (\p7) -- (\x7, \y6) -- (\p6) -- (\x6,\y5) ;

\end{tikzpicture}

%%% Local Variables:
%%% mode: latex
%%% TeX-master: "generators-powers-bivariate"
%%% End:
  \end{subfigure}
  \caption{Left: Visualization of the point $g$ from
    Definition~\ref{definition:stable-components}, with $r=3$. Right:
    Above $y^{\beta}$ is the staircase of $y^{\beta}\Apart$ (in blue),
    and to the right of $x^{\alpha}$ is the staircase of
    $x^{\alpha}\Bpart$ (in green).  The striped area in the
    bottom-right corner of $y^{\beta}\Bpart$ is the staircase of
    $x^{\alpha-u}y^{\beta}\Hpart$. }
  \label{figure:ABH}
\end{figure}

\begin{remark}\label{remark:relation-ABH-LRM}
  With the notation of Definition~\ref{definition:stable-components},
  we have
  \begin{align*}
  \Apart &= \left((x^u,y^v)^{r+1}J\right):y^{\beta},\\
  \Hpart &= \left((x^u,y^v)^{r+1}J\right):x^{\alpha-u}y^{\beta},\text{ and}\\
  \Bpart &= \left((x^u,y^v)^{r+1}J\right):x^{\alpha}.
  \end{align*}
  In particular, $g=x^{\alpha}y^{\beta}$ is the link point of
  $\Apart\odot\Bpart=(x^u,y^v)^{r+1}J$, see
  Figure~\ref{figure:ABH}. Moreover,
  \begin{equation*}
    \leftpart   =\mingens{A}y^{\beta}, \quad
    \rightpart \uplus \{g\} = \mingens{\Bpart}x^{\alpha},
    \quad\text{and}\quad
    \middlepart \uplus \{x^{\alpha-u}y^{\beta+v}\}= \mingens{\Hpart}x^{\alpha-u}y^{\beta}.
  \end{equation*}
\end{remark}

\begin{corollary}\label{corollary:one-segment-link}
  Let $u$, $v\in\N$ and $J\subseteq\K[x,y]$ be an anchored monomial
  ideal, $r \ge \left\lceil \frac{\maxset{y}{J}}{v}\right\rceil$, and
  $\Apart$, $\Hpart$, $\Bpart$ the $r$-segments of $(x^u,y^v)J$.
  Then, for all $\ell\in \N_0$,
  \begin{equation*}
    (x^u,y^v)^{r+1+\ell}J = \Apart\odot \Hpart^{\odot\ell}\odot \Bpart.
  \end{equation*}
\end{corollary}
\begin{proof}
  It follows from Theorem~\ref{theorem:generators-one-segment} in
  combination with the last equalities in
  Remark~\ref{remark:relation-ABH-LRM} that
  \begin{align*}
    &\mingens[l]{(x^u,y^v)^{r+1+\ell}J}
      =y^{v\ell}\leftpart\uplus \biguplus_{j=1}^{\ell}x^{ju}y^{(\ell-j)v}\middlepart \uplus x^{u\ell}\rightpart\\
    &=y^{v\ell+\beta}\mingens{\Apart}\uplus \biguplus_{j=0}^{\ell-1}\left(x^{\alpha +ju}y^{\beta+(\ell-j-1)v}\mingens{\Hpart}\setminus\{x^{\alpha+ju}y^{\beta+(\ell-j)v}\}\right)\\
    &\white{=y^{v\ell+\beta}\mingens{\Apart}}\uplus x^{\alpha+u\ell}\mingens{\Bpart}\setminus\{x^{\alpha+u\ell}y^{\beta}\}.
  \end{align*}
  This is exactly the minimal generating set of
  \begin{equation*}
    \Apart\odot \Hpart^{\odot\ell}\odot \Bpart
    =\Apart y^{v\ell+\beta}+\sum_{j=0}^{\ell-1}\Hpart x^{\alpha+ju}y^{\beta+(\ell-1-j)v}+\Bpart x^{\alpha+\ell u},
  \end{equation*}
  cf.~Remark~\ref{remark:count-circ-generators}. Note that the
  generators $x^{\alpha+ju}y^{\beta+(\ell-j)v}$ for $0\le j\le \ell$
  are the link points in $\Apart\odot \Hpart^{\odot\ell}\odot \Bpart$.
\end{proof}

\begin{remark}
  Corollary~\ref{corollary:one-segment-link} confirms that
  $r$-segments behave as intended and rephrases
  Theorem~\ref{theorem:generators-one-segment} in the language of
  ideal links. Continuing Figures~\ref{figure:example-thm-2}
  and~\ref{figure:ABH}, we provide a visualization in
  Figure~\ref{figure:thm-3-one-segment}.
\end{remark}

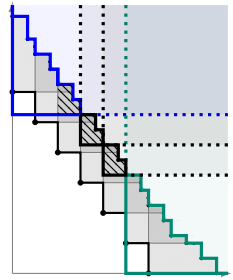
\begin{figure}[h]
  \centering
  \begin{minipage}[b]{0.55\textwidth}
    \caption[A visualization of
    Corollary~\ref{corollary:one-segment-link}.]{In
      $(x^u,y^v)^{r+3}J$, the staircase of $\Hpart$ (striped black) is
      repeated $\ell=2$ times in the middle. On the top left (in blue)
      is the staircase of $\Apart$ and on the bottom right (in green)
      is the staircase of $\Bpart$. Note that there is a (shifted)
      copy of $\Hpart$ included in~$\Apart$.}
    \label{figure:thm-3-one-segment}
  \end{minipage}\hspace*{-2em}
  \begin{minipage}[b]{0.4\textwidth}
    \centering \begin{tikzpicture}
  \pgfmathsetmacro{\scalefactor}{0.1}
  \pgfmathsetmacro{\xa}{3}
  \pgfmathsetmacro{\yb}{4}
  \pgfmathsetmacro{\xshift}{0}
  \pgfmathsetmacro{\yshift}{0}
  \pgfmathsetmacro{\power}{6}
  \pgfmathsetmacro{\positionx}{0}
  \pgfmathsetmacro{\positiony}{-10}
  \pgfmathsetmacro{\offset}{0.2}
  \renewcommand*{\Jstring}{0/10,2/7,3/5,5/4,7/2,9/0}%

  \computeParameters[1]
  \drawAxes
  \drawFullJ[2]{1}
  \drawStairs{1}

  \coordinate (g) at ($\scalefactor*(3,5)$);
  \coordinate (h1)at ($\scalefactor*(0,10)$);
  \coordinate (h2)at ($\scalefactor*(2,7)$);
  \coordinate (h3)at ($\scalefactor*(9,0)$);
  \coordinate (h4)at ($\scalefactor*(7,2)$);
  \coordinate (h5)at ($\scalefactor*(5,4)$);

 %%%%%%%%%%%%%%%%%%%%%%%%%%%%%%%%%%%%%%%%%%
 %%%%%%%%%%%%%%%%%%%%%%%%%%%%%%%%%%%%%%%%%%
 % draw A
  \draw[blue, very thick]
  let
  \p1 = ($\power*(g1)$),
  \p2 = ($\power*(g1) - (g1) +(g2)$),
  \p3 = ($(\p1)+ (h1)$),
  \p4 = ($(\p1)+ (h2)$),
  \p5 = ($(\p1)+ (g) $),
  \p6 = ($(\p2)+ (h2)$),
  \p7 = ($(\p2)+ (g) $),
  in
  (yend) -- (\p3) -- (\x4,\y3) -- (\p4) -- (\x5,\y4) -- (\p5) --(\x6,\y5) --(\p6) --(\x7,\y6)--(\p7);

  \draw[blue, very thick]
  let
  \p1 = ($\power*(g1) -   (g1) +  (g2)$),
  \p2 = ($\power*(g1) - 2*(g1) +2*(g2)$),
  \p3 = ($(\p1)+ (h1)$),
  \p4 = ($(\p1)+ (h2)$),
  \p5 = ($(\p1)+ (g) $),
  \p6 = ($(\p2)+ (h2)$),
  \p7 = ($(\p2)+ (g) $),
  in
  (\p4) -- (\x5,\y4) -- (\p5) --(\x6,\y5) --(\p6) --(\x7,\y6)--(\p7);

  \draw[blue, dotted, very thick]
  let
  \p1 = ($\power*(g1) - 2*(g1) +2*(g2)$),
  \p2 = ($(\p1)+ (g) $),
  \p3 = (xend),
  in
  (\p2) -- (\x3,\y2);

  \draw[blue, very thick]
  let
  \p1 = ($\power*(g1) - 2*(g1) +2*(g2)$),
  \p2 = ($(\p1)+ (g) $),
  in
  % (\x2,\y1) circle (3pt);
  (yend) -- (0,\y2) -- (\p2);

  \fill[blue,fill opacity=0.05]
  let
  \p1 = ($\power*(g1)$),
  \p2 = ($(\p1)+ (h1)$),
  \p3 = ($(\p1)+ (h2)$),
  \p4 = ($(\p1)+ (g)$),
  \p5 = ($\power*(g1) - (g1) +(g2)$),
  \p6 = ($(\p5)+ (h2)$),
  \p8 = (xend),
  \p9 = (yend),
  in
  (\p9) -- (\p2) -- (\x3,\y2) -- (\p3) --(\x4,\y3) -- (\p4) -- (\x6,\y4) -- (\x6,\y9);

  \fill[blue,fill opacity=0.05]
  let
  \p1 = ($\power*(g1) - (g1) +(g2)$),
  \p2 = ($\power*(g1) - 2*(g1) +2*(g2)$),
  \p3 = ($(\p1)+ (h1)$),
  \p4 = ($(\p1)+ (h2)$),
  \p5 = ($(\p1)+ (g) $),
  \p6 = ($(\p2)+ (h2)$),
  \p7 = ($(\p2)+ (g) $),
  \p8 = (xend),
  \p9 = (yend),
  in
  (\x4,\y9) -- (\p4) -- (\x5,\y4) -- (\p5) --(\x6,\y5) --(\p6) --(\x7,\y6)--(\p7) -- (\x8,\y7) --(\x8,\y9);

 %%%%%%%%%%%%%%%%%%%%%%%%%%%%%%%%%%%
 % draw B
  \draw[PineGreen, very thick]
  let
  \p1 = ($\power*(g2) - 2*(g2)  +2*(g1)$),
  \p2 = ($(\p1)+ (g) $),
  \p3 = (xend),
  in
  (\p2) -- (\x2,0) -- (xend);

  \draw[PineGreen, very thick]
  let
  \p1 = ($\power*(g2)$),
  \p2 = ($\power*(g2) - (g2) +(g1)$),
  \p3 = ($(\p1)+ (h3)$),
  \p4 = ($(\p1)+ (h4)$),
  \p5 = ($(\p1)+ (h5) $),
  \p6 = ($(\p1)+ (g)$),
  \p7 = ($(\p1)+ (h2) $),
  \p8 = ($(\p2)+ (g) $),
  in
 (xend) -- (\p3) -- (\x3,\y4) -- (\p4) -- (\x4,\y5) -- (\p5) --(\x5,\y6) --(\p6) --(\x6,\y7)--(\p7) --(\x7,\y8)--(\p8);

  \draw[PineGreen, very thick]
  let
  \p1 = ($\power*(g2) - 2*(g2)  +2*(g1)$),
  \p2 = ($\power*(g2) - (g2) +(g1)$),
  \p3 = ($(\p2)+ (g) $),
  \p4 = ($(\p2)+ (h2) $),
  \p5 = ($(\p1)+ (g)$),
  in
  (\p3) -- (\x3,\y4) -- (\p4) -- (\x4,\y5) -- (\p5);

  \draw[PineGreen, dotted, very thick]
  let
  \p1 = ($\power*(g2) - 2*(g2) +2*(g1)$),
  \p2 = ($(\p1)+ (g) $),
  \p3 = (yend),
  in
  (\x2,\y3) -- (\p2);

  \fill[PineGreen, fill opacity=0.05]
  let
  \p1 = ($\power*(g2)$),
  \p2 = ($\power*(g2) - (g2) +(g1)$),
  \p3 = ($(\p1)+ (h3)$),
  \p4 = ($(\p1)+ (h4)$),
  \p5 = ($(\p1)+ (h5) $),
  \p6 = ($(\p1)+ (g)$),
  \p7 = ($(\p1)+ (h2) $),
  \p8 = ($(\p2)+ (g) $),
  \p9 = (xend),
  in
 (\p9) -- (\p3) -- (\x3,\y4) -- (\p4) -- (\x4,\y5) -- (\p5) --(\x5,\y6) --(\p6) --(\x6,\y7)--(\p7) --(\x7,\y8)--(\p8) -- (\x9,\y8);

  \fill[PineGreen, opacity=0.05]
  let
  \p1 = ($\power*(g2) - 2*(g2)  +2*(g1)$),
  \p2 = ($\power*(g2) - (g2) +(g1)$),
  \p3 = ($(\p2)+ (g) $),
  \p4 = ($(\p2)+ (h2) $),
  \p5 = ($(\p1)+ (g)$),
  \p6 = (xend),
  \p7 = (yend),
  in
  (\x5,\y7) -- (\p5) -- (\x4,\y5) -- (\p4) -- (\x3,\y4) -- (\p3) -- (\x6,\y3) -- (\x6,\y7);

 %%%%%%%%%%%%%%%%%%%%%%%%%%%%%%%%%%%
 % draw H

  \fill[pattern=north west lines]
  let
  \p1 = ($\power*(g1) - (g1) + (g2)$),
  \p2 = ($\power*(g1) - 2*(g1) +2*(g2)$),
  \p3 = ($(\p1)+ (h1)$),
  \p4 = ($(\p1)+ (h2)$),
  \p5 = ($(\p1)+ (g) $),
  \p6 = ($(\p2)+ (h2)$),
  \p7 = ($(\p2)+ (g) $),
  in
  (\p5) -- (\x5, \y7) -- (\p7) -- (\x7, \y6) -- (\p6) -- (\x6,\y5) -- (\p5);

  %%%%%%%%%%%%%%%%%%%%%%%%%%%%%%%%%%%%%%%%%%%
  \draw[black, very thick, pattern=north west lines]
  let
  \p1 = ($\power*(g1) - 2*(g1) + 2*(g2)$),
  \p2 = ($\power*(g1) - 3*(g1) +3*(g2)$),
  \p3 = ($(\p1)+ (h1)$),
  \p4 = ($(\p1)+ (h2)$),
  \p5 = ($(\p1)+ (g) $),
  \p6 = ($(\p2)+ (h2)$),
  \p7 = ($(\p2)+ (g) $),
  in
  (\p5) -- (\x5, \y7) -- (\p7) -- (\x7, \y6) -- (\p6) -- (\x6,\y5) -- (\p5);

  \draw[black, dotted, very thick]
  let
  \p1 = ($\power*(g1) - 3*(g1) +3*(g2)$),
  \p2 = ($(\p1)+ (g) $),
  \p3 = (xend),
  in
  (\x3,\y2) -- (\p2);
  \draw[black, dotted, very thick]
  let
  \p1 = ($\power*(g1) - 2*(g1) + 2*(g2)$),
  \p2 = ($(\p1)+ (g) $),
  \p3 = (yend),
  in
  (\x2,\y3) -- (\p2);

  \fill[black, opacity=0.05]
  let
  \p1 = ($\power*(g1) - 2*(g1) + 2*(g2)$),
  \p2 = ($\power*(g1) - 3*(g1) +3*(g2)$),
  \p3 = ($(\p1)+ (h1)$),
  \p4 = ($(\p1)+ (h2)$),
  \p5 = ($(\p1)+ (g) $),
  \p6 = ($(\p2)+ (h2)$),
  \p7 = ($(\p2)+ (g) $),
  \p8 = (xend),
  \p9 = (yend),
  in
  (\x5,\y9) -- (\p5)  -- (\x6,\y5) -- (\p6) -- (\x7, \y6) -- (\p7) -- (\x8,\y7) -- (\x8,\y9);

  %%%%%%%%%%%%%%%%%%%%%%%%%%%%%%%%%%%%%%%%%%%
      \draw[black, very thick, pattern=north west lines]
  let
  \p1 = ($\power*(g1) - 3*(g1) + 3*(g2)$),
  \p2 = ($\power*(g1) - 4*(g1) + 4*(g2)$),
  \p3 = ($(\p1)+ (h1)$),
  \p4 = ($(\p1)+ (h2)$),
  \p5 = ($(\p1)+ (g) $),
  \p6 = ($(\p2)+ (h2)$),
  \p7 = ($(\p2)+ (g) $),
  in
  (\p5) -- (\x5, \y7) -- (\p7) -- (\x7, \y6) -- (\p6) -- (\x6,\y5) -- (\p5);

  \draw[black, dotted, very thick]
  let
  \p1 = ($\power*(g1) - 4*(g1) +4*(g2)$),
  \p2 = ($(\p1)+ (g) $),
  \p3 = (xend),
  in
  (\x3,\y2) -- (\p2);
  \draw[black, dotted, very thick]
  let
  \p1 = ($\power*(g1) - 3*(g1) + 3*(g2)$),
  \p2 = ($(\p1)+ (g) $),
  \p3 = (yend),
  in
  (\x2,\y3) -- (\p2);

  \fill[black, opacity=0.05]
  let
  \p1 = ($\power*(g1) - 3*(g1) + 3*(g2)$),
  \p2 = ($\power*(g1) - 4*(g1) + 4*(g2)$),
  \p3 = ($(\p1)+ (h1)$),
  \p4 = ($(\p1)+ (h2)$),
  \p5 = ($(\p1)+ (g) $),
  \p6 = ($(\p2)+ (h2)$),
  \p7 = ($(\p2)+ (g) $),
  \p8 = (xend),
  \p9 = (yend),
  in
  (\x5,\y9) -- (\p5)  -- (\x6,\y5) -- (\p6) -- (\x7, \y6) -- (\p7) -- (\x8,\y7) -- (\x8,\y9);

\end{tikzpicture}

%%% Local Variables:
%%% mode: latex
%%% TeX-master: "generators-powers-bivariate"
%%% End:

    \vspace*{0.5ex}
  \end{minipage}\hfill
\end{figure}

\begin{remark}\label{remark:size-alpha-beta}
  With the notation of Definition~\ref{definition:stable-components},
  \begin{equation*}
    \beta < (r+1)v \quad\text{and}\quad \alpha \le (r+1)u
  \end{equation*}
  holds.  The first inequality follows from
  Corollary~\ref{corollary:find-g}, while the second follows from
  Remark~\ref{remark:y-sections-and-x-degree}, since by definition,
  $g=x^{\alpha}y^{\beta}$ is a minimal generator of $(x^u,y^v)^{r+1}J$
  that is in $\mathcal{U}_r$. As noted in
  Remark~\ref{remark:y-sections-and-x-degree}, if we have equality
  $\alpha = (r+1)u$, then this implies $rv = \dist_y J$,
  $g=x^{(r+1)u}y^{\dist_y(J)}$, and $\beta = \dist_y(J)$.
\end{remark}

We are now set to prove the main result of this section.
\begin{theorem}\label{theorem:gluing-more-segments}
  Let $I\subseteq \K[x,y]$ be an anchored monomial ideal such that
  $\mingens{I}=P^{\ast}(I)=\{g_1, \dots, g_{k+1}\}$ and the $g_i$ are
  ordered in descending $y$-degree. Further, let $J\subseteq \K[x,y]$
  be an anchored monomial ideal, and for $1\le i\le k$, let
  $v_i\coloneqq \maxsetdeg{y}{g_i,g_{i+1}}$,
  $u_i\coloneqq \maxsetdeg{x}{g_i,g_{i+1}}$, and
  \begin{equation*}
  r\ge \left\lceil\max_{1\le i\le k}\!\left\{ \frac{\maxset{y}{J}}{v_i} \right\}\right\rceil.
  \end{equation*}
  Then, for all
  $\ell\ge 0$,
  \begin{equation*}
    \sum_{i=1}^{k}(g_i, g_{i+1})^{r+1+\ell}J
    = \Cpart_{0}\odot
    \bigodot_{i=1}^{k}\left(\Hpart_{i}^{\odot\ell} \odot \combpart{i} \right)\!,
  \end{equation*}
  where, for $1\le i<k$, $\Apart_{i}$, $\Hpart_{i}$, $\Bpart_{i}$ are
  the $r$-segments of $(g_i,g_{i+1})J$,
  \begin{equation*}
    \combpart{i} \coloneqq \Bpart_{i}\cdot y^{(r+1)v_{i+1} - \maxsetdeg{y}{\Bpart_{i+1}}} + \Apart_{i+1}\cdot x^{(r+1)u_{i}-\maxsetdeg{x}{\Apart_{i}}},
  \end{equation*}
  $\combpart{0}\coloneqq \Apart_1$, and $\combpart{k}\coloneqq \Bpart_k$.
\end{theorem}

\begin{figure}[h]
  \centering
    \begin{minipage}[h]{0.24\textwidth}
    \centering
    \fbox{
      \begin{minipage}[h]{0.8\textwidth}
        \begin{tikzpicture}[baseline=(current bounding box.north)]
  \pgfmathsetmacro{\firstxa}{2}
  \pgfmathsetmacro{\firstyb}{6}
  \pgfmathsetmacro{\secondxa}{3}
  \pgfmathsetmacro{\secondyb}{4}

  \pgfmathsetmacro{\scalefactor}{0.12}
  \pgfmathsetmacro{\positionx}{-10}
  \pgfmathsetmacro{\positiony}{-10}
  \pgfmathsetmacro{\offset}{0.2}
  \renewcommand*{\Jstring}{0/10,2/7,3/5,5/4,7/2,9/0}%

  \coordinate (g) at ($\scalefactor*(3,5)$);
  \coordinate (h1)at ($\scalefactor*(0,10)$);
  \coordinate (h2)at ($\scalefactor*(2,7)$);
  \coordinate (h3)at ($\scalefactor*(9,0)$);
  \coordinate (h4)at ($\scalefactor*(7,2)$);
  \coordinate (h5)at ($\scalefactor*(5,4)$);

  % % upper segment
  \pgfmathsetmacro{\xa}{\firstxa}
  \pgfmathsetmacro{\xshift}{0}
  \pgfmathsetmacro{\yb}{\firstyb}
  \pgfmathsetmacro{\yshift}{0}
  \pgfmathsetmacro{\yshift}{\secondyb}
  \pgfmathsetmacro{\power}{2}

  \computeParameters[1]
  \drawStairs{0}
  \drawFullJ[2]{0.5}

  \draw[dotted, thick] ($\power*(g1) +(g1) - (g2)$) -- ($\power*(g2)$);

  \coordinate (g2p) at ($\power*(g2)$);
  \coordinate (g2pu) at ($\power*(g2)-(g2)+(g1)$);
  \coordinate (g2puu) at ($\power*(g2)-2*(g2)+2*(g1)$);
  \coordinate (g2puuu) at ($\power*(g2)-3*(g2)+3*(g1)$);
  \coordinate (g2puuuu) at ($\power*(g2)-4*(g2)+4*(g1)$);

  \coordinate (go) at ($(g2puu)+(g)$);
%   draw B
  \draw[PineGreen, very thick]%,]
  let
  \p1 = (g2p),
  \p2 = (g2pu),
  \p3 = ($(\p1)+ (h3)$),
  \p4 = ($(\p1)+ (h4)$),
  \p5 = ($(\p1)+ (h5) $),
  \p6 = ($(\p1)+ (g)$),
  \p7 = ($(\p1)+ (h2) $),
  \p8 = ($(\p1)+ (h1) $),
  in
  (\p3) -- (\x3,\y4) -- (\p4) -- (\x4,\y5) -- (\p5)  --(\x5,\y6) --(\p6) --(\x6,\y7)--(\p7) --(\x7,\y8)--(\p8);

  \draw[PineGreen, very thick]
  let
  \p1 = (g2puu),
  \p2 = (g2pu),
  \p3 = ($(\p2)+ (h1) $),
  \p6 = (g2p),
  \p7 = ($(\p6)+ (h1) $),
  \p8 = ($(\p6)+ (h3) $),
  in
  (\p8)--(\x3, \y6) -- (\p3) -- (\x7,\y3) --(\p7);

  \node[circle, draw, fill=red, inner sep=1.5pt] at (g2p) {};

\end{tikzpicture}

%%% Local Variables:
%%% mode: latex
%%% TeX-master: "generators-powers-bivariate"
%%% End:
      \end{minipage}
    }
    \fbox{
      \begin{minipage}[h]{0.8\textwidth}
        \begin{tikzpicture}[baseline=(current bounding box.north)]
    \pgfmathsetmacro{\firstxa}{2}
    \pgfmathsetmacro{\firstyb}{6}
    \pgfmathsetmacro{\secondxa}{3}
    \pgfmathsetmacro{\secondyb}{4}

    \pgfmathsetmacro{\scalefactor}{0.12}
    \pgfmathsetmacro{\positionx}{-10}
    \pgfmathsetmacro{\positiony}{-10}
    \pgfmathsetmacro{\offset}{0.2}
    \def\Jarray{(0,10),(2,7),(3,5),(5,4),(7,2),(9,0)}
    \renewcommand*{\Jstring}{0/10,2/7,3/5,5/4,7/2,9/0}%

    \coordinate (g) at ($\scalefactor*(3,5)$);
    \coordinate (h1)at ($\scalefactor*(0,10)$);
    \coordinate (h2)at ($\scalefactor*(2,7)$);
    \coordinate (h3)at ($\scalefactor*(9,0)$);
    \coordinate (h4)at ($\scalefactor*(7,2)$);
    \coordinate (h5)at ($\scalefactor*(5,4)$);

    \pgfmathsetmacro{\power}{2};

    % %lower segment
    \pgfmathsetmacro{\xa}{\secondxa}
    \pgfmathsetmacro{\xshift}{\firstxa}
    \pgfmathsetmacro{\yb}{\secondyb}
    \pgfmathsetmacro{\yshift}{0}

    \computeParameters[1]
    \drawStairs{0}
    \drawFullJ[2]{0.5}
    \draw[dotted, thick] ($\power*(g1)$) -- ($\power*(g2) + (g2) - (g1)$);

    \coordinate (g2p) at ($\power*(g1)$);
    \coordinate (g2pb)  at ($\power*(g1)-(g1)+(g2)$);
    \coordinate (g2pbb) at ($\power*(g1)-2*(g1)+2*(g2)$);

    % % % A2
    \draw[blue, very thick,]
    let
    \p1 = (g2p),
    \p2 = (g2pb),
    \p3 = ($(\p1)+ (h1)$),
    \p4 = ($(\p1)+ (h2)$),
    \p5 = ($(\p1)+ (g) $),
    \p6 = ($(\p2)+ (h2)$),
    \p7 = ($(\p2)+ (g) $),
    in
    (\p3) -- (\x4,\y3) -- (\p4) -- (\x5,\y4) -- (\p5) --(\x6,\y5) --(\p6);

    \draw[blue, very thick]
    let
    \p1 = (g2pb),
    \p2 = (g2pbb),
    \p3 = ($(\p1)+ (h1)$),
    \p4 = ($(\p1)+ (h2)$),
    \p5 = ($(\p1)+ (g) $),
    \p6 = ($(\p2)+ (h2)$),
    \p7 = ($(\p2)+ (g) $),
    \p8 = (g2p),
    \p9 = ($(\p8)+ (h1)$),
    in
    (\p4) -- (\x5,\y4) -- (\p5) --(\x6,\y5) --(\p6) --(\x7,\y6)--(\p7) -- (\x8,\y7) -- (\p9);

    \node[circle, fill=red, draw, inner sep=1.5pt] at (g2p) {};

    % label={[text=black]left:$g_i^{r+1+\ell}$}
 \end{tikzpicture}

%%% Local Variables:
%%% mode: latex
%%% TeX-master: "generators-powers-bivariate"
%%% End:
      \end{minipage}
    }
    \end{minipage}
    \begin{minipage}[h]{0.55\textwidth}
     \centering
     \fbox{
       \begin{tikzpicture}[baseline=(current bounding box.north)]
   \pgfmathsetmacro{\firstxa}{2}
   \pgfmathsetmacro{\firstyb}{6}
   \pgfmathsetmacro{\secondxa}{3}
   \pgfmathsetmacro{\secondyb}{4}

   \pgfmathsetmacro{\scalefactor}{0.15}
   \pgfmathsetmacro{\positionx}{-10}
   \pgfmathsetmacro{\positiony}{-10}
   \pgfmathsetmacro{\offset}{0.2}
   \renewcommand*{\Jstring}{0/10,2/7,3/5,5/4,7/2,9/0}%

   \coordinate (g) at ($\scalefactor*(3,5)$);
   \coordinate (h1)at ($\scalefactor*(0,10)$);
   \coordinate (h2)at ($\scalefactor*(2,7)$);
   \coordinate (h3)at ($\scalefactor*(9,0)$);
   \coordinate (h4)at ($\scalefactor*(7,2)$);
   \coordinate (h5)at ($\scalefactor*(5,4)$);

   % % upper segment
   \pgfmathsetmacro{\xa}{\firstxa}
   \pgfmathsetmacro{\xshift}{0}
   \pgfmathsetmacro{\yb}{\firstyb}
   \pgfmathsetmacro{\yshift}{0}
   \pgfmathsetmacro{\yshift}{\secondyb}
   \pgfmathsetmacro{\power}{1}

   \computeParameters[1]
   \drawStairs{0}
   \draw[very thick, dotted]
   let
   \p1 = ($\power*(g2) -(g2)+(g1)$),
   \p2 = ($(\p1) + \scalefactor*(-\xa,\yb)$),
   in
   (\p2)-- (\p1);

   \coordinate (g2pu) at ($\power*(g2)-(g2)+(g1)$);
   \coordinate (g2puu) at ($\power*(g2)-2*(g2)+2*(g1)$);
   \coordinate (g2puuu) at ($\power*(g2)-3*(g2)+3*(g1)$);
   \coordinate (g2puuuu) at ($\power*(g2)-4*(g2)+4*(g1)$);

   \coordinate (go) at ($(g2puu)+(g)$);

   % %lower segment
   \pgfmathsetmacro{\xa}{\secondxa}
   \pgfmathsetmacro{\xshift}{\firstxa}
   \pgfmathsetmacro{\yb}{\secondyb}
   \pgfmathsetmacro{\yshift}{0}

   \computeParameters[1]
   \drawStairs{0}
   \draw[very thick, dotted]
   let
   \p1 = ($\power*(g1) -(g1)+(g2)$),
   \p2 = ($(\p1) + \scalefactor*(\xa,-\yb)$),
   in
   (\p2)-- (\p1);

   \coordinate (g2m) at ($(g1)$);
   \coordinate (g3m) at ($(g2)$);
   \coordinate (g2p) at ($\power*(g1)$);
   \coordinate (g3p) at ($\power*(g2)$);
   \coordinate (g2pb)  at ($\power*(g1)-(g1)+(g2)$);
   \coordinate (g2pbb) at ($\power*(g1)-2*(g1)+2*(g2)$);
   \coordinate (g2pbbb) at ($\power*(g1)-3*(g1)+3*(g2)$);

   \coordinate (gu) at ($(g2pbb)+(g)$);

   % % color lower C0
   \draw[blue, very thick,]
   let
   \p1 = (g2p),
   \p2 = (g2pb),
   \p3 = ($(\p1)+ (h1)$),
   \p4 = ($(\p1)+ (h2)$),
   \p5 = ($(\p1)+ (g) $),
   \p6 = ($(\p2)+ (h2)$),
   \p7 = ($(\p2)+ (g) $),
   in
   (\p3) -- (\x4,\y3) -- (\p4) -- (\x5,\y4) -- (\p5) --(\x6,\y5) --(\p6);

   \draw[blue, very thick]
   let
   \p1 = (g2pb),
   \p2 = (g2pbb),
   \p3 = ($(\p1)+ (h1)$),
   \p4 = ($(\p1)+ (h2)$),
   \p5 = ($(\p1)+ (g) $),
   \p6 = ($(\p2)+ (h2)$),
   \p7 = ($(\p2)+ (g) $),
   \p8 = (g2p),
   \p9 = ($(\p8)+ (h1)$),
   in
   (\p4) -- (\x5,\y4) -- (\p5) --(\x6,\y5) --(\p6) --(\x7,\y6)--(\p7) -- (\x8,\y7) -- (\p9);

   \fill[blue, very thick, fill,fill opacity=0.1]
   let
   \p1 = (g2p),
   \p2 = (g2pb),
   \p3 = ($(\p1)+ (h1)$),
   \p4 = ($(\p1)+ (h2)$),
   \p5 = ($(\p1)+ (g) $),
   \p6 = ($(\p2)+ (h2)$),
   \p7 = ($(\p2)+ (g) $),
   in
   (\p3) -- (\x4,\y3) -- (\p4) -- (\x5,\y4) -- (\p5) --(\x6,\y5) --(\p6) -- (\x3,\y6);

   \fill[blue, very thick, fill,  fill opacity=0.1]
   let
   \p1 = (g2pb),
   \p2 = (g2pbb),
   \p3 = ($(\p1)+ (h1)$),
   \p4 = ($(\p1)+ (h2)$),
   \p5 = ($(\p1)+ (g) $),
   \p6 = ($(\p2)+ (h2)$),
   \p7 = ($(\p2)+ (g) $),
   \p8 = (g2p),
   \p9 = ($(\p8)+ (h1)$),
   in
   (\p4) -- (\x5,\y4) -- (\p5) --(\x6,\y5) --(\p6) --(\x7,\y6)--(\p7) -- (\x8,\y7) -- (\p9) -- (\x9,\y4);

   % % color upper Ck
   \draw[PineGreen, very thick]
   let
   \p1 = (g2p),
   \p2 = (g2pu),
   \p3 = ($(\p1)+ (h3)$),
   \p4 = ($(\p1)+ (h4)$),
   \p5 = ($(\p1)+ (h5) $),
   \p6 = ($(\p1)+ (g)$),
   \p7 = ($(\p1)+ (h2) $),
   \p8 = ($(\p1)+ (h1) $),
   in
   (\p3) -- (\x3,\y4) -- (\p4) -- (\x4,\y5) -- (\p5)  --(\x5,\y6) --(\p6) --(\x6,\y7)--(\p7) --(\x7,\y8)--(\p8);

   \draw[PineGreen, very thick]
   let
   \p1 = (g2puu),
   \p2 = (g2pu),
   \p3 = ($(\p2)+ (h1) $),
   \p6 = (g2p),
   \p7 = ($(\p6)+ (h1) $),
   \p8 = ($(\p6)+ (h3) $),
   in
   (\p8)--(\x3, \y6) -- (\p3) -- (\x7,\y3) --(\p7);

   \fill[PineGreen,  pattern color=PineGreen, pattern=dots ]%,]
   let
   \p1 = (g2p),
   \p2 = (g2pu),
   \p3 = ($(\p1)+ (h3)$),
   \p4 = ($(\p1)+ (h4)$),
   \p5 = ($(\p1)+ (h5) $),
   \p6 = ($(\p1)+ (g)$),
   \p7 = ($(\p1)+ (h2) $),
   \p8 = ($(\p1)+ (h1) $),
   in
   (\p3) -- (\x3,\y4) -- (\p4) -- (\x4,\y5) -- (\p5)  --(\x5,\y6) --(\p6) --(\x6,\y7)--(\p7) --(\x7,\y8)--(\p8) -- (\x8,\y3);

   \fill[PineGreen, pattern color=PineGreen, pattern=dots]
   let
   \p1 = (g2puu),
   \p2 = (g2pu),
   \p3 = ($(\p2)+ (h1) $),
   \p6 = (g2p),
   \p7 = ($(\p6)+ (h1) $),
   \p8 = ($(\p6)+ (h3) $),
   in
   (\p8)--(\x3, \y6) -- (\p3) -- (\x7,\y3) --(\p7) -- (\x7,\y8);

  % % mark x- and y-shift
  \draw [decorate, blue, thick, decoration = {brace}]
  let
  \p1 = (0.3,0),
  \p2 = ($(g2p) + (\p1)$),
  \p3 = ($(g2pbb) + (g) + (\p1)$),
  in
  (\x3,\y2) -- node[black,right, scale=0.7, xshift=2pt] { $(r+1)v_{i+1}-\maxsetdeg{y}{\Bpart_{i+1}}$}  (\p3);

   \draw [PineGreen, thick, dashed]
   let
   \p1 = (g2p),
   \p2 = (xend),
   \p3 = ($(g2pbb)+(g)$),
   \p4 = ($(g2pu) + (h1)$),
   in
   (\x4,\y1)  -- (\x4,\y3);

   \draw [PineGreen, thick, dashed]
   let
   \p1 = (g2p),
   \p2 = (xend),
   \p4 = ($(g2p) + (h3)$),
   in
   (\p4)  -- (\x2,\y4);

  \draw [decorate, PineGreen, thick, decoration = {brace, mirror}]
  let
  \p1 = (0,-0.1),
  \p2 = ($(g2p) + (\p1)$),
  \p3 = ($(g2pbb) + (g) + (\p1)$),
  \p4 = ($(g2pu) + (h1) + (\p1)$),
  in
  (\x4,\y3) -- node[black, below, scale=0.7, xshift=-15pt, yshift=-2pt] { $(r+1)u_{i}-\maxsetdeg{x}{\Apart_{i}}$}  (\x2,\y3);

  % % mark g_i power
  \node[circle, draw, fill=red, inner sep=2pt] at (g2p) {};
  \draw[-latex, shorten >=4pt, red, thick]
  ($(g2p) - (0.8,0.4)$) node[xshift=6pt,,label={[scale=0.7,text=black]left:$g_{i+1}^{r+1+\ell}$}]{} -- (g2p);

\end{tikzpicture}

%%% Local Variables:
%%% mode: latex
%%% TeX-master: "generators-powers-bivariate"
%%% End:
     }
  \end{minipage}
  \caption{Left: The lower section of $(g_{i},g_{i+1})^{r+1+\ell}J$
    and the upper section of $(g_{i+1},g_{i+2})^{r+1+\ell}J$, where
    $g_{i+1}^{r+1+\ell}$ appears as the lowest and uppermost (red)
    dot, respectively. Right: $\Bpart_{i}$ (dotted green) and
    $\Apart_{i+1}$ (shaded blue) overlap at $g_{i+1}^{r+1+\ell}$
    resulting in a new staircase, namely that of $\combpart{i}$. Note
    the required shifts in $x$- and $y$-direction to align
    $\Bpart_{i}$ and $\Apart_{i+1}$ before summing them up.}
  \label{figure:concatenation-component}
\end{figure}
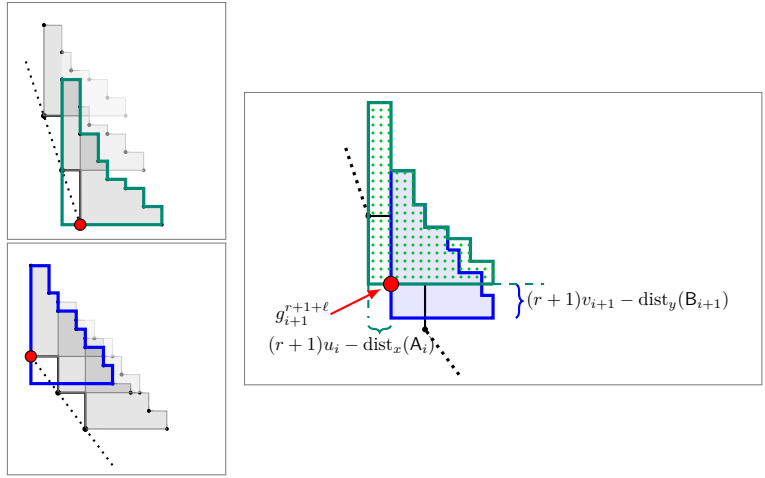

\begin{proof}
  Throughout, we use the notation
  $\alpha_i \coloneqq \maxsetdeg{x}{\Apart_i}$ and
  $\beta_i \coloneqq \maxsetdeg{y}{\Bpart_i}$. By
  Remark~\ref{remark:size-alpha-beta},
  \begin{equation}\label{eq:concatenation-shifts}
    (r+1)u_{i}-\alpha_{i} \ge 0 \quad\text{and}\quad (r+1)v_{i+1}-\beta_{i+1} > 0.
  \end{equation}
  We argue that for $1\le i\le k-1$
  \begin{equation}\label{eq:dxC}
    \maxsetdeg{x}{\combpart{i}}
    =  \alpha_{i+1}  + u_{i}(r+1) - \alpha_i
  \end{equation}
  and
  \begin{equation}\label{eq:dyC}
    \maxsetdeg{y}{\combpart{i}}
    =  \beta_{i}   + v_{i+1}(r+1) -  \beta_{i+1}.
  \end{equation}

  \eqref{eq:dxC}: Both ideals $\Apart_{i+1}$ and $\Bpart_{i}$ are
  anchored. Note that $\Apart_{i+1}\cdot x^{(r+1)u_{i}-\alpha_{i}}$
  contains a minimal generator of $y$-degree equal $0$ whereas all
  elements in $\Bpart_{i}\cdot y^{(r+1)v_{i+1} - \beta_{i+1}}$ have
  positive $y$-degree. It follows that
  \begin{equation*}
    \maxsetdeg{x}{\combpart{i}}
    = \maxsetdeg{x}{\Apart_{i+1}\cdot x^{(r+1)u_{i}-\alpha_{i}} }
    =  \alpha_{i+1}  + (r+1)u_{i}-\alpha_{i}.
  \end{equation*}

  \eqref{eq:dyC}: If $(r+1)u_{i}-\alpha_{i} > 0$, then we can argue
  analogously to the above.  If $(r+1)u_{i}-\alpha_{i} = 0$, then
  Remark~\ref{remark:size-alpha-beta} implies that
  $rv_i = \maxset{y}{J}$ and
  $\beta_i = \maxsetdeg{y}{\Bpart_i} = \maxset{y}{J}$.  Since
  $\dist_y\Cpart_i$ is given by the maximum of the $y$-degrees of the
  two summands of $\Cpart_i$, i.e.,
  \begin{align*}
    \dist_y\Cpart_i = \max\{\beta_{i} +(r+1)v_{i+1} - \beta_{i+1}, \dist_y\Apart_{i+1}\},
  \end{align*}
  and, by Remark~\ref{remark:relation-ABH-LRM},
  $\dist_y\Apart_{i+1}=(r+1)v_{i+1}+ \maxset{y}{J}-\beta_{i+1}$, the
  assertion follows.

  We turn our attention to the assertion of the theorem and proceed by
  induction on $k$. The base case $k=1$ is
  Corollary~\ref{corollary:one-segment-link}.

  Now let $k>1$.  By the assumption that $(g_1,\dots, g_{k+1})$ is
  anchored, we have $\gcd(g_1,\dots, g_k)=y^{\deg_yg_k}=y^{v_k}$.
  Using this, we apply the induction hypothesis to the anchored ideal
  $(g_1/y^{v_k}, \dots, g_k/y^{v_k})$ to conclude that
  \begin{equation*}
    \sum_{i=1}^{k-1}\left((g_i,g_{i+1})^{r+1+\ell}J\right)
    = \left(L  \odot \Bpart_{k-1}\right)y^{(r+1+\ell)v_k}
  \end{equation*}
  where
  \begin{equation*}
    L \coloneqq \Cpart_{0}\odot
    \bigodot_{i=1}^{k-2}\left(\Hpart_{i}^{\odot\ell} \odot \combpart{i} \right)
    \odot \Hpart_{k-1}^{\odot\ell}.
  \end{equation*}
  Similarly, setting $u =\deg_x g_k$, we know that
  \begin{equation*}
    (g_k,g_{k+1})^{r+1+\ell}J = \left(\Apart_k\odot R\right)\cdot x^{u(r+1+\ell)}
    \quad\text{with}\quad R \coloneqq \Hpart_k^{\odot \ell}\odot \Cpart_k.
  \end{equation*}

  Therefore,
  \begin{equation}\label{proof:glueing-1}
    \sum_{i=1}^{k}\left((g_i, g_{i+1})^{r+1+\ell}J\right)
    = (L \odot \Bpart_{k-1})y^{v_k(r+1+\ell)} + \left(\Apart_k\odot R\right)\cdot x^{u(r+1+\ell)},
  \end{equation}
  which is equal to
  \begin{equation}\label{proof:glueing-2}
    \left(Ly^{\beta_{k-1}} +  \Bpart_{k-1}x^{\maxset{x}{L}}  \right)\cdot y^{v_k(r+1+\ell)}
    + \left(\Apart_ky^{\maxset{y}{R}} + Rx^{\alpha_k}\right)\cdot x^{u(r+1+\ell)}.
  \end{equation}

  Before we continue manipulating~\eqref{proof:glueing-2}, we verify a
  few handy equations. Note that $\maxsetdeg{x}{\Hpart_i} = u_i$ and
  $\maxsetdeg{y}{\Hpart_i} = v_i$ for $1\le i \le k$, and recall that
  $u = \deg_xg_k = \sum_{i=1}^{k-1}u_i$.  Thus, using~\eqref{eq:dxC},
  \begin{align*}
    \maxset{x}{L}
    &= \alpha_1 + \ell\sum_{i=1}^{k-1}u_i + \sum_{i=1}^{k-2}\maxsetdeg{x}{\combpart{i}}
      = \alpha_{k-1} + \ell u + (r+1)\sum_{i=1}^{k-2}u_i \\
    &= \alpha_{k-1} + (r+1+\ell) u - (r+1)u_{k-1}.
  \end{align*}
  It follows that
  \begin{equation}
    \label{eq:1}
    \maxsetdeg{x}{L\odot \combpart{k-1}}
    = \maxset{x}{L} + \maxsetdeg{x}{\combpart{k-1}}
    = \alpha_{k} + (\ell+r+1) u
  \end{equation}
  and
  \begin{equation}\label{eq:2}
    u(r+1+\ell)-\maxset{x}{L}
    = (r+1)u_{k-1}-\alpha_{k-1} \ge 0,
  \end{equation}
  where the last inequality comes
  from~\eqref{eq:concatenation-shifts}. Moreover, since
  $\maxset{y}{R} = \beta_k + \ell v_k$ we have
  (using~\eqref{eq:concatenation-shifts} again)
  \begin{equation}\label{eq:3}
    v_k(r+1+\ell)-\maxset{y}{R} =  (r+1)v_k - \beta_k >0
  \end{equation}
  and, using~\eqref{eq:dyC},
  \begin{equation}\label{eq:4}
    \maxsetdeg{y}{\combpart{k-1}\odot R}
    =  \maxsetdeg{y}{\combpart{k-1}} +  \maxset{y}{R}
    = \beta_{k-1} + (\ell+r+1) v_k.
  \end{equation}
  Using~\eqref{eq:2} and~\eqref{eq:3}, the middle two summands
  of~\eqref{proof:glueing-2} can be rearranged to
  \begin{equation*}
    \left(\Bpart_{k-1} y^{(r+1)v_k - \beta_k}
      + \Apart_kx^{(r+1)u_{k-1}-\alpha_{k-1}}\right)x^{\maxset{x}{L}}y^{\maxset{y}{R}}.
  \end{equation*}
  Thus, using~\eqref{eq:1} and \eqref{eq:4}, we conclude
  that~\eqref{proof:glueing-2} equals
  \begin{equation*}
    Ly^{\maxsetdeg{y}{\combpart{k-1}\odot R}}
    + \combpart{k-1}x^{\maxset{x}{L}}y^{\maxset{y}{R}}
    +  Rx^{\maxsetdeg{x}{L\odot \combpart{k-1}}}
    = L \odot \combpart{k-1} \odot R
  \end{equation*}
  which completes the proof.
\end{proof}

\begin{corollary}
  With the assumptions and notation from
  Theorem~\ref{theorem:gluing-more-segments}, we have
  \begin{align*}
    \mu\!\left(\sum_{i=1}^{k}(g_i, g_{i+1})^{r+\ell+1}J\right)
    =1+ \sum_{i=0}^{k}(\mu(\combpart{i})-1)
    + \ell\sum_{i=1}^{k}\left(\mu(\Hpart_{i})-1\right)\!.
  \end{align*}
\end{corollary}
\begin{proof}
  This is an immediate consequence of
  Theorem~\ref{theorem:gluing-more-segments} in combination with
  Remark~\ref{remark:count-circ-generators}.
\end{proof}

\begin{remark}\label{remark:Ci}
  We argue that it is \textbf{not} required to know the $r$-segments
  of the individual summands $(g_i, g_{i+1})^{r+1}J$ to compute
  $\combpart{i}$ and $\Hpart_i$ in
  Theorem~\ref{theorem:gluing-more-segments}.
  \begin{enumerate}
  \item For $\combpart{i}$: The link points $h_1$, $\dots$, $h_k$ of
    \begin{equation*}
      S\coloneqq \sum_{i=1}^{k}(g_i, g_{i+1})^{r+1}J =
      \bigodot_{i=0}^{k}\combpart{i}
    \end{equation*}
    can be determined directly from $\mingens{S}$ by their $y$-degree,
    that is,
    \begin{equation*}
      \deg_yh_i=\min\!\left\{\deg_y f\longmid f\in\mingens[l]{S}, \deg_yf\ge rv_i+(r+1)\deg_yg_{i+1}\right\}\!,
    \end{equation*}
    for $1\le i\le k$.
    Hence, we use Remark~\ref{remark:link-to-sum}, set
    $h_0\coloneqq g_1^{r+1}$ and $h_{k+1}\coloneqq g_{k+1}^{r+1}$, and
    conclude that
    \begin{align*}
      \Cpart_i
      &= S:\gcd(h_i, h_{i+1}) \\
      &=\red{\left(g\in\mingens{S}\mid \deg_yh_{i+1}\le\deg_yg\le\deg_yh_i\right)}.
    \end{align*}

  \item For $\Hpart_{i}$: By Remark~\ref{remark:relation-ABH-LRM} and
    Theorem~\ref{theorem:generators-one-segment}, we have
    \begin{equation*}
      \Hpart_i = \left( (g_i, g_{i+1})^{r+1} J\right): \frac{h_i}{x^{u_i}}.
    \end{equation*}
    We can immediately conclude that
    \begin{equation*}
    	S:\frac{h_1}{x^{u_1}}=\left( (g_1, g_{2})^{r+1} J\right): \frac{h_1}{x^{u_1}}=\Hpart_1.
    \end{equation*}
    Further, we know by Remark~\ref{remark:size-alpha-beta}, that
    $\frac{h_i}{x^{u_i}}\mingens{\Hpart_i}$ is contained in
    \begin{equation*}
      T \coloneqq \{g\in \K[x,y] \mid \deg_y h_{i} \le \deg_y g \le \deg_y h_{i} + v_i \}.
    \end{equation*}
    A straight-forward verification shows that for $2\le i\le k$,
    \begin{equation*}
      \gcd(h_{i-1}, h_i)\Bpart_{i-1}\cdot y^{(r+1)v_{i}-\maxsetdeg{y}{\Bpart_{i}}} \cap T
      =
      g_{i}^{r+1}J \cap T.
    \end{equation*}
    This further implies that
    $S\cap T = (g_{i}, g_{i+1})^{r+1}J \cap T$ and thus,
    \begin{equation*}
      \Hpart_i = S: \frac{h_i}{x^{u_i}}.
    \end{equation*}
\end{enumerate}
\end{remark}

\newcommand{\rx}{r_{x}(P,D)}
\newcommand{\ry}{r_{y}(P,D)}
\newcommand{\rdot}{r_{\xy}(P,D)}
\newcommand{\hx}[1]{h^x_{#1}}
\newcommand{\hy}[1]{h^y_{#1}}
\newcommand{\hdot}[1]{h^{(y)}_{#1}}

\section{Minimal generating sets of powers}\label{section:min-gens}
We now apply the preceding results to describe the minimal generators
of large powers of an ideal~$I$.

\begin{convention}\label{convention:section5}
  Throughout, let $I$ be an anchored monomial ideal and
  $P=\{g_1,\dots, g_{k+1}\}$ such that
  $P(I)\subseteq P \subseteq P^{\ast}(I)$ and the $g_i$ are ordered in
  descending $y$-degree.  Further, with $\delta_P$ as in
  Notation~\ref{notation:delta}, and $d_P$ as in
  Notation~\ref{notation:dI}, let
  $D\ge D_P\coloneqq (\mu(I)-|P|)\cdot \delta_P + |P|\cdot d_P$, set
  \begin{equation*}
    \ry \coloneqq \left\lceil D\cdot \max_{1\le i\le k}
      \frac{\maxset{y}{I}}{\maxsetdeg{y}{g_i,g_{i+1}}}\right\rceil,
  \end{equation*}
  and let $s\ge D+\ry+1$.
\end{convention}

\begin{remark}
  We assume that $I$ is anchored for better readability. Note that
  $I^n=\gcd(I)^n\cdot \red{I}^n$ holds for all $n\ge 1$.
\end{remark}

We first use Theorem~\ref{theorem:I^D+l} to establish that
for all $\ell\ge 0$
\begin{equation*}
  I^{D+\ell}=\sum_{i=1}^{k}(g_i,g_{i+1})^{\ell}I^D
\end{equation*}
holds.  Note that the weakly persistent generators of the ideal
generated by $P$ are precisely the elements of $P$.  Further, we
observe that $\maxsetdeg{y}{I^D}=D\maxset{y}{I}$ and hence
\begin{equation*}
  r \coloneqq s-D-1\ge \ry= \left\lceil\max_{1\le i\le k}\frac{\maxsetdeg{y}{I^D}}{\maxsetdeg{y}{g_{i},g_{i+1}}}\right\rceil\!.
\end{equation*}
Therefore, the conditions for
Theorem~\ref{theorem:gluing-more-segments} with $J=I^D$ are satisfied
and we obtain that for all $\ell\ge0$
\begin{align*}
  I^{s+\ell} = \sum_{i=1}^{k}(g_i,g_{i+1})^{r+1+\ell} I^D = \Cpart_{0}\odot
  \bigodot_{i=1}^{k}\left(\Hpart_{i}{}^{\odot\ell} \odot \Cpart_i \right)\!,
\end{align*}
where $\Hpart_i$ and $\Cpart_i$ are as in
Theorem~\ref{theorem:gluing-more-segments}.

\begin{definition}\label{definition:s-stable-components}
  We call the sets $\Hpart_i$ and $\Cpart_i$ the \emph{$(s,y)$-stable
    components} of $I$ with respect to $P$ and $D$. Note that these
  components can be determined directly from $I^s$ without having to
  compute the individual $r$-segments, see Remark~\ref{remark:Ci} with
  $J=I^D$.
\end{definition}
The preceding arguments complete the proof of the main theorem of this
section, which we state here.
\begin{theorem}\label{theorem:I-with-k-persistent-stabilizes}
  We use the notation of Convention~\ref{convention:section5} and let
  $(\Cpart_{i})_{i=0}^k$ and $(\Hpart_i)_{i=1}^k$ be the
  $(s,y)$-stable components of $I$ with respect to~$P$ and~$D$.  Then
  for all $\ell \ge0$,
  \begin{equation*}
    I^{s+\ell}
    = \Cpart_0\ylink\bigylink_{i=1}^{k}\left(\Hpart_{i}^{\ylink\ell} \ylink \combpart{i} \right).
  \end{equation*}
  \qed
\end{theorem}
\begin{remark}
  The link points $h_1^{(y)}$, \dots, $h_k^{(y)}$ of
  $I^s=\bigylink_{i=0}^k\Cpart_i$ can be determined from the minimal
  generators of $I^s$, see Remark~\ref{remark:Ci}.  They are ordered
  in descending $y$-degree, see
  Figure~\ref{figure:s-stable-components}, and fulfill
  \begin{equation*}
    \deg_{y}g_{i+1}^s \le \deg_{y}h_i^{(y)} \le \deg_{y}g_{i}^{r+1}.
  \end{equation*}
  The lower bound is due to
  $r\ge \ry\ge\frac{D\dist_{y}I}{\dist_{y}(g_i,g_{i+1})}$ and the
  upper bound follows from Corollary~\ref{corollary:find-g}.
\end{remark}

\begin{remark}
  At the beginning of Section~\ref{section:staircase-factors} we chose
  to partition elements of ideals with regular staircase factors based
  on their $y$-degrees (Lemma~\ref{lemma:y-sections},
  Remark~\ref{remark:why-y-sections}). Consequently, definitions and
  results so far are phrased with assertions about the $y$-degree.
  Throughout we have notationally pointed out the dependence of this
  choice where needed, that is, in the link $\ylink$, in $\ry$, in the
  $(s,y)$-stable components, in the link points $h_i^{(y)}$, and in
  the order of the elements of $P$.  Switching the roles of the
  variables does not affect the values of $\delta_P$ and $d_P$.
\end{remark}

For the next few results (until including
Remark~\ref{remark:end-of-xy}) we take this option of a variable
switch into account.

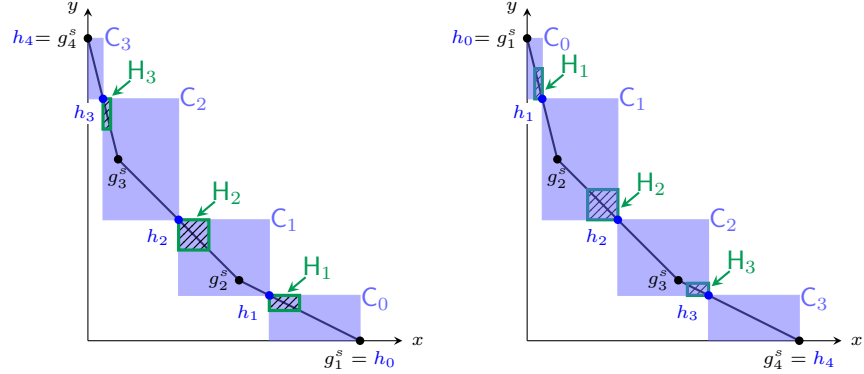
\begin{figure}
  \begin{minipage}{0.45\textwidth}
	\centering
	\begin{tikzpicture}[scale=0.8]
	% Axes
	\draw[-{stealth}] (0,0) -- (5.2,0) node[anchor=west] {\scriptsize$x$};
	\draw[-{stealth}] (0,0) -- (0,5.5) node[anchor=east] {\scriptsize$y$};
	
	% Points
	\coordinate (g4) at (0,5);
	%  \coordinate (H1) at (1,3.2);
	\coordinate (h1) at (3,0.75);
	%  \coordinate (h0) at (0,2.33);
	\coordinate (g3) at (0.5,3);
	\coordinate (g2) at (2.5,1);
	\coordinate (g1) at (4.5,0);
	\coordinate (h2) at (1.5,2);
	\coordinate (h3) at (0.25,4);
	
	% Vertical and horizontal dashed lines
	%  \draw[dashed, blue!60] (h1) -- (h0);
	%  \draw[dashed, blue!60] (h2) -- (h3);

	% Lines connecting points
	\draw[thick] (g1) -- (g2) -- (g3) -- (g4);
	
	%H_i's
	\coordinate (H1) at (3.8,1.2);
	\node[ForestGreen] at (H1) {$\Hpart_1$};
	\draw[ForestGreen,thick, -{stealth}] ($(H1.south) + (-0.2,-0.15)$) -> ($(H1.south east)+(-0.5,-0.4)$);
	
	\coordinate (H2) at (2.3,2.45);
	\node[ForestGreen] at (H2) {$\Hpart_2$};
	\draw[ForestGreen,thick, -{stealth}] ($(H2.south) + (-0.2,-0.15)$) -> ($(H2.south east)+(-0.5,-0.4)$);
	
	\coordinate (H3) at (0.9,4.45);
	\node[ForestGreen] at (H3) {$\Hpart_3$};
	\draw[ForestGreen,thick, -{stealth}] ($(H3.south) + (-0.2,-0.15)$) -> ($(H3.south east)+(-0.5,-0.4)$);

	% Region Labels
	
	\filldraw[draw = blue!60, fill = blue!60, opacity = 0.5] (h1) -- (4.5,0.75) -- (g1) -- (3,0) -- (h1);
	\filldraw[draw = ForestGreen, pattern=north east lines, very thick] (h1) -- (3.5,0.75) -- (3.5,0.5) -- (3,0.5) -- (h1);
	\node[blue!60] at (4.75,0.7) {$\Cpart_0$};

	\filldraw[draw = blue!60, fill = blue!60, opacity = 0.5] (h2) -- (3,2) -- (h1) -- (1.5,0.75) -- (h2);
	\filldraw[draw = ForestGreen, pattern=north east lines, very thick] (h2) -- (2,2) -- (2,1.5) -- (1.5,1.5) -- (h2);
	\node[blue!60] at (3.25,2) {$\Cpart_1$};

	\filldraw[draw = blue!60, fill = blue!60, opacity = 0.5] (h3) -- (1.5,4) -- (h2) -- (0.25,2) -- (h3);
	\filldraw[draw = ForestGreen, pattern=north east lines, very thick] (h3) -- (0.375,4) -- (0.375,3.5) -- (0.25,3.5) -- (h3);
	\node[blue!60] at (1.75,4) {$\Cpart_2$};
	
	\filldraw[draw = blue!60, fill = blue!60, opacity = 0.5] (g4) -- (0.25,5) -- (h3) -- (0,4) -- (g4);
	\node[blue!60] at (0.5,5) {$\Cpart_3$};
	
	% Points with labels
	
	\fill (g1) circle (2pt) node[below] {\scriptsize$g_1^s=\color{blue}{h_0}$};
	\fill[blue] (h1) circle (2pt) node[below left] {\scriptsize$h_1$};
	%  \fill[blue!60] (h0) circle (2pt) node[left] {$h_0$};
	\fill[blue] (h2) circle (2pt) node[below left] {\scriptsize$h_2$};
	\fill[blue] (h3) circle (2pt) node[below left = 2pt, rectangle, fill = white, inner sep=0.9pt] {\scriptsize$h_3$};
	%  \fill (H1) circle (2pt) node[above right] {$H_1$};
	\fill (g2) circle (2pt) node[left] {\scriptsize$g_2^s$};
	\fill (g3) circle (2pt) node[below] {\scriptsize$g_3^s$};
	\fill (g4) circle (2pt) node[left] {\scriptsize$\color{blue}{h_4}\color{black}{=g_4^s}$};
	\end{tikzpicture}
\end{minipage}
\begin{minipage}{0.45\textwidth}
	\centering
	\begin{tikzpicture}[scale=0.8]
	% Axes
	\draw[-{stealth}] (0,0) -- (5.2,0) node[anchor=west] {\scriptsize$x$};
	\draw[-{stealth}] (0,0) -- (0,5.5) node[anchor=east] {\scriptsize$y$};
	
	% Points
	\coordinate (g4) at (0,5);
	%  \coordinate (H1) at (1,3.2);
	\coordinate (h1) at (3,0.75);
	%  \coordinate (h0) at (0,2.33);
	\coordinate (g3) at (0.5,3);
	\coordinate (g2) at (2.5,1);
	\coordinate (g1) at (4.5,0);
	\coordinate (h2) at (1.5,2);
	\coordinate (h3) at (0.25,4);
	
	% Vertical and horizontal dashed lines
	%  \draw[dashed, blue!60] (h1) -- (h0);
	%  \draw[dashed, blue!60] (h2) -- (h3);

	% Lines connecting points
	\draw[thick] (g1) -- (g2) -- (g3) -- (g4);
	
	%H_i's
	\coordinate (H3) at (3.55,1.3);
	\node[ForestGreen] at (H3) {$\Hpart_3$};
	\draw[ForestGreen,thick, -{stealth}] ($(H3.south) + (-0.2,-0.15)$) -> ($(H3.south east)+(-0.5,-0.4)$);
	
	\coordinate (H2) at (2.05,2.6);
	\node[ForestGreen] at (H2) {$\Hpart_2$};
	\draw[ForestGreen,thick, -{stealth}] ($(H2.south) + (-0.2,-0.15)$) -> ($(H2.south east)+(-0.5,-0.4)$);
	
	\coordinate (H1) at (0.8,4.55);
	\node[ForestGreen] at (H1) {$\Hpart_1$};
	\draw[ForestGreen,thick, -{stealth}] ($(H1.south) + (-0.2,-0.15)$) -> ($(H1.south east)+(-0.5,-0.4)$);

	% Region Labels
	
	\filldraw[draw = blue!60, fill = blue!60, opacity = 0.5] (h1) -- (4.5,0.75) -- (g1) -- (3,0) -- (h1);
	\filldraw[draw = ForestGreen, pattern=north east lines, very thick] (h1) -- (3,0.95) -- (2.65,0.95) -- (2.65,0.75) -- (h1);
	\node[blue!60] at (4.75,0.7) {$\Cpart_3$};
	
	\filldraw[draw = blue!60, fill = blue!60, opacity = 0.5] (h2) -- (3,2) -- (h1) -- (1.5,0.75) -- (h2);
	\filldraw[draw = ForestGreen, pattern=north east lines, very thick] (h2) -- (1.5,2.5) -- (1,2.5) -- (1,2) -- (h2);
	\node[blue!60] at (3.25,2) {$\Cpart_2$};
	
	\filldraw[draw = blue!60, fill = blue!60, opacity = 0.5] (h3) -- (1.5,4) -- (h2) -- (0.25,2) -- (h3);
	\filldraw[draw = ForestGreen, pattern=north east lines, very thick] (h3) -- (0.25,4.5) -- (0.125,4.5) -- (0.125,4) -- (h3);
	\node[blue!60] at (1.75,4) {$\Cpart_1$};
	
	\filldraw[draw = blue!60, fill = blue!60, opacity = 0.5] (g4) -- (0.25,5) -- (h3) -- (0,4) -- (g4);
	\node[blue!60] at (0.5,5) {$\Cpart_0$};
	
	% Points with labels
	
	\fill (g1) circle (2pt) node[below] {\scriptsize$g_4^s=\color{blue}{h_4}$};
	\fill[blue] (h1) circle (2pt) node[below left] {\scriptsize$h_3$};
	%  \fill[blue!60] (h0) circle (2pt) node[left] {$h_0$};
	\fill[blue] (h2) circle (2pt) node[below left] {\scriptsize$h_2$};
	\fill[blue] (h3) circle (2pt) node[below left = 2pt, rectangle, fill = white, inner sep=0.9pt] {\scriptsize$h_1$};
	%  \fill (H1) circle (2pt) node[above right] {$H_1$};
	\fill (g2) circle (2pt) node[left] {\scriptsize$g_3^s$};
	\fill (g3) circle (2pt) node[below] {\scriptsize$g_2^s$};
	\fill (g4) circle (2pt) node[left] {\scriptsize$\color{blue}{h_0}\color{black}{=g_1^s}$};
	\end{tikzpicture}
\end{minipage}
  \caption{Switching the roles of $x$ and $y$: A simplified
    visualisation of the $(s,\xy)$-stable components
    $(\Cpart_{i})_{i=0}^k$ and $(\Hpart_i)_{i=1}^k$ ($\xy=x$ on the
    left and $\xy=y$ on the right) of an ideal with $k=|P|-1=3$.}
  \label{figure:s-stable-components}
\end{figure}

\begin{notation}
  With the notation of
  Theorem~\ref{theorem:I-with-k-persistent-stabilizes}, let
  \begin{equation*}
    r(P,D) \coloneqq \min\{\rx, \ry\}.
  \end{equation*}
\end{notation}

\begin{remark}
  Theorem~\ref{theorem:I-with-k-persistent-stabilizes} states that all
  information about large powers of $I$ is encoded in $I^{s}$ for any
  $s\ge D+ r(P,D) + 1$.
\end{remark}

\begin{corollary}\label{corollary:runtime}
  For $s \ge D_P + r(P,D_P) + 1$ and $\ell \ge 0$, the computation of
  $\mingens{I^{s+\ell}}$ from $\mingens{I^{s}}$ takes
  $\mathcal{O}(\ell)$ additions of (monomial) exponents.
\end{corollary}
\begin{remark}
  Theorem~\ref{theorem:I-with-k-persistent-stabilizes} leads to
  significantly faster computations of large powers of $I$. For a
  runtime comparison we refer to Section~\ref{sec:runtime}.
\end{remark}

Recall that the Hilbert function of the fibre ring of a monomial
ideal---which counts the number of generators of its powers---eventually
becomes a polynomial function
(cf.~\cite[Theorem~6.1.3]{Herzog-Hibi:2011:monideals}).
Theorem~\ref{theorem:I-with-k-persistent-stabilizes} yields an explicit description of this
polynomial.
\begin{corollary}\label{corollary:mu}
  Let $I\subseteq\K[x,y]$ be a monomial ideal and
  $s\ge D+ r(P,D) + 1$.  Then for all $\ell\ge 0$,
  \begin{equation*}
    \mu(I^{s+\ell})=\mu(I^{s}) + \ell\left(\mu(I^{s+1})-\mu(I^{s})\right)\!.
  \end{equation*}
\end{corollary}

\begin{remark}\label{remark:bound-patter-stability-number}
  For $\xy \in\{x,y\}$ and $|P|>2$, we have
  $r_{\bullet}(P,D_P) \le D_P\dist_{\xy}(I)$ and
  $D_P\le \mu(I)(\dist_{\bullet} I-1)$. Thus
  \begin{align*}
    D_P+r_{\bullet}(P,D_P)+1\le \mu(I)\Big(\dist_{\xy}(I)^2-1\Big)+1.
  \end{align*}
  If $|P|=2$, then $r(P,D_P) = D_P$ and hence
  \begin{align*}
    D_P+ r(P,D_P) + 1 = 2(\mu(I) - 2)\Big(\min\{\dist_xI, \dist_yI\} - 1\Big) + 1.
  \end{align*}
\end{remark}

\begin{corollary}\label{corollary:I-with-2-persistent-stabilizes}
  Let $I\subseteq\K[x,y]$ be a monomial ideal with $P(I)=\{x^a,y^b\}$,
  $d\coloneqq\min\{a,b\}$, and $s\ge 2(\mu(I)-2)(d-1)+1$.  Then for
  all $\ell\ge 0$
  \begin{equation*}
    I^{s+\ell}=\Cpart_0\dotlink \mathsf{H}_1^{\dotlink\ell}\dotlink \mathsf{C}_1,
  \end{equation*}
  where $\Cpart_0$, $\Cpart_1$, $\Hpart_1$ are the $(s, \xy)$-stable
  components of $I$ with respect to $P=P(I)$ and $D=D_P$.
\end{corollary}
\begin{proof}
  This is the special case of
  Theorem~\ref{theorem:I-with-k-persistent-stabilizes} with $k=1$.
\end{proof}
\begin{remark}\label{remark:end-of-xy}
  If $P=\{x^a,y^b\}$, then the $(s-D_P-1)$-segments of
  $(x^a,y^b)I^{D_P}$ (w.r.t.~$y$, see
  Definition~\ref{definition:stable-components}) coincide with the
  $(s,y)$-stable components of $I$ with respect to $P$, i.e.,
  \begin{equation*}
    \Apart = \Cpart_0,\quad \Hpart = \Hpart_1,\quad\text{ and }\quad \Bpart = \Cpart_1.
  \end{equation*}
\end{remark}

We now summarize how the minimal generators of $I^s$, $I^{s+1}$, and
$I^{s+2}$ can be explicitly described from the $(s,\xy)$-stable
components.

For better readability we switch back to stating the results with
respect to the variable $y$.

\begin{corollary}\label{corollary:mingens-shift}
  We use the notation of Convention~\ref{convention:section5}.
  Further, let $(\Cpart_i)_{i=0}^k$, $(\Hpart_i)_{i=1}^k$ be the
  $(s,y)$-stable components of $I$ with respect to $P$ and $D$, and
  $h_0$, \dots, $h_k$ be the link points of
  $I^s=\Cpart_{0}\ylink\cdots\ylink\Cpart_k$.  We set
  $q_i\coloneqq \gcd(h_i, h_{i+1})$ for $0\le i \le k$ and
  $\hat{q}_i \coloneqq \frac{h_i}{x^{\dist_x(g_i,g_{i+1})}}$.  Then
  \begin{equation*}
    \mingens[l]{I^{s}} = \mingens{q_0\Cpart_0} \uplus \biguplus_{i=1}^k \mingensast{q_i\Cpart_{i}},
  \end{equation*}
  \begin{equation*}
    \mingens[l]{I^{s+1}} = g_1\cdot \mingens{q_0\Cpart_0}
    \uplus \biguplus_{i=1}^k g_{i+1}\Big(\mingensast{q_i\Cpart_{i}}
    \uplus \mingensast{\hat{q}_i\Hpart_i}\Big),
  \end{equation*}
  and
  \begin{equation*}
    \mingens[l]{I^{s+2}} = g_1^2\mingens{q_0\Cpart_0}
    \uplus \biguplus_{i=1}^k g_{i+1}^2\Big(\mingensast{q_i\Cpart_{i}}
    \uplus \mingensast{\hat{q}_i\Hpart_i}\Big)
    \uplus \biguplus_{i=1}^k g_ig_{i+1}\mingensast{\hat{q}_i\Hpart_i},
  \end{equation*}
  where $\mingensast{J}$ denotes the minimal generating set of a
  monomial ideal $J$ excluding the minimal generator of maximal
  $y$-degree.
\end{corollary}
\begin{proof}
  The first assertion follows from
  Theorem~\ref{theorem:I-with-k-persistent-stabilizes} in combination
  with Remark~\ref{remark:link-to-sum}. Note that
  \begin{equation*}
    \dist_x(\Cpart_0 \odot \dots \odot \Cpart_{i-1}) = \deg_xh_i \quad \text{ and }\quad \dist_y(\Cpart_{i} \odot \dots \odot \Cpart_{k}) = \deg_yh_i.
  \end{equation*}

  Moreover, again by
  Theorem~\ref{theorem:I-with-k-persistent-stabilizes}, we have
  \begin{equation}\label{eq:Is+1}
    I^{s+1} = \Cpart_0\odot\Hpart_1 \odot \cdots\odot \Hpart_k \odot\Cpart_k.
  \end{equation}

  Let $u_i \coloneqq \dist_x\Hpart_i = \dist_x(g_i,g_{i+1})$ and
  $v_i \coloneqq \dist_y\Hpart_i = \dist_y(g_i,g_{i+1})$ for
  $1\le i \le k$.  With this notation,
  $\deg_xg_i= \sum_{j=1}^{i-1}{u_{j}}$ and
  $\deg_yg_i = \sum_{j=i}^{k}{v_j}$.

  For $1\le i \le k$, let $w_i$ be the link point of
  $(\Cpart_0\odot\Hpart_1 \odot \cdots\odot \Hpart_{i-1}
  \odot\Cpart_{i-1})$ and
  $(\Hpart_{i} \odot\Cpart_{i}\odot \cdots\odot \Hpart_{k}
  \odot\Cpart_{k})$. Then
  \begin{align*}
    \deg_xw_i &= \deg_xh_i + \sum_{j=1}^{i-1} u_j\quad \text{and}\\
    \deg_yw_i &= \deg_yh_i + \sum_{j=i}^{k} v_j,
  \end{align*}
  which implies $w_i = h_ig_{i}$.

  Similarly, for the link point $m_i$ of
  $(\Cpart_0\odot\Hpart_1 \odot \cdots \odot \Cpart_{i-1} \odot
  \Hpart_i)$ and
  $(\Cpart_{i}\odot \cdots\odot \Hpart_{k} \odot\Cpart_{k})$ we
  conclude that $m_i = h_{i}g_{i+1}$.

  The monomials $w_1$, $m_1$, $w_2$, \dots, $w_{k}$, $m_{k}$ are the
  link points of the link~\eqref{eq:Is+1} from left to right. With
  $m_0 \coloneqq h_0g_{1}$, $w_{k+1} \coloneqq h_{k+1}g_{k+1} $ and
  Remark~\ref{remark:link-to-sum}, we have
  \begin{equation*}
    I^{s+1} =
    \sum_{i=0}^k\gcd(m_i, w_{i+1})\Cpart_i + \sum_{i=1}^k\gcd(w_{i},m_{i})\Hpart_i.
  \end{equation*}

  A straightforward verification shows that for $1\le i \le k$
  \begin{equation*}
    \gcd(m_i, w_{i+1}) = g_{i+1}q_{i} \quad\text{ and }\quad\gcd(m_{i},w_{i}) = \gcd(g_i,g_{i+1})h_{i} = g_{i+1}\hat{q}_i.
  \end{equation*}

  Finally, it follows from the definition of the link that the
  individual summands intersect exactly at the linking points which
  are excluded in $\mingensast{\Hpart_i}$ and $\mingensast{\Cpart_i}$,
  making them pairwise disjoint. This proves the second assertion.

  For the third assertion, we proceed analogously. We determine the
  link points of
  \begin{align*}
    &(\Cpart_0\odot\Hpart_1^{\odot 2} \odot \cdots\odot \Hpart_{i-1}^{\odot 2} \odot\Cpart_{i-1}) \quad\quad\,\,\,\,\,\text{ and }
      (\Hpart_{i}^{\odot 2} \odot\Cpart_{i}\odot \cdots\odot \Hpart_{k}^{\odot 2} \odot\Cpart_{k})  \\
    &(\Cpart_0\odot\Hpart_1^{\odot 2} \odot \cdots\odot \Hpart_{i-1}^{\odot 2} \odot\Cpart_{i-1} \odot \Hpart_i) \,\,\,\,\,\text{ and } \,\,\,\,
      (\Hpart_{i} \odot\Cpart_{i}\odot \cdots\odot \Hpart_{k}^{\odot 2} \odot\Cpart_{k})  \\
    &(\Cpart_0\odot\Hpart_1^{\odot 2} \odot \cdots\odot \Hpart_{i-1}^{\odot 2} \odot\Cpart_{i-1}\odot\Hpart_{i}^{\odot 2}) \text{ and } \quad\quad\,\,\,\,\,
      (\Cpart_{i}\odot \cdots\odot \Hpart_{k}^{\odot 2} \odot\Cpart_{k})
  \end{align*}
  which are $w_i \coloneqq h_ig_i^2$, $u_i\coloneqq h_ig_ig_{i+1}$,
  and $m_i \coloneqq h_ig_{i+1}^2$.  With $m_0 \coloneqq h_0g_1^2$ and
  $w_{k+1} \coloneqq h_{k+1}g_{k+1}^2$, we have
  \begin{equation*}
    I^{s+2} =
    \sum_{i=0}^k\gcd(m_i, w_{i+1})\Cpart_i + \sum_{i=1}^k\gcd(w_{i},u_{i})\Hpart_i
    + \sum_{i=1}^k\gcd(u_{i},m_{i})\Hpart_i.
  \end{equation*}
  Since $\gcd(m_i, w_{i+1}) = g_{i+1}^2q_i$,
  $\gcd(u_i, m_{i}) = g_{i+1}^2\widehat{q_i}$, and
  $\gcd(u_i, w_{i}) = g_ig_{i+1}\widehat{q_i}$, the third assertion
  follows.
\end{proof}

\begin{remark}
  It follows from the corollary and its proof that
  $\Hpart_i = I^{s+1}:g_{i+1}\hat{q}_i$.
\end{remark}

\begin{remark}
  Corollary~\ref{corollary:mingens-shift} shows how the minimal
  generators of $I^{s+1}$ and $I^{s+2}$ are computed from
  $\mingens{I^s}$. For $I^{s+\ell}$ one multiplies the $q_i\Cpart_i$
  with $g_{i+1}^{\ell}$, and $\widehat q_i \Hpart_i$ with
  $g_i^{\ell_1}g_{i+1}^{\ell_2}$ for all $\ell_1$, $\ell_2$ with
  $\ell_1 + \ell_2 = \ell$.
\end{remark}

\begin{corollary}\label{corollary:bivariate:shift}
  With the assumptions and notation of
  Corollary~\ref{corollary:mingens-shift},
  \begin{equation*}
    \mingens{I^{s+1}}=\biguplus_{f\in \mingens{I^s}} f\cdot G_f,
  \end{equation*}
  where, with the notation $v_i\coloneqq \dist_{y}(g_{i},g_{i+1})$,
  \begin{equation*}
    G_f =
    \begin{cases}
      \{g_{i}\} &  \text{if }1\le i\le k\text{ with }\deg_{y}h_i + v_i < \deg_{y}f < \deg_{y}h_{i-1},\\
      \{g_{i},g_{i+1}\} & \text{if }1\le i\le k\text{ with }\deg_{y}h_{i} \le \deg_{y}f  \le \deg_{y}h_{i} + v_i,\\
      \{g_{k+1}\} & \text{if } \deg_{y}f< \deg_{y}h_{k}.
    \end{cases}
  \end{equation*}
\end{corollary}

\begin{proof}
  Observe that for $1\le i \le k+1$,
  \begin{align*}
    S_i
    &\coloneqq    \left\{f\in \mingens{I^s} \longmid
      \deg_{y}h_i +v_i < \deg_{y}f < \deg_{y}h_{i-1} \right\} \\
    &\subseteq    \left\{f\in \mingens{I^s} \longmid
      \deg_{y}h_i  \le \deg_{y}f\le \deg_{y}h_{i-1} \right\} \\
    &= \mingens{q_{i-1}\Cpart_{i-1}}.
  \end{align*}
  Moreover, for $1\le i \le k$,
  \begin{align*}
    T_i &\coloneqq \left\{f\in \mingens{I^s} \longmid
          \deg_{y}h_i \le \deg_{y}f \le \deg_{y}h_i + v_i \right\} = \mingens{\hat{q}_{i}\Hpart_{i}}\subseteq \mingens{q_{i-1}\Cpart_{i-1}}.
  \end{align*}
  Finally,
  \begin{equation*}
    T_{k+1}\coloneqq \left\{f\in \mingens{I^s} \longmid \deg_{y}f < \deg_{y}h_k \right\}\subseteq \mingens{q_{k}\Cpart_k}.
  \end{equation*}
  Note that the sets $S_i$ and $T_i$ for $1\le i\le k+1$ form a
  partition of $\mingens{I^s}$.  The assertion now follows from
  Corollary~\ref{corollary:mingens-shift}.
\end{proof}

\begin{corollary}\label{corollary:back-shift}
  With the assumptions and notation of
  Corollary~\ref{corollary:mingens-shift}, let $\ell\in \{1,2\}$,
  $g\in \mingens{I^{s+\ell}}$ and $i$ such that
  $\deg_{y}g_{i+1}^{s+\ell} \le \deg_{y} g \le \deg_{y}g_i^{s+\ell}$.

  If $\ell = 1$, then
  \begin{equation*}
    g \in
    \begin{cases}
      g_{i}\mingens{I^s} & \text{ if } \deg_{y}g \ge \deg_{y}h_ig_i \\
      g_{i+1}\mingens{I^s} & \text{ if } \deg_{y}g \le \deg_{y}h_ig_i.
    \end{cases}
  \end{equation*}

  If $\ell = 2$, then
  \begin{equation*}
    g \in
    \begin{cases}
      g_{i}\mingens{I^{s+1}}   & \text{ if } \deg_{y}g \ge \deg_{y}h_ig_ig_{i+1}\\
      g_{i+1}\mingens{I^{s+1}} & \text{ if } \deg_{y}g \le \deg_{y}h_ig_ig_{i+1}.
    \end{cases}
  \end{equation*}
\end{corollary}
\begin{proof}
  In the proof of Corollary~\ref{corollary:mingens-shift}, we have
  seen that $h_ig_i$ is the link point between
  $g_{i+1}\Big(q_i\Cpart_i + \widehat q_i\Hpart_i\Big)$ and
  $g_iq_i\Cpart_i$.  Moreover, $\deg_{y}h_ig_ig_{i+1}$ is the link
  point between
  $g_{i+1}^2\Big(q_i\Cpart_i + \widehat q_i\Hpart_i\Big)$ and
  $g_ig_{i+1}\widehat q_i\Hpart_i + g_i^2q_i\Cpart_i$.  The assertion
  follows from a comparison of degrees.
\end{proof}

We now present examples to conclude this section.

\begin{example}\label{example:small}
  Let $I=(y^2,x^2y,x^3)$. We apply
  Corollary~\ref{corollary:I-with-2-persistent-stabilizes} with
  \begin{equation*}
    P=P(I)=\{y^2,x^3\}
  \end{equation*}
  to give a complete description of the generators of large powers of
  $I$.
  \begin{enumerate}
  \item We start by computing $D_P=1$,
    $r = r_{x}(P,D_P) = r_{y}(P,D_P) = 1$, and $s=3$.
  \item Next, we compute that $h^y_1 = x^6y^2$ is the minimal
    generator of $I^3$ with $y$-degree at least $r\cdot 2 + 0=
    2$. Recall $h^y_0 = y^6$ and $h^y_2 = x^9$.
  \item We compute the $(3,y)$-stable components with respect to $P$:
    \begin{align*}
      \Cpart_0 &= I^3:y^2 = (y^4,x^2y^3,x^3y^2,x^5y,x^6), \\
      \Hpart_1 &= I^3:x^{3}y^2 = (y^2,x^2y,x^3), \text{ and} \\
      \Cpart_1 &= I^3:x^6 = (y^2,x^2y,x^3).
    \end{align*}
    Recall that $\mingensast{J}$ denotes the minimal generating set of
    a monomial ideal $J$ excluding the minimal generator of maximal
    $y$-degree.
  \end{enumerate}
  From this we obtain that for all $\ell\ge0$,
  \begin{equation*}
    I^{3+\ell}=(y^4,x^2y^3,x^3y^2,x^5y,x^6)\odot(y^2,x^2y,x^3)^{\odot\ell}\odot(y^2,x^2y,x^3),
  \end{equation*}
  and hence
  \begin{equation*}
    \mingens{I^{3+\ell}}= y^{2 + 2\ell}\mingens{\Cpart_0}\uplus\biguplus_{j=1}^{\ell} x^{3+3j}y^{2 + 2(\ell-j)}\mingensast{\Hpart_1}\uplus x^{6+3\ell}\mingensast{\Cpart_1},
  \end{equation*}
  where $\mingensast{J}$ denotes the set of minimal generators of an
  ideal $J$, without the minimal generator of largest $y$-degree.  In
  particular, we have $\mu(I^{3+\ell})=7+ 2\ell$.
\end{example}

\begin{example}\label{example:big}
  Let
  $I=\big(y^{10},xy^9,x^2y^5,x^4y^4,x^5y^3,x^6y^2,x^{12}y,x^{15}\big)$.
  The computations for this example are done in
  SageMath\footnote{\label{anc}\emph{section\_6.ipynb} as ancillary
    file on the arXiv page \url{https://arxiv.org/abs/2503.21466} of
    this paper.}.  We compute the persistent generators of $I$
  \begin{equation*}
    P(I)=\{y^{10},x^2y^5,x^6y^2,x^{15}\},
  \end{equation*}
  and with $P=P(I)$ we obtain $D_P=40$, $r(P,D_P)=r_y(D,P)=200$, and
  $s=241$.

  With Theorem~\ref{theorem:I^D+l}, we compute
  \begin{equation*}
    I^{241} =  \left((y^{10},x^2y^5)^{201} + (x^2y^5,x^6y^2)^{201} + (x^6y^2,x^{15})^{201}\right)I^{40}.
  \end{equation*}
  Now, with the notation of
  Definition~\ref{definition:s-stable-components},
  \begin{align*}
    h^y_0 = y^{2410},\quad h^y_1 = x^{162}y^{2005},\quad h^y_2 = x^{753}y^{1002},\quad
    h^y_3 = x^{1815}y^{400},\quad h^y_4 = x^{3615},
  \end{align*}
  and hence the $(241,y)$-stable components of $I$ with respect to $P$
  are

  \begin{minipage}{0.4\textwidth}
    \begin{align*}
      \Cpart_0 &= I^{241}:y^{2005},\\
      \Cpart_1 &= I^{241}:x^{162}y^{1002},\\
      \Cpart_2 &= I^{241}:x^{753}y^{400},\\
      \Cpart_3 &= I^{241}:x^{1815}.
    \end{align*}
  \end{minipage}
  \begin{minipage}{0.4\textwidth}
    \begin{align*}
      \Hpart_1 &= I^{241}:x^{160}y^{2005},\\
      \Hpart_2 &= I^{241}:x^{749}y^{1002},\\
      \Hpart_3 &= I^{241}:x^{1806}y^{400},\\
               &\phantom{= I^{241}:x^{1815}\uplus.}
    \end{align*}
  \end{minipage}
  \vspace{1em}

  Analogously to the example above, these ideals can be used to write
  down the minimal generators of $I^{241+\ell}$ explicitly.  Further,
  it follows that for all $\ell\ge 0$,
  \begin{equation*}
    \mu(I^{241+\ell})=1688 + 7\ell.
  \end{equation*}
\end{example}

\section{Runtime in practice}\label{sec:runtime}
The results of Section~\ref{section:min-gens} provide a method for
faster computations of high powers of monomial ideals. We implemented
this method in SageMath~\cite{sagemath} (Version~9.5) as follows: Given a monomial
ideal $I$, we use the lower bounds for $s$ and $D$ as in
Theorem~\ref{theorem:I-with-k-persistent-stabilizes}.  In a
preprocessing step, we compute $I^s$ using
Theorem~\ref{theorem:I^D+l}.  For this, we compute $I^D$ by classical
exponentiation. In the implementation we take advantage of the fact
that all ideals involved are bivariate.  After that, for higher powers
$I^{s+\ell}$ for $\ell\ge 0$, we employ
Theorem~\ref{theorem:I-with-k-persistent-stabilizes}.

We compare the runtime of our implementation against the built-in
ideal exponentiation of Macaulay2~\cite{M2} (Version~1.21; ideals are of type
\texttt{MonomialIdeal}).  The results of our comparison are displayed
in Table~\ref{table:runtime}.  These figures suggest a substantial
speed advantage of our implementation.  Additionally, we included
runtimes for Macaulay2 when using Theorem~\ref{theorem:I^D+l} to
compute $I^{D+\ell}$ from $I^D$, which already demonstrates
significant runtime improvements.  All computations were done on a
machine equipped with an AMD EPYC 9474F 48-Core Processor @4.10GHz
(192 cores) and 1536GB RAM.

We tested the method on four
different ideals, which are provided in an additional
file\footnoteref{anc} and choose $P=P(I)$ in all four cases.
The ideal from Example~\ref{example:big} is $I_2$ in the table. In the
following, we write $D\coloneqq D_P$ and $s \coloneqq D + r(P,D)
+1$. We use $(s,y)$-stable components whenever $r(P,D) = r_y(P,D)$ and
$(s,x)$-stable components otherwise.

\begin{table}[h]
  \begin{tabular}{lrrr}
    \toprule
    \phantom{ideal}& $\mu(I)$  & $|P(I)|$ & $\maxset{}{I}$\\
    \midrule
    $I_1$          &  $5$     &  $3$  &  $7$  \\
    $I_2$          &  $8$     &  $4$  &  $15$ \\
    $I_3$	   &  $10$    &  $7$  &  $12$ \\
    $I_4$	   &  $15$    &  $4$  &  $24$ \\
    \bottomrule
  \end{tabular}
  \caption{An overview of the parameters $\mu(I)$, $|P(I)|$, and
    $\maxset{}{I}=\max\{\maxset{x}{I},\maxset{y}{I}\}$ of the four
    test ideals.}
  \label{table:examples}
\end{table}

\begin{table}[h]
        \scalebox{0.931}{
  \footnotesize
  \begin{tabular}{llbbrrrrr}
    \toprule
    && \multicolumn{2}{c}{preprocessing} &$s+10^2$ &$s+10^3$&$s+10^4$&$s+10^5$&$s+10^6$\\[1ex]
    && \multicolumn{1}{c}{$I^D$}& \multicolumn{1}{c}{$I^s$}&\multicolumn{5}{c}{ $I^{s+\ell}$}\\
    \midrule
    \multirow{4}{2em}{$I_1$}	    && $D=13$    &  $s=45$ & &&&&  \\[1ex]
    &SageMath (\ref{theorem:I^D+l},\,\ref{theorem:I-with-k-persistent-stabilizes})& $0.005$ & $0.01$  & $0.04$  & $0.35$    & $4.30$   & $51.35$ & $584.89$ \\
    &M2 (\ref{theorem:I^D+l})     & $0.0006$  &  $\ast$ & $0.05$  & $1.69$  & $1503.34$ & --       & --\\
    &M2	 (built-in)                  & $\ast$    & $\ast$  & $0.08$  & $22.25$ &$34898.4$  & --       & --\\
    \midrule
    \multirow{4}{2em}{$I_2$}        && $D=40$   &  $s=241$ &  &&&&  \\[1ex]
    &SageMath (\ref{theorem:I^D+l},\,\ref{theorem:I-with-k-persistent-stabilizes})& $0.12$   & $0.30$  & $0.13$  & $0.69$  & $7.44$   & $87.12$ & $980.50$ \\
    &M2 (\ref{theorem:I^D+l})     & $0.02$   & $\ast$  & $0.36$  & $8.53$  &$3152.86$ & --      & --\\
    &M2	(built-in)                   & $\ast$   & $\ast$  & $8.02$  & $411.03$&  --      & --      & --\\
    \midrule
    \multirow{4}{2em}{$I_3$}        && $D=76$  &  $s=989$ &  &&&&  \\[1ex]
    &SageMath (\ref{theorem:I^D+l},\,\ref{theorem:I-with-k-persistent-stabilizes})& $0.73$  & $6.29$  & $0.52$  & $1.49$  & $12.12$  & $139.68$& $1551.92$ \\
    &M2 (\ref{theorem:I^D+l})     & $0.18$  & $\ast$  & $15.39$ & $69.82$ & $18724.5$& --      & --\\
    &M2	(built-in)		     & $\ast$  & $\ast$  &$607.32$ &$5050.71$& --       & --      & --\\
    \midrule
    \multirow{4}{2em}{$I_4$}        && $D=238$   &  $s=2064$ &  &&&&  \\[1ex]
    &SageMath (\ref{theorem:I^D+l},\,\ref{theorem:I-with-k-persistent-stabilizes})& $28.45$   & $47.13$ & $2.28$   & $4.15$  & $20.53$  & $209.50$ & $2305.38$ \\
    &M2 (\ref{theorem:I^D+l})     & $84.81$   & $\ast$  & $71.77$  & $176.13$& $13546.8$& --       & --\\
    &M2	(built-in)		      & $\ast$    & $\ast$  & $>12$h   & --     & --        & --       & --\\
    \bottomrule
  \end{tabular}}
  \caption{\footnotesize The two columns under ``preprocessing'' show
    the times required to compute $I^D$ and $I^s$, where ``$\ast$''
    indicates that the corresponding method does not use that
    preprocessing step. The remaining columns present the additional
    times needed to compute $I^{s+10^i}$ after preprocessing. Cells
    containing ``--'' indicate that the estimated computation time
    would be prohibitively large and is therefore omitted.}
                \label{table:runtime}
        \end{table}

\bibliographystyle{plain}
\bibliography{bibliography}

\end{document}